\documentclass[reqno]{amsart}
\usepackage{amssymb}
\usepackage{amsmath}
\usepackage{amsthm}
\usepackage[usenames]{color}
\usepackage{graphicx}
\usepackage{cite}
\usepackage{bbm}

\DeclareMathOperator{\dv}{div\,}

\usepackage[colorlinks,linkcolor=black,anchorcolor=black,citecolor=black]{hyperref}
\usepackage[margin=1in]{geometry} 
\usepackage{marginnote}

\usepackage{color}

\newtheorem{thm}{Theorem}[section]

\newtheorem{lem}[thm]{Lemma}
\newtheorem{prop}[thm]{Proposition}
\newtheorem{rem}{Remark}[section]

\numberwithin{equation}{section}

\newcommand{\eps}{\epsilon}

\newcommand{\vertiii}[1]{{\left\vert\kern-0.25ex\left\vert\kern-0.25ex\left\vert #1
		\right\vert\kern-0.25ex\right\vert\kern-0.25ex\right\vert}}

\makeatletter
\@namedef{subjclassname@2020}{%
	\textup{2020} Mathematics Subject Classification}
\makeatother

\begin{document}

\title[The compressible Navier--Stokes equations] 
{High Mach number limit of the compressible Navier--Stokes equations in critical Besov spaces}

\author[J. Ni] {Jinkai Ni}
\address[JKN]{School of Mathematics, Nanjing University, Nanjing
 210093, P. R. China}
\email{jinkaini123@gmail.com}

\author[Z. Zhang]{Zhipeng Zhang} 
\address[ZPZ]{School of Mathematical Sciences, Ocean University of China, Qingdao   266100, P. R. China}  
\email{zhangzp@ouc.edu.cn}

\begin{abstract}
We investigate the high Mach number limit for the scaled compressible Navier--Stokes system in the critical  Besov framework. In the scaled momentum equation, the pressure force is represented by the term \(\varepsilon^2\nabla a^\varepsilon\), where $\varepsilon$ is the inverse Mach number;
as \(\varepsilon\to0\), the formal limiting system is the compressible pressureless Navier--Stokes system. The analysis is complicated by the absence of density dissipation in the limiting model and by the highest-order coupling created by the viscous terms. For \(d\geq2\), we prove the global well-posedness of the scaled system for small initial data and obtain estimates that are uniform with respect to $\varepsilon$. A crucial ingredient is a parameter-dependent lower-order
estimate for \(\varepsilon a^\varepsilon\), which compensates for the purely transport nature of the density equation and allows the uniform bounds to be closed. Based on these estimates, we justify the high Mach number limit and recover a global strong solution to the pressureless Navier--Stokes system.
For \(d\geq3\), we further derive quantitative error estimates between the scaled solutions and the pressureless limiting solution. More precisely, on each fixed finite time interval, if the initial discrepancy is of order 
\(\mathcal{O}(\varepsilon)\),
then the corresponding lower-order critical Besov error satisfies the same   rate, which yields a quantitative justification of the pressureless limit.
\end{abstract}

\date{\today}
	

\maketitle
\thispagestyle{empty}

\section{Introduction}
\subsection{Models and related works}
In this paper, we focus on the study of the global well-posedness and convergence of the solutions to the isentropic compressible Navier--Stokes equations in the limit as the Mach number tends to infinity.
The system under consideration reads 
\begin{equation}\label{1.1a}
\left\{
\begin{aligned}
&\partial_t\rho^\varepsilon+\dv (\rho^\varepsilon u^\varepsilon)=0,\\
&\partial_t (\rho^\varepsilon u^\varepsilon)
+\dv(\rho^\varepsilon u^\varepsilon
\otimes u^\varepsilon)
+\varepsilon^2\nabla P(\rho^\varepsilon)-\mu \Delta u^\varepsilon-(\mu+\lambda) \nabla\dv u^\varepsilon=0
\end{aligned}
\right.
\end{equation}
with the initial data
\begin{align}
{(\rho^\varepsilon,u^\varepsilon )|}_{t=0}=(\rho^\varepsilon_0, u^\varepsilon_0),\nonumber
\end{align}
where $\rho^\varepsilon(t,x)\in\mathbb{R}_+$ denotes the density, $u^\varepsilon(t,x)\in\mathbb{R}^d$ the velocity, and $P(\rho^\varepsilon)=K(\rho^\varepsilon)^\gamma$ ($K > 0$, $\gamma > 1$) the pressure.   
The viscosity coefficients $\mu$ and $\lambda$ are constants satisfying the physical constraints:
\begin{align*}
\mu>0,\qquad 2\mu + d\lambda\geq 0,   
\end{align*}
where $d$ is the spatial dimension. The parameter $\varepsilon$ is the inverse of the Mach number defined as the ratio of the speed of flows to that of sound. We restrict our attention to the case where the fluid domain is the whole space $\mathbb{R}^d$.

Over the past few decades, a substantial number of papers have focused on the well-posedness and low Mach number limit  of the system \eqref{1.1a}. This research area traces back to the works of Nash \cite{N} and Serrin \cite{S}. Assuming that the density is strictly positive, they respectively established the local existence and uniqueness of classical solutions.
The first result regarding global classical solutions was achieved by Matsumura--Nishida \cite{MN-JMKU-1980} for initial data close to a non-vacuum equilibrium in $H^3$. Subsequently, Hoff \cite{Hoff-1,Hoff-2} demonstrated the global existence of weak solutions for discontinuous initial data with small energy.
In the presence of vacuum, owing to the degeneracy of the momentum equation, the problem becomes significantly more challenging. 
The global existence of weak solutions with large initial data in $\mathbb{R}^d$ was first obtained by Lions \cite{Lions} for $\gamma \geq \frac{3d}{d+2}$ with $d=2,3$. 
Later, Feireisl, Novotn\'y and Petzeltov\'a in \cite{FNP} extended the result of Lions to the case $\gamma > \frac{3}{2}$ for $d = 3$. Jiang and Zhang in \cite{JZ-1,JZ-2} proved the global existence of weak solutions with vacuum for any $\gamma>1$ for spherically symmetric or axisymmetric initial data. We specifically note that, except for a special case in \cite{Hoff-3}, the uniqueness and regularity of Lions--Feireisl's weak solutions remain open problems.

A natural approach to address the problem of uniqueness is to identify a functional setting that is as extensive as possible, within which both existence and uniqueness are guaranteed.
Such an approach was initially proposed by Fujita and Kato’s work \cite{FK-1964-ARMA}, where the key concept is that optimal function spaces for the well-posedness of incompressible Navier--Stokes equations exhibit scaling invariance. 
Therefore, in recent years, scaling invariance has played  a crucial role in the theory of global well-posedness. 
Danchin \cite{Danchin-IM-00} demonstrated the global well-posedness of \eqref{1.1a} in the $L^2$-type critical Besov space. Later, this result was extended to  the case of the general $L^p$-type critical Besov spaces by
Charve and Danchin \cite{CD-2010-ARMA}, Chen et al. \cite{CMZ-2010-CPAM} and Haspot \cite{Haspot-2011-ARMA}. 
For the related topics on strong solutions with the critical regularity, Danchin \cite{Danchin-02} showed  that the small solution  converges to the incompressible flow as Mach number tends to zero in the critical $L^2$-type Besov space.  
Danchin and He \cite{DH} proved the same results in the   critical $L^p$-type $(2 \leq p <4)$ framework  for the high-frequency component. Recently, Fujii \cite{F-24} proved that a   unique global solution exists for arbitrarily large initial perturbation in the critical $L^2$-type Besov space,  provided that the Mach number is sufficiently small, and further showed that the solution converges to the corresponding incompressible flow in some scaling critical norms. Interested readers may refer to \cite{Danchin-CPDE-2007,DX-ARMA-2017,Xj-CMP-2019,X-X-2021-JDE,Haspot-2011-ARMA} for the large-time behavior of global solutions.

Formally, letting $\varepsilon\rightarrow0$ in \eqref{1.1a} and assuming that  the limits $\rho^\varepsilon\rightarrow \rho$ and $u^\varepsilon\rightarrow u$ exist, the pressure term in the momentum equation \eqref{1.1a}$_2$ vanishes and we derive the following pressureless Navier--Stokes system
\begin{equation}\label{1.5}
\left\{
\begin{aligned}
&\partial_t\rho+\dv (\rho u)=0,\\
&\partial_t (\rho u)
+\dv(\rho u
\otimes u)
-\mu \Delta u-(\mu+\lambda) \nabla\dv u=0
\end{aligned}
\right.
\end{equation}
with the initial data
\begin{align}\label{1.5.1}
{(\rho, u)|}_{t=0}=(\rho _0,u_0).
\end{align}
The above system arises in the description of astrophysical phenomena. Recently, significant progress has been made on the pressureless Navier--Stokes system \eqref{1.5}. Danchin \cite{Dr-2024-PMP} studied two-dimensional viscous pressureless flows with large density variations. Guo et al. \cite{GTZ-JDE-2025} proved the global well-posedness and stability of classical solutions under smallness assumptions in Sobolev spaces. Specifically, they derived the optimal decay rate of the velocity $u$, i.e., $\|u(t)\|_{L^2}\lesssim (1+t)^{-\frac{3}{4}}$, while the density remains uniformly bounded rather than decaying. Wang et al. \cite{WWX-2025-arXiv} considered Fujita-Kato type solutions in $\mathbb{R}^3$ with only bounded density. More recently, Li and the first and second authors of this paper established global well-posedness, uniform stability, and optimal time-decay rates in the critical Besov space $\dot B^{\frac{d}{2}}_{2,1}\times \dot B^{\frac{d}{2}-1}_{2,1}$ in  \cite{LNZ-2026-arxiv}. These results clarify the role of the density in pressureless dynamics: the density is propagated uniformly in time but does not satisfy decay estimates comparable to those of the velocity, which stands in sharp contrast with the classical isentropic compressible Navier--Stokes system \cite{Danchin-IM-00}. However,  
to the best of our knowledge, there is currently no rigorous justification for the limiting process as $\varepsilon\rightarrow 0$, i.e., high Mach number limit.

The purpose of this paper is to provide a rigorous justification for the high Mach number limit.
We focus on the case $\gamma=2$. More precisely, we take
$P(\rho)=K\rho^2$ in \eqref{1.1a}, and
after normalizing the harmless constant \(2K\) in the pressure term, obtain the following system:
\begin{equation}\label{1.1}
\left\{
\begin{aligned}
&\partial_t\rho^\varepsilon+\dv (\rho^\varepsilon u^\varepsilon)=0,\\
&\partial_t (\rho^\varepsilon u^\varepsilon)
+\dv(\rho^\varepsilon u^\varepsilon
\otimes u^\varepsilon)
+\varepsilon^2\rho^\varepsilon\nabla\rho^\varepsilon-\mu \Delta u^\varepsilon-(\mu+\lambda) \nabla\dv u^\varepsilon=0.
\end{aligned}
\right.
\end{equation}
Furthermore, we reformulate the original problem \eqref{1.1} in terms of the following perturbed variables:
\begin{align}
a^\varepsilon=\rho^\varepsilon-1, \quad u^\varepsilon=u^\varepsilon-0.\nonumber
\end{align}
Then, $(a^\varepsilon,u^\varepsilon)$ satisfies the following system
\begin{equation}\label{1.3}
\left\{
\begin{aligned}
&\partial_ta^\varepsilon+
\dv u^\varepsilon=-\dv (a^\varepsilon u^\varepsilon),\\
&\partial_t u^\varepsilon-\mu \Delta u^\varepsilon-(\mu +\lambda )\nabla\dv u^\varepsilon
+\varepsilon^2\nabla a^\varepsilon\\
&\quad=-u^\varepsilon\cdot\nabla u^\varepsilon
-\mu f(a^\varepsilon)\Delta u^\varepsilon
-(\mu+\lambda) f(a^\varepsilon)\nabla\dv u^\varepsilon,
\end{aligned}
\right.
\end{equation}
where $f(a^\varepsilon)=\frac{a^\varepsilon}{1+a^\varepsilon}$, with the initial condition
\begin{align}\label{1.4}
(a^\varepsilon,u^\varepsilon)|_{t=0}=(\rho^\varepsilon_0-1,u^\varepsilon_0).
\end{align}

Since we need to obtain uniform-in-$\varepsilon$ global estimates for $(a^\varepsilon,u^\varepsilon)$, the density dissipation structure induced by the pressure term is no longer valid. Our strategy relies on analyzing the solutions separately in different frequency regimes. 
To this end, we perform the following asymptotic analysis for the eigenvalues associated with the linearized system of \eqref{1.3}. 
We apply the operators 
$$\mathbb{Q}=-\nabla(-\Delta)^{-1}\dv \quad\text{and}\quad \mathbb{P}={\rm Id}+\nabla(-\Delta)^{-1}\dv$$ to \eqref{1.3}, and then obtain 
\begin{equation}\label{1.6.1}
\left\{
\begin{aligned}
&\partial_ta^\varepsilon+
\dv\mathbb{Q}u^\varepsilon=-\dv (a^\varepsilon u^\varepsilon),\\
&\partial_t \mathbb{Q}u^\varepsilon-(2\mu+\lambda) \Delta\mathbb{Q}u^\varepsilon
+\varepsilon^2\nabla a^\varepsilon
=\mathbb{Q}\big(-u^\varepsilon\cdot\nabla u^\varepsilon
-\mu f(a^\varepsilon)\Delta u^\varepsilon
-(\mu+\lambda) f(a^\varepsilon)\nabla\dv u^\varepsilon\big),\\
&(a^\varepsilon,\mathbb{Q}u^\varepsilon)|_{t=0}=(\rho^\varepsilon_0-1,\mathbb{Q}u^\varepsilon_0),
\end{aligned}
\right.
\end{equation}
and 
\begin{equation}\label{1.6.2}
\left\{
\begin{aligned}
&\partial_t \mathbb{P} u^\varepsilon-\mu \Delta\mathbb{P} u^\varepsilon=\mathbb{P}\big(-u^\varepsilon\cdot\nabla u^\varepsilon
-\mu f(a^\varepsilon)\Delta u^\varepsilon
-(\mu+\lambda) f(a^\varepsilon)\nabla\dv u^\varepsilon\big),\\
&\mathbb{P}u^\varepsilon|_{t=0}=\mathbb{P}u^\varepsilon_0.
\end{aligned}
\right.
\end{equation}
For the compressible part $\mathbb{Q}u^\varepsilon$, we consider the linearized system of \eqref{1.6.1} as follows
\begin{equation}\label{1.7}
\left\{
\begin{aligned}
&\partial_ta^\varepsilon+
\dv\mathbb{Q} u^\varepsilon=0,\\
&\partial_t \mathbb{Q}u^\varepsilon-(2\mu+\lambda) \Delta\mathbb{Q}u^\varepsilon
+\varepsilon^2\nabla a^\varepsilon=0,\\
&(a^\varepsilon,\mathbb{Q}u^\varepsilon)|_{t=0}=(\rho^\varepsilon_0-1,\mathbb{Q}u^\varepsilon_0).
\end{aligned}
\right.
\end{equation}
Let $G_1(t,x)$ be the Green function of \eqref{1.7}. Taking the Fourier transform of \eqref{1.7} on $x$, we derive 
\begin{equation}\label{1.8}
\left\{
\begin{aligned}
&\partial_t\widehat{G_1}(t,\xi)+B_1\widehat{G_1}(t,\xi)=0,\\
&\widehat{G_1}(0,\xi)=I,
\end{aligned}
\right.
\end{equation}
where 
$$B_1=
\begin{pmatrix}
0\quad& i\xi^\top \\
i\varepsilon^2\xi\quad & (2\mu+\lambda)|\xi|^2
\end{pmatrix}.
$$
After straightforward calculations, the determinant of the characteristic matrix of $-B_1$ is given by
\begin{align}
\text{det}(\lambda^{\prime} I+B_1)=\big(\lambda^\prime+(2\mu+\lambda)|\xi|^2\big)^{d-1}\big({\lambda^\prime}^2+(2\mu+\lambda)|\xi|^2\lambda^\prime+\varepsilon^2|\xi|^2\big)=0,
\end{align}
which implies that the eigenvalues of $-B_1$ are
\begin{align}
&\lambda_0=-(2\mu+\lambda)|\xi|^2\,\, (\text{multiplicity (d-1)}),\nonumber\\
&\lambda_1=-\frac{(2\mu+\lambda)|\xi|^2}{2}+
\sqrt{\frac{(2\mu+\lambda)^2|\xi|^4}{4}-\varepsilon^2|\xi|^2},\nonumber\\
&\lambda_2=-\frac{(2\mu+\lambda)|\xi|^2}{2}-
\sqrt{\frac{(2\mu+\lambda)^2|\xi|^4}{4}-\varepsilon^2|\xi|^2}.\nonumber
\end{align}
For the incompressible $\mathbb{P}u^\varepsilon$, considering the linearized system of \eqref{1.6.2}, we have 
\begin{equation}\label{1.11.1}
\left\{
\begin{aligned}
&\partial_t \mathbb{P} u^\varepsilon-\mu \Delta\mathbb{P} u^\varepsilon=0,\\
&\mathbb{P} u^\varepsilon|_{t=0}=\mathbb{P}u^\varepsilon_0.
\end{aligned}
\right.
\end{equation}
The corresponding eigenvalues are readily obtained as
$$\lambda_3=-\mu|\xi|^2\,\, (\text{multiplicity d}).$$
Then, simple calculations yield the asymptotic behavior of the eigenvalues
$\lambda_j(j=1,2)$ as follows.
For the frequency regime $|\xi|<<\varepsilon$, the eigenvalues satisfy 
\begin{align}
&\lambda_1=-\frac{(2\mu+\lambda)|\xi|^2}{2}+
i\varepsilon|\xi|+\varepsilon^{-1}\mathcal{O}(|\xi|^3),\nonumber\\
&\lambda_2=-\frac{(2\mu+\lambda)|\xi|^2}{2}-
i\varepsilon|\xi|+\varepsilon^{-1}\mathcal{O}(|\xi|^3),\nonumber
\end{align}
where $i=\sqrt{-1}$.
For the frequency regime $|\xi|>>\varepsilon$, we have 
\begin{align}
&\lambda_1=-\frac{\varepsilon^2}{(2\mu+\lambda)}+\varepsilon^{4}\mathcal{O}(|\xi|^{-2}),\nonumber\\
&\lambda_2=-(2\mu+\lambda)|\xi|^2
+\varepsilon^{2}\mathcal{O}(1).\nonumber
\end{align}

Therefore, the above spectral analysis suggests using the frequency scale \(\varepsilon\) as the first threshold to divide the frequency space into two parts: \(|\xi|>\varepsilon\) and \(|\xi|\leq \varepsilon\).
In addition, we attempt to estimate $(a^\varepsilon,u^\varepsilon)$ in the frequency regime $|\xi|\geq\varepsilon$ using the effective velocity. However, we find that the terms appearing in \eqref{4.51}, namely
$$
\varepsilon^2\| \omega^\varepsilon\|^{m+h}_{L^1_t(\dot B^{\frac{d}{2}-1}_{2,1})}\quad \text{and} \quad\varepsilon^4\|\nabla(-\Delta)^{-1}a^\varepsilon\|^{m+h}_{L^1_t(\dot B^{\frac{d}{2}-1}_{2,1})},
$$
cannot be absorbed by the dissipation terms
$$
\|\omega^\varepsilon\|^{m+h}_{L^1_t(\dot B^{\frac{d}{2}+1}_{2,1})}
\quad\text{and}\quad\varepsilon^2\|a^\varepsilon\|^{m+h}_{L^1_t(\dot B^{\frac{d}{2}}_{2,1})},
$$
respectively, where $\omega^\varepsilon$ is the effective velocity, and the definition of $\|\cdot\|^{m+h}_{\dot B^s_{p,r}}$ will be given at the end of this subsection.
To overcome this difficulty, we further decompose the frequency regime $|\xi|>\varepsilon$ into two parts:  $\varepsilon\leq|\xi|<\varepsilon^{1-\delta_c}$ and $|\xi|\geq\varepsilon^{1-\delta_c}$, where $0<\delta_c<1$ is a  constant. 
In summary, we divide the frequency space into the low-frequency regime $|\xi|\leq \varepsilon$, medium-frequency $\varepsilon<|\xi|\leq \varepsilon^{1-\delta_c}$, and high-frequency regime $|\xi|>\varepsilon^{1-\delta_c}$, and estimate $(a^\varepsilon,u^\varepsilon)$ in each regime separately.
Given that our analysis will depend on the dyadic Littlewood--Paley decomposition to split the frequencies, we also introduce the thresholds
$$J_0^\varepsilon:=[\log_2\varepsilon]\quad\text{and}\quad J_1^\varepsilon:=[(1-\delta_c)\log_2\varepsilon].$$

Following the above frequency decomposition, 
we divide distributions into low-,
medium-, and high-frequency parts. 
For any \(g\in\mathcal S'_h\), we write
\begin{align*}
    g=\sum_{k\in\mathbb Z}\dot\Delta_k g
\end{align*}
in \(\mathcal S'_h\),  and define
\begin{align*}
g^\ell
:=
\sum_{k\leq J^\varepsilon_0}\dot\Delta_k g,\qquad
g^m
:=
\sum_{J^\varepsilon_0< k\le J^\varepsilon_1}\dot\Delta_k g,\qquad
g^h
:=
\sum_{k>J^\varepsilon_1}\dot\Delta_k g,
\end{align*}
where $\dot\Delta_k$ is the classical Littlewood--Paley frequency-localization operator, see \cite{BCD-Book-2011}.
We also introduce the shorthand notation
\begin{align*}
 g^{\ell+m}:=g^\ell+g^m,
\qquad
g^{m+h}:=g^m+g^h, 
\end{align*}
and the corresponding semi-norms, for any \(s\in\mathbb R\), \(1\le p,r\le\infty\), and
\(\varkappa\in[1,\infty]\),
\begin{equation*}
\left\{
\begin{aligned}
\|g\|^\ell_{\dot B^s_{p,r}}
&:=
\Big\|
\big\{2^{ks}\|\dot\Delta_k g\|_{L^p}\big\}_{k\leq J^\varepsilon_0}
\Big\|_{\ell^r},
\\
\|g\|^m_{\dot B^s_{p,r}}
&:=
\Big\|
\big\{2^{ks}\|\dot\Delta_k g\|_{L^p}\big\}_{J^\varepsilon_0< k\le J^\varepsilon_1}
\Big\|_{\ell^r},
\\
\|g\|^h_{\dot B^s_{p,r}}
&:=
\Big\|
\big\{2^{ks}\|\dot\Delta_k g\|_{L^p}\big\}_{k>J^\varepsilon_1}
\Big\|_{\ell^r},
\end{aligned}
\right.
\end{equation*}
and, similarly,
\begin{equation*}
\left\{
\begin{aligned}
\|g\|^\ell_{\widetilde L^\varkappa_T(\dot B^s_{p,r})}
&:=
\Big\|
\big\{2^{ks}\|\dot\Delta_k g\|_{L^\varkappa_T(L^p)}\big\}_{k\leq J^\varepsilon_0}
\Big\|_{\ell^r},
\\
\|g\|^m_{\widetilde L^\varkappa_T(\dot B^s_{p,r})}
&:=
\Big\|
\big\{2^{ks}\|\dot\Delta_k g\|_{L^\varkappa_T(L^p)}\big\}_{J^\varepsilon_0< k\le J^\varepsilon_1}
\Big\|_{\ell^r},
\\
\|g\|^h_{\widetilde L^\varkappa_T(\dot B^s_{p,r})}
&:=
\Big\|
\big\{2^{ks}\|\dot\Delta_k g\|_{L^\varkappa_T(L^p)}\big\}_{k>J^\varepsilon_1}
\Big\|_{\ell^r}.
\end{aligned}
\right.
\end{equation*}
When \(r=\infty\), the above \(\ell^r\)-norms are understood in the usual sense.

For the low- and high-frequency parts, it follows from the standard Bernstein inequality that the following elementary estimates hold, i.e., for any \(s'\!>0\),
\begin{equation}\label{1.13.1}
\left\{
\begin{aligned}
&\|g^\ell\|_{\dot B^s_{p,r}}
\lesssim
\|g\|^\ell_{\dot B^s_{p,r}}
\lesssim  \varepsilon^{s^{\prime}}
\|g\|^\ell_{\dot B^{s-s'}_{p,r}},
\\
& {\|g^h\|_{\dot B^s_{p,r}}
\lesssim
\|g\|^h_{\dot B^s_{p,r}}
\lesssim
\varepsilon^{s'(\delta_c-1)}\|g\|^h_{\dot B^{s+s'}_{p,r}}},
\\
&\|g^\ell\|_{\widetilde L^\varkappa_T(\dot B^s_{p,r})}
\lesssim
\|g\|^\ell_{\widetilde L^\varkappa_T(\dot B^s_{p,r})}
\lesssim
\varepsilon^{s^{\prime}}\|g\|^\ell_{\widetilde L^\varkappa_T(\dot B^{s-s'}_{p,r})},
\\
& {\|g^h\|_{\widetilde L^\varkappa_T(\dot B^s_{p,r})}
\lesssim
\|g\|^h_{\widetilde L^\varkappa_T(\dot B^s_{p,r})}
\lesssim \varepsilon^{s'(\delta_c-1)}
\|g\|^h_{\widetilde L^\varkappa_T(\dot B^{s+s'}_{p,r})}},
\end{aligned}
\right.
\end{equation}
here and in \eqref{1.13.2}, the implicit constants do not depend on $J^\varepsilon_0$ and $J^\varepsilon_1$. For the medium-frequency part, 
based on the definitions of $\|\cdot\|^m_{\dot B^s_{p,r}}$ and $\|\cdot\|^m_{\widetilde L^\varkappa_T(\dot B^s_{p,r})}$, we directly obtain that, for any
\(s_1,s_2\in\mathbb R\) satisfying $s_1\leq s_2$, the following Bernstein-type inequalities hold, 
\begin{equation}\label{1.13.2}
\left\{
\begin{aligned}
&\|g\|^m_{\dot B^{s_1}_{p,r}}
\lesssim
\varepsilon^{s_1-s_2} \|g\|^m_{\dot B^{s_2}_{p,r}},\\
&\|g\|^m_{\dot B^{s_2}_{p,r}}
\lesssim
\varepsilon^{(s_2-s_1)(1-\delta_c)}\|g\|^m_{\dot B^{s_1}_{p,r}},\\
&\|g\|^m_{\widetilde L^\varkappa_T(\dot B^{s_1}_{p,r})}
\lesssim
\varepsilon^{s_1-s_2} \|g\|^m_{\widetilde L^\varkappa_T(\dot B^{s_2}_{p,r})},\\
&\|g\|^m_{\widetilde L^\varkappa_T(\dot B^{s_2}_{p,r})}
\lesssim
\varepsilon^{(s_2-s_1)(1-\delta_c)} \|g\|^m_{\widetilde L^\varkappa_T(\dot B^{s_1}_{p,r})}.
\end{aligned}
\right.
\end{equation}
Obviously, for fixed $\varepsilon$, we directly derive from \eqref{1.13.2} the following equivalence relations, i.e., 
for any
\(s_1,s_2\in\mathbb R\),
\begin{align*}
\|g\|^m_{\dot B^{s_1}_{p,r}}
\lesssim
\|g\|^m_{\dot B^{s_2}_{p,r}},
\qquad
\|g\|^m_{\widetilde L^\varkappa_T(\dot B^{s_1}_{p,r})}
\lesssim
\|g\|^m_{\widetilde L^\varkappa_T(\dot B^{s_2}_{p,r})}.
\end{align*}
These estimates in \eqref{1.13.1} and \eqref{1.13.2} will be used repeatedly.

Finally, we shall use the following hybrid Besov norms adapted to the above
low-medium-high decomposition. If \(s<s'\), we define
\begin{align*}
\|g\|_{\dot B^s_{p,r}\cap\dot B^{s'}_{p,r}}
:=
\|g\|^{\ell+m}_{\dot B^s_{p,r}}
+\|g\|^h_{\dot B^{s'}_{p,r}},
\end{align*}
where
\begin{align*}
\|g\|^{\ell+m}_{\dot B^s_{p,r}}
:=
\|g\|^\ell_{\dot B^s_{p,r}}
+\|g\|^m_{\dot B^s_{p,r}}.    
\end{align*}
If \(s>s'\), we define
\begin{align*}
\|g\|_{\dot B^s_{p,r}+\dot B^{s'}_{p,r}}
:=
\|g\|^\ell_{\dot B^s_{p,r}}
+\|g\|^{m+h}_{\dot B^{s'}_{p,r}},
\end{align*}
where
\begin{align*}
 \|g\|^{m+h}_{\dot B^{s'}_{p,r}}
:=
\|g\|^m_{\dot B^{s'}_{p,r}}
+\|g\|^h_{\dot B^{s'}_{p,r}}.   
\end{align*}
The same convention will be used for the corresponding Chemin-Lerner norms.


\subsection{Main results}
We first focus on the local well-posedness of strong solutions.
\begin{thm}[Local well-posedness]\label{Th0}
Let \(d\ge2\) and \(0<\varepsilon<1\). Assume that
\begin{align*}
\varepsilon a_0^\varepsilon\in \dot B^{\frac d2-1}_{2,1},
\qquad
 u_0^\varepsilon\in \dot B^{\frac d2-1}_{2,1},  \qquad
a_0^\varepsilon\in \dot B^{\frac d2}_{2,1},
\end{align*}
and that there exists a small constant \(\kappa_0>0\) such that
\begin{align}\label{NZG3.2}
 \|a_0^\varepsilon\|_{\dot B^{\frac{d}{2}}_{2,1}}\leq \kappa_0.
\end{align}
Then, there exists \(T_\varepsilon>0\) such that the Cauchy problem
\eqref{1.3}--\eqref{1.4} admits a unique strong solution
\((a^\varepsilon, u^\varepsilon)\) on \([0,T_\varepsilon]\), satisfying
\begin{align*}
 1+a^\varepsilon(t,x)>0,
\qquad (t,x)\in[0,T_\varepsilon]\times\mathbb R^d,   
\end{align*}
and
\begin{equation*}
\left\{
\begin{aligned}
&a^\varepsilon
\in
\mathcal C\big([0,T_\varepsilon];\dot B^{\frac d2}_{2,1}\big),\\
& u^\varepsilon
\in
\mathcal C\big([0,T_\varepsilon];\dot B^{\frac d2-1}_{2,1}\big)
\cap
L^1\big(0,T_\varepsilon;\dot B^{\frac d2+1}_{2,1}\big),\\
&\varepsilon a^\varepsilon
\in
\mathcal C\big([0,T_\varepsilon];\dot B^{\frac d2-1}_{2,1}\big)
\cap
L^1\big(0,T_\varepsilon;\dot B^{\frac d2+1}_{2,1}
+\dot B^{\frac d2}_{2,1}\big).
\end{aligned}
\right.
\end{equation*}
\end{thm}

\begin{rem} 
The smallness condition on \(a_0^\varepsilon\) imposed in \eqref{NZG3.2}
is only used to carry out the Eulerian iterative argument, in particular to
absorb the terms involving $f(a^\varepsilon)\mathcal{A}{u}^\varepsilon$ 
in the Lam\'{e} estimate. In fact, this smallness assumption can be replaced by the weaker non-vacuum condition that \(1+a_0^\varepsilon\) is bounded away
from zero, provided one uses the Lagrangian approach; see
\cite[Section 3.2 and Theorem 3]{Danchin-2018} and also the related critical space theory in \cite{Danchin-IM-00}. Since the Eulerian approach of \cite[Section 3.1]{Danchin-2018} is sufficient for the present argument, we work under the smallness assumption \eqref{NZG3.2}.
\end{rem}

\begin{rem}
The proof of Theorem \ref{Th0} follows the Eulerian iterative argument  in \cite[Section 3.1]{Danchin-2018}. Within the current \(L^2\)-based framework, the low-order Cauchy contraction closes directly when \(d\geq3\). In the endpoint case where \(d = 2\), the same contraction argument is no longer applicable in the \(\ell^1\)-Besov setting. Consequently, we employ the standard endpoint procedure: existence is derived from the compactness argument based on Friedrichs regularization, while uniqueness is obtained from the logarithmic interpolation inequality and Osgood's lemma; refer to \cite[Section~3.1]{Danchin-2018} and \cite{Danchin-IM-00}.
\end{rem}

\begin{rem}
Compared with \cite[Section 3.1]{Danchin-2018}, an additional point   in our scaled system is the parameter-dependent lower-order bound for $\varepsilon a^\varepsilon$ in $L^\infty_T(\dot B^{\frac d2-1}_{2,1})
\cap
L^1_T(\dot B^{\frac d2+1}_{2,1}+\dot B^{\frac d2}_{2,1})$, 
which is obtained only after closing the main estimates for
\(a^\varepsilon_n\) and \( u^\varepsilon_n\).
\end{rem}

Based on Theorem \ref{Th0}, we establish the following uniform global well-posedness result.
\begin{thm}[Uniform-in-$\varepsilon$ global estimate]\label{Th1}
Let \(d\ge2\) and \(0<\varepsilon<1\). Assume that the initial data
\((a_0^\varepsilon, u_0^\varepsilon)\) satisfy
\begin{align*}
 \varepsilon (a_0^{\varepsilon})^\ell\in \dot B^{\frac d2-1}_{2,1},
\qquad u_0^\varepsilon\in \dot B^{\frac d2-1}_{2,1},
\qquad
(a_0^{\varepsilon})^{m+h}\in \dot B^{\frac d2}_{2,1}.   
\end{align*} 
There exists a constant \(\delta_0>0\), independent of \(\varepsilon\), such that if
\begin{align}\label{NJK1.13}
\mathcal{X}_{\varepsilon,0}:=&\,
\varepsilon\|a^\varepsilon_0\|^\ell_{\dot B^{\frac d2-1}_{2,1}}
+\|u^\varepsilon_0\|_{\dot B^{\frac d2-1}_{2,1}}+\|a^\varepsilon_0\|^{m+h}_{\dot B^{\frac d2}_{2,1}}
\le \delta_0,
\end{align}
then the Cauchy problem \eqref{1.3}--\eqref{1.4} admits a unique global strong solution
\((a^\varepsilon,u^\varepsilon)\) satisfying
\begin{align*}
    \rho^\varepsilon=1+a^\varepsilon>0.
\end{align*}
Moreover,
for all \(t\ge0\),
\begin{align}\label{NJK1.14}
\mathcal X_\varepsilon(t)
\leq C\mathcal X_{\varepsilon,0},
\end{align}
where \(C>0\) is independent of both \(t\) and \(\varepsilon\), and
\begin{align}\label{NJK1.15}
\mathcal{X}_\varepsilon(t):=&\,
\varepsilon\|a^\varepsilon\|^\ell_{\widetilde L_t^\infty(\dot B^{\frac d2-1}_{2,1})}
+\|u^\varepsilon\|^\ell_{\widetilde L_t^\infty(\dot B^{\frac d2-1}_{2,1})}+\varepsilon\|a^\varepsilon\|^\ell_{ L_t^1(\dot B^{\frac d2+1}_{2,1})}+\|u^\varepsilon\|^\ell_{ L_t^1(\dot B^{\frac d2+1}_{2,1})}\nonumber\\
&\,+\|a^\varepsilon\|^{m+h}_{\widetilde L_t^\infty(\dot B^{\frac d2}_{2,1})}
+\|u^\varepsilon\|^{m+h}_{\widetilde L_t^\infty(\dot B^{\frac d2-1}_{2,1})}+\varepsilon^2\|a^\varepsilon\|^{m+h}_{L_t^1(\dot B^{\frac d2}_{2,1})}
+\|u^\varepsilon\|^{m+h}_{ L_t^1(\dot B^{\frac d2+1}_{2,1})}.
\end{align}
\end{thm}

\begin{rem}
The argument does not seem to yield  an estimate for  $((a^\varepsilon)^h,(u^\varepsilon)^h)$ in the critical $L^p$-framework by using the effective velocity. The obstruction comes from the estimates of the medium-frequency part for some nonlinear terms, such as  
$\|f(a^\varepsilon)\Delta u^\varepsilon\|^m_{L^1(\dot B^{\frac{d}{2}-1}_{2,1})}$. Applying Bony’s paraproduct decomposition, we rewrite it as
\begin{align}\label{1.17.2}
f(a^\varepsilon)\Delta u^\varepsilon = T_{\Delta u^\varepsilon}f(a^\varepsilon) + R(f(a^\varepsilon),\Delta u^\varepsilon) + T_{f(a^\varepsilon)}\Delta u^\varepsilon.
\end{align}
Through detailed calculation, we observe that the term $\|T_{\Delta u^\varepsilon}f(a^\varepsilon)+R(f(a^\varepsilon),\Delta u^\varepsilon)+T_{ f(a^\varepsilon)}\Delta (u^\varepsilon)^{l+m}\|^m_{L^1(\dot B^{\frac{d}{2}-1}_{2,1})}$ can be controlled by 
$$\|u^\varepsilon\|_{L^1(\dot B^{\frac{d}{p}+1}_{p,1})}
\Big(\varepsilon\|a^\varepsilon\|^\ell_{\widetilde L^\infty(\dot B^{\frac{d}{2}-1}_{2,1})}
+\|a^\varepsilon\|^m_{\widetilde L^\infty(\dot B^{\frac{d}{2}}_{2,1})}+\|a^\varepsilon\|^h_{\widetilde L^\infty(\dot B^{\frac{d}{p}}_{p,1})}\Big),$$
for $p\in[2,\min(4,2d/(d-2))]$ with $p\neq 4$ if $d=2$, which is the desired bound.
However, for the remaining term $\|T_{ f(a^\varepsilon)}\Delta (u^\varepsilon)^{h}\|^m_{L^1(\dot B^{\frac{d}{2}-1}_{2,1})}$, by arguing analogously to  the approach in \cite[Chapter 4]{Danchin-2018}, the smallest upper bound we can obtain is 
$$\varepsilon^{1-\delta}\|a^\varepsilon\|_{\widetilde L^\infty(\dot B^{\frac{d}{p}-1}_{p,1})}\|u^\varepsilon\|_{L^1(\dot B^{\frac{d}{p}+1}_{p,1})}.$$
Clearly, this bound is not the desired one. 
For the term $\|a^\varepsilon\dv u^\varepsilon\|^m_{L^1(\dot B^{\frac{d}{2}}_{2,1})}$, 
a similar problem also exists.
Consequently, it becomes difficult to obtain estimates for $((a^\varepsilon)^h,(u^\varepsilon)^h)$ in the critical $L^p$-framework.
\end{rem}

With the uniform-in-$\varepsilon$ global estimate in hand,
we justify the limiting process as $\varepsilon\rightarrow 0$. 
The global well-posedness of the limiting system \eqref{1.5}--\eqref{1.5.1} then follows from that of the 
Navier--Stokes equations in the critical Besov framework.
\begin{thm}[High Mach number limit]\label{Th2}
Let \(d\geq2\). Assume that
\begin{align*}
 a_0\in  {\dot B^{\frac d2-1}_{2,1}}\cap\dot B^{\frac d2}_{2,1},
\qquad
 u_0\in \dot B^{\frac d2-1}_{2,1}.   
\end{align*}
There exists a sufficiently small constant \(\delta_1>0\) such that if
\begin{align}\label{HM2.1}
\mathcal X_0:=  {\|a_0\|_{\dot B^{\frac d2-1}_{2,1}}}
+\|a_0\|_{\dot B^{\frac d2}_{2,1}}
+\| u_0\|_{\dot B^{\frac d2-1}_{2,1}}
\leq \delta_1,
\end{align}
then the following statements hold:
\begin{itemize}
\item There exists a family of approximate initial data $\{(a_0^\varepsilon, u_0^\varepsilon)\}_{0<\varepsilon<\varepsilon_0}$ such that $a_0^\varepsilon\to a_0$ in $ {\dot B^{\frac{d}{2}-1}_{2,1}}\cap\dot B^{\frac{d}{2}}_{2,1}$ and $ u_0^\varepsilon\to  u_0$ in $\dot B^{\frac{d}{2}-1}_{2,1}$ as $\varepsilon\to0$.
Moreover, for all \(0<\varepsilon<\varepsilon_0\), the system \eqref{1.3} with initial data $(a^\varepsilon, u^\varepsilon)|_{t=0}=
(a_0^\varepsilon, u_0^\varepsilon)$
admits a unique global strong solution
\((a^\varepsilon, u^\varepsilon)\), and the family satisfies the uniform
estimate \eqref{NJK1.14}.
 
\item As \(\varepsilon\to0\), there exists a pair of functions \((a, u)\) such that, up to
a subsequence,  
\begin{equation}\label{NZGG1.18}
\left\{
\begin{aligned}
&a^\varepsilon\rightharpoonup  a\,
\qquad\text{weakly-* in}\qquad
L_{\rm loc}^\infty(\mathbb R^+;\dot B^{\frac d2}_{2,1}),\\
& u^\varepsilon\rightharpoonup u
\qquad\text{weakly-* in}\qquad
L_{\rm loc}^\infty( \mathbb R^+;\dot B^{\frac d2-1}_{2,1}),\\
& u^\varepsilon\rightharpoonup  u
\qquad\hbox{weakly in}\qquad\,\,\,\,\,
L_{\rm loc}^1(\mathbb R+;\dot B^{\frac d2+1}_{2,1}),\\
&a^\varepsilon\to a\,
\qquad\text{strongly in}\qquad\,\,
\mathcal C_{\rm loc}(\mathbb R^+;L^2_{\rm loc} ),\\
& u^\varepsilon\to u
\qquad\text{strongly in}\qquad\,
L_{\rm loc}^2(\mathbb R^+;L^2_{\rm loc} ).
\end{aligned}
\right.
\end{equation}
\item The limit \((a, u)\) is the unique global strong solution to the perturbation system around $(1,0)$ of the system \eqref{1.5} with initial data $
(a, u)|_{t=0}=(a_0, u_0)$. 
Furthermore, $\rho=1+a>0$, and
\begin{align*}
a\in
\mathcal C(\mathbb R^+;\dot B^{\frac d2}_{2,1}),
\qquad
{u}\in
\mathcal C(\mathbb R^+;\dot B^{\frac d2-1}_{2,1})
\cap
L^1(\mathbb R^+;\dot B^{\frac d2+1}_{2,1}),   
\end{align*}
and, for all \(t\geq0\),
\begin{align}\label{HM2.2}
\|a\|_{\widetilde L_t^\infty(\dot B^{\frac d2}_{2,1})}
+
\|{u}\|_{\widetilde L_t^\infty(\dot B^{\frac d2-1}_{2,1})}
+
\|{u}\|_{L_t^1(\dot B^{\frac d2+1}_{2,1})}
\leq
C\mathcal{X}_0,
\end{align}
\end{itemize}
where \(C>0\) is independent of time and initial data.
\end{thm}

The convergence achieved in Theorem \ref{Th2} is based on compactness and is thus qualitative. To further quantify the high Mach number limit from \eqref{1.3} to the perturbation system of   \eqref{1.5}, we establish error estimates between the scaled solution \((a^\varepsilon, u^\varepsilon)\) and the limiting solution \((a, u)\). More precisely, on any fixed time interval \([0,T]\), we demonstrate that an \(\mathcal{O}(\varepsilon)\) initial discrepancy results in an \(\mathcal{O}(\varepsilon)\) convergence rate in a lower-order critical Besov norm. The limitation to finite time intervals is due to the pressure remainder \(\varepsilon^2\nabla a^\varepsilon\), which cannot be globally controlled in time under the current uniform estimates for the density.

\begin{thm}[Finite-time high Mach number error estimate]\label{Th3}
Let \(d\geq3\), $(a^\varepsilon, u^\varepsilon)$
be the global strong solution to the scaled compressible Navier--Stokes system \eqref{1.3} obtained in Theorem \ref{Th1}, with initial data $(a^\varepsilon_0,u^\varepsilon_0)$, 
and $(1+a,{u})$ 
be the corresponding global strong solution to the limiting pressureless Navier--Stokes system \eqref{1.5} obtained in Theorem \ref{Th2},
with initial data $(a_0, u_0)$.
Assume that the smallness assumptions \eqref{NJK1.13} and \eqref{HM2.1} hold, and 
\begin{align}\label{HM3.1}
\|a^\varepsilon_0-a_0\|_{\dot B^{\frac d2-1}_{2,1}}
+
\|{u}^\varepsilon_0-{u}_0\|_{\dot B^{\frac d2-2}_{2,1}}
\leq \varepsilon .
\end{align}
Then, for any fixed \(T>0\), there exists a constant \(C_T>0\), independent of
\(\varepsilon\), such that
\begin{align}\label{HM3.2}
&\|a^\varepsilon-a\|_{\widetilde L_t^\infty
(\dot B^{\frac d2-1}_{2,1})}
+
\|{u}^\varepsilon-{u}\|_{\widetilde L_t^\infty
(\dot B^{\frac d2-2}_{2,1})}
+
\|{u}^\varepsilon-{u}\|_{L_t^1
(\dot B^{\frac d2}_{2,1})}\nonumber\\
&\quad+\|\partial_{t}({u}^\varepsilon-{u})\|_{  L_t^1
(\dot B^{\frac d2-2}_{2,1})}\leq
C_{T}\varepsilon,
\end{align}
for any $t\in[0,T]$.
\end{thm}

\begin{rem}
Theorem \ref{Th3} is formulated for finite-time intervals primarily due to the pressure remainder $\varepsilon^2\nabla a^\varepsilon$ in the velocity error equation. In the lower-order error estimate, this term is controlled as follows:
\begin{align*} 
\varepsilon^2\|\nabla a^\varepsilon\|_{L_T^1(\dot B^{\frac d2 - 2}_{2,1})}\lesssim\varepsilon T\|\varepsilon a^\varepsilon\|_{\widetilde L_T^\infty(\dot B^{\frac d2 - 1}_{2,1})}\lesssim C_T\varepsilon.   
\end{align*}
Consequently, the constant depends on \(T\). A global-in-time estimate would necessitate a uniform bound of the form
\begin{align*}
\|\varepsilon a^\varepsilon\|_{L_t^1(\dot B^{\frac d2 - 1}_{2,1})}\leq C,\qquad t\geq0,    
\end{align*}
for a uniform constant $C > 0$. However, such a bound is not provided by Theorem \ref{Th1} because the density equation lacks intrinsic dissipation. We also confine the estimate to \(d\geq3\). After reducing one derivative, the case \(d = 2\) becomes the endpoint $\dot B^0_{2,1}\times\big(\dot B^{-1}_{2,1}\cap L_T^1\dot B^1_{2,1}\big)$, where the standard \(\ell^1\)-Besov stability estimate fails to yield a Lipschitz error bound. Therefore, the current   \(\mathcal{O}(\varepsilon)\) estimate is proved only for \(d\geq3\).
\end{rem}

\subsection{Strategies of the proofs}
We first discuss the proof of the local well-posedness result in Theorem \ref{Th0}. For each fixed \(\varepsilon>0\), the scaled compressible Navier--Stokes system \eqref{1.3}--\eqref{1.4} is regarded as a quasilinear transport-parabolic system. The local construction is based on Danchin's Eulerian iteration scheme \cite[Section 3]{Danchin-2018}. Although the pressure term \(\varepsilon^2\nabla a^\varepsilon\) has no significant impact for a fixed \(\varepsilon>0\), the local construction needs to keep track of the  lower-order quantity \(\varepsilon a^\varepsilon\), which explains why the local norm is not only established on
\begin{align*}
a^\varepsilon\in \widetilde L^\infty_T(\dot B^{\frac d2}_{2,1}),\qquad {u}^\varepsilon\in \widetilde L^\infty_T(\dot B^{\frac d2 - 1}_{2,1})\cap L^1_T(\dot B^{\frac d2+1}_{2,1}),
\end{align*}
but also includes the additional lower-order component
\begin{align*}
\varepsilon a^\varepsilon\in\widetilde L^\infty_T(\dot B^{\frac d2 - 1}_{2,1})\cap L^1_T(\dot B^{\frac d2+1}_{2,1}+\dot B^{\frac d2}_{2,1}).
\end{align*}
Based on the approximate system \eqref{G3.2}, we first propagate \(a_n^\varepsilon\) in the critical space 
\(\widetilde L^\infty_{T_{\varepsilon}}(\dot B^{\frac d2}_{2,1})\) by means of the transport estimate, which also ensures the uniform lower bound of \(1+a_n^\varepsilon\), such that  \(f(a_n^\varepsilon)\) is well-defined.
Then we estimate \(u_n^\varepsilon\) by decomposing it into the free Lam\'{e} evolution and a remainder term. 
The free part is made small by choosing the local time \(T_\varepsilon\) sufficiently small, while the remainder is controlled through the maximal regularity estimate.
After the velocity estimates are obtained, we return to the density equation at one lower regularity level and derive the bound for \(\varepsilon a_n^\varepsilon\). 

Compared with the classical local theory in \cite[Section 3]{Danchin-2018}, the main new point is the presence of the scaled pressure force \(\varepsilon^2\nabla a^\varepsilon\). The argument is not only the standard transport-Lam\'{e} iteration: one has to combine the critical estimate for \(a_n^\varepsilon\), the parabolic estimate for \(u_n^\varepsilon\), and the lower-order weighted estimate for \(\varepsilon a_n^\varepsilon\). 
Once these uniform bounds of $(a^\varepsilon_{n},u_{n}^\varepsilon)$ are obtained, the convergence of the approximate sequence follows from lower-order stability estimates for \(d\ge3\), while the endpoint case \(d=2\) is treated by the usual compactness argument. 
Uniqueness is then proved by the corresponding lower-order stability estimate, and time continuity in the critical spaces follows from a standard low-medium-high frequency decomposition.

For the proof of Theorem \ref{Th1}, spectral analysis reveals that the behavior of the solution differs in the two frequency regimes: $|\xi|>\varepsilon$ and $|\xi|<\varepsilon$, respectively. In the low-frequency regime $|\xi|\leq \varepsilon$, we perform a basic $L^2$ estimate for $(a^\varepsilon,u^\varepsilon)$. Since we obtain an estimate for $\varepsilon a^\varepsilon$ in $\widetilde L^\infty_T(\dot B^{\frac{d}{2}-1}_{2,1})$, rather than for $a^\varepsilon$ in $\widetilde L^\infty_T(\dot B^{\frac{d}{2}}_{2,1})$, we need to be careful when estimating the nonlinear terms involving the density. In particular, we  repeatedly use  the following inequality: 
$$
\|a^\varepsilon\|_{\dot B^{\frac{d}{2}}_{2,1}}^\ell\lesssim \varepsilon\|a^\varepsilon\|_{\dot B^{\frac{d}{2}-1}_{2,1}}^\ell,
$$
which is specific to the low-frequency regime. 
For the frequency regime $|\xi|>\varepsilon$, we attempt to estimate $(a^\varepsilon,u^\varepsilon)$ by using the effective velocity. Unfortunately, we find that the threshold $2^{J_0^\varepsilon}$ tends to $ 0$  (and hence becomes infinitesimally small) as $\varepsilon$ approaches zero. As a result, it is impossible to select a sufficiently large threshold such that the following two terms are absorbed: 
$$
\varepsilon^2\| \omega^\varepsilon\|^{m+h}_{L^1_t(\dot B^{\frac{d}{2}-1}_{2,1})}\quad \text{and} \quad\varepsilon^4\|\nabla(-\Delta)^{-1}a^\varepsilon\|^{m+h}_{L^1_t(\dot B^{\frac{d}{2}-1}_{2,1})}.
$$
To overcome this obstacle, we divide the frequency regime $|\xi|> \varepsilon$ into two parts: $\varepsilon<|\xi|\leq \varepsilon^{1-\delta_c}$ (medium-frequency) and $|\xi|>\varepsilon^{1-\delta_c}$ (high-frequency). In Lemma \ref{L4.3}, we obtain the estimates of $(a^\varepsilon,u^\varepsilon)$ in the medium-frequency part by taking the basic $L^2$ estimate. 
However, we 
only derive the dissipative estimate $\varepsilon^2\|u^\varepsilon\|^m_{ L^1_t(\dot B^{\frac{d}{2}+1})}$.
Since the dissipative structure of $u^\varepsilon$ is preserved, the velocity field should admit the estimate $\|u^\varepsilon\|^m_{ L^1_t(\dot B^{\frac{d}{2}+1})}$ (without the $\varepsilon^2$ factor). 
We utilize the optimal regularity estimate in Lemma \ref{Lemma2.6} to achieve this objective.
For the high-frequency part, we evaluate  $(a^\varepsilon,u^\varepsilon)$ by introducing the effective velocity. Since our goal is    to obtain uniform estimates with respect to $\varepsilon$, we construct the following effective velocity:
$$\omega^\varepsilon=\nabla(-\Delta)^{-1}(\frac{\varepsilon^2a^\varepsilon}{2\mu+\lambda}-\dv u^\varepsilon),$$ which is different from that in \cite{Danchin-2018}.
In particular, the terms $$
\varepsilon^2\| \omega^\varepsilon\|^{h}_{L^1_t(\dot B^{\frac{d}{2}-1}_{2,1})}\quad \text{and} \quad\varepsilon^4\|\nabla(-\Delta)^{-1}a^\varepsilon\|^{h}_{L^1_t(\dot B^{\frac{d}{2}-1}_{2,1})}
$$
can be bounded by 
$\varepsilon^{2\delta_{c}}\|\omega^\varepsilon\|^h_{L^1_t(\dot B^{\frac{d}{2}+1}_{2,1})}$ and $\varepsilon^{2+2\delta_{c}}\|a^\varepsilon\|^h_{L^1_t(\dot B^{\frac{d}{2}}_{2,1})}$,
respectively. Here, $\varepsilon^{2\delta_{c}}$ provides smallness, thus enabling the above terms to be absorbed.

Finally, we present the main ideas behind the proofs of Theorems \ref{Th2} and \ref{Th3}. The proof of Theorem \ref{Th2} is based on compactness arguments and the uniform estimates established in Theorem \ref{Th1}. Starting from the initial data of the limiting pressureless system \eqref{1.5}–\eqref{1.5.1},
we first choose a family of approximating data \((a_0^\varepsilon, u_0^\varepsilon)\) converging to \((a_0, u_0)\) in the critical Besov space and satisfying the smallness condition required by Theorem \ref{Th1}. The corresponding global solutions \((a^\varepsilon, u^\varepsilon)\) then enjoy uniform bounds independent of both \(t\) and \(\varepsilon\), which yield weak compactness in the critical spaces.
To obtain sufficient compactness for the nonlinear terms, we further   control $\partial_t a^\varepsilon$ and $\partial_t u^\varepsilon$ in lower-order Besov spaces. The Aubin--Lions lemma and a diagonal argument then yield strong local convergence of $a^\varepsilon$ and $u^\varepsilon$. This strong convergence allows us to pass to the limit in the transport and viscous nonlinearities, while the pressure term vanishes due to the weighted estimate on $\varepsilon a^\varepsilon$. Consequently, the limit satisfies the pressureless Navier--Stokes system \eqref{1.5}--\eqref{1.5.1}.

Theorem \ref{Th3} gives  the convergence rate of the above limiting process on any fixed time interval. Instead of using compactness, we subtract the scaled system and the pressureless system and estimate the error:
\begin{align*}
 \delta a:=a^\varepsilon-a,\qquad 
\delta u:= u^\varepsilon- u .  
\end{align*}
in a lower-order critical Besov framework. 
The main new forcing term is the pressure defect \(\varepsilon^2\nabla a^\varepsilon\), which is estimated by the parameter-dependent bound for \(\varepsilon a^\varepsilon\) and contributes only \(\mathcal{O}(\varepsilon)\) on finite time intervals. 
All remaining nonlinear difference terms are linear in \((\delta a,\delta {u})\) and are controlled by the uniform smallness of the scaled and limiting solutions. 
This yields an inequality of the form (see \eqref{G5.4}):
\begin{align*}
\delta\mathcal X(t)
\le
C_T\Big(
\|a_0^\varepsilon-a_0\|_{\dot B^{\frac d2-1}_{2,1}}
+
\|u_0^\varepsilon-u_0\|_{\dot B^{\frac d2-2}_{2,1}}
+
\varepsilon
+(\delta_0+\delta_1)\delta\mathcal X(t) \Big).    
\end{align*}
Hence, if the initial discrepancy is of the order \(\mathcal{O}(\varepsilon)\), then the lower-order critical Besov error remains of the order \(\mathcal{O}(\varepsilon)\) on every finite time interval, which proves the  finite-time error estimate and completes the proof of Theorem \ref{Th3}.

\subsection{Outline of this paper}
The rest of this paper is organized as follows. 
In Section 2, we recall some notations used throughout this paper, 
as well as the product and commutator estimates. 
In Section 3, we establish the local well-posedness of the scaled compressible Navier–Stokes system \eqref{1.3}–\eqref{1.4}. 
Section 4 is devoted to proving uniform global \emph{a priori} estimates in an $L^2$-type Besov framework.
More precisely, by introducing two $\varepsilon$-dependent frequency thresholds of order $\mathcal{O}(\varepsilon)$ and order $\mathcal{O}(\varepsilon^{1-\delta_c})$, we decompose the solution into low-, medium-, and high-frequency parts. 
The low- and medium-frequency parts are treated by the refined energy methods, while the high-frequency part is handled by an effective velocity-variable method.
Finally, we justify the high Mach number limit and prove the convergence rate of solutions to the pressureless Navier--Stokes system \eqref{1.5}--\eqref{1.5.1} as $\varepsilon \to 0$.

\section{Preliminaries}
\subsection{Notations}

Throughout this paper, \(C>0\) denotes a generic harmless constant, which may change from line to line. We use the notation \(u\lesssim v\) to mean \(u\le Cv\), and \(u\approx v\) to mean \(u\lesssim v\) and \(v\lesssim u\). For a Banach space \(X\) and \(u,v\in X\), we set $\|(u,v)\|_X:=\|u\|_X+\|v\|_X$. 
For \(T>0\) and \(\varkappa\in[1,\infty]\), we write $
L_T^\varkappa(X):=L^\varkappa([0,T];X)$ 
for the space of measurable functions \(u:[0,T]\to X\) such that \(t\mapsto \|u(t)\|_X\) belongs to \(L^\varkappa(0,T)\). Finally, \(\mathcal C([0,T];X)\) stands for the space of continuous \(X\)-valued functions on \([0,T]\).

\subsection{Analysis tools in Besov spaces}
We first introduce the following embedding and interpolation inequalities.
\begin{lem}\label{LA.2}{\rm(\!\!\cite[Chapter 2]{BCD-Book-2011})}
The following properties hold:
\begin{itemize}
\item{} For $s\in\mathbb{R}$, $1\leq p_{1}\leq p_{2}\leq \infty$ and $1\leq r_{1}\leq r_{2}\leq \infty$, it holds
\begin{equation}\notag
\begin{aligned}
\dot{B}^{s}_{p_{1},r_{1}}\hookrightarrow \dot{B}^{s-d (\frac{1}{p_1} -\frac{1}{p_2})}_{p_{2},r_{2}}.
\end{aligned}
\end{equation}
\item{} For $1\leq p\leq q\leq\infty$, we have the following chain of continuous embedding:
\begin{equation}\nonumber
\begin{aligned}
  \dot{B}^{0}_{p,1}\hookrightarrow L^{p}\hookrightarrow \dot{B}^{0}_{p,\infty}\hookrightarrow \dot B_{q,\infty}^{\varrho},\quad \varrho=-d\Big(\frac{1}{p}-\frac{1}{q}\Big).
\end{aligned}
\end{equation}
\item{} If $p<\infty$, then $\dot{B}^{\frac{d}{p}}_{p,1}$ is continuously embedded in the set of continuous functions decaying to 0 at infinity.

\item{}  The following real interpolation property is satisfied for $1\leq p\leq \infty$, $s_1<s_2$, and $\theta\in (0,1)$:
\begin{align}\label{A.1}
\|g\|_{\dot B_{p,1}^{\theta s_1+(1-\theta)s_2}}\lesssim \frac{1}{\theta(1-\theta)(s_2-s_1)}\|g\|_{\dot B_{p,\infty}^{s_1}}^{\theta} \|g\|_{\dot B_{p,1}^{s_2}}^{1-\theta}  .  
\end{align}

\end{itemize}
\end{lem}

To control the nonlinear terms, we require the following Moser-type product estimates in Besov spaces:
\begin{lem} \label{LA.3}{\rm(\!\!\cite[Chapter 2]{BCD-Book-2011})}
The following statements hold:
\begin{itemize}
\item{} Let $s>0$, $1\leq p,r\leq \infty$.  Then $\dot B^{s}_{p,r}\cap L^\infty$ is an algebra and
 \begin{align}\label{A.2}
\|gh\|_{\dot B_{p,r}^s}\lesssim \|g\|_{L^\infty}\|h\|_{\dot B_{p,r}^s}+\|h\|_{L^\infty}\|g\|_{\dot B_{p,r}^s}.    
 \end{align}

\item{} Let $s_1,s_2>0$ and $p$ satisfy $2\leq p\leq \infty$, $s_1\leq \frac{d}{p}$, $s_2\leq \frac{d}{p}$, and $s_1+s_2>0$.  Then it holds 
 \begin{align}\label{A.3}
\|gh\|_{\dot B_{p,1}^{s_1+s_2-\frac{d}{p}}}\lesssim \|g\|_{ \dot B_{p,1}^{s_1 }}\|h\|_{\dot B_{p,1}^{s_2}} .    
 \end{align}

\item{} Let $s_1,s_2>0$ and $p$ satisfy $2\leq p\leq \infty$, $s_1\leq \frac{d}{p}$, $s_2<\frac{d}{p}$, and $s_1+s_2\geq0$.  Then it holds 
 \begin{align}\label{A.4}
\|gh\|_{\dot B_{p,\infty}^{s_1+s_2-\frac{d}{p}}}\lesssim \|g\|_{ \dot B_{p,1}^{s_1 }}\|h\|_{\dot B_{p,\infty}^{s_2}} .    
 \end{align}
    \end{itemize}
\end{lem}
We now present the following lemma concerning the continuity of composite functions:

\begin{lem}\label{LA.4}{\rm(\!\!\cite[Chapter 2]{BCD-Book-2011})}
Let $F:\mathbb{R}\rightarrow \mathbb{R}$ be a smooth function such that $F(0)=0$. Then, for any $1\leq p,r\leq \infty$, $s>0$, and $g\in \dot{B}_{p,r}^s \cap L^\infty$, it holds that $F(g) \in \dot{B}_{p,r}^s \cap L^\infty$, and
\begin{align}\label{A.5}
\|F(g)\|_{\dot{B}^{s}_{p,r}} \leq C_g \|g\|_{\dot{B}_{p,r}^s},
\end{align}
where the constant $C_g > 0$ depends only on $\|g\|_{L^\infty}$, $F'$, $s$, and $d$.
\end{lem}

Next, we consider the following linear transport equation:
\begin{align} \label{G2.6}
\left\{\begin{aligned}
&\partial_t g+ {u} \cdot\nabla g=h, \quad\, x\in \mathbb{R}^d,\quad t>0,  \\
& g(x,0)=g_0(x),\quad \,\,\,\,\,  x\in \mathbb{R}^d.
 \end{aligned}
 \right.
\end{align}

\begin{lem}\label{Lemma2.5}{\rm(\!\!\cite[Chapter 3]{BCD-Book-2011})}
Let $ T > 0 $, $ -\frac{d}{2} < s \leq \frac{d}{2} + 1 $,
 $ 1 \leq r \leq \infty $, $ g_0 \in \dot{B}_{2,r}^s $, $ u \in L^1(0,T; \dot{B}_{2,1}^{\frac{d}{2}+1}) $, 
 and $ h\in L^1(0,T; \dot{B}_{2,r}^s) $. Then there exists a constant $ C > 0 $, 
 independent of $ T $ and $  {g_0} $, such that the solution $g$ to \eqref{G2.6} satisfies
\begin{align}\label{G2.7}
\|g\|_{\widetilde L_T^\infty(\dot B_{2,r}^s)} 
\leq \exp\bigg\{C \|\nabla {u}\|_{L_T^1(\dot B_{2,1}^{\frac{d}{2}})} \bigg\} 
\bigg( \|g_0\|_{\dot B_{2,r}^s} + \int_0^T \|h\|_{\dot B_{2,r}^s}  {\rm d}\tau\bigg).
\end{align}
Moreover, if $ r < \infty $, then the solution $ g \in \mathcal{C}([0,T]; \dot B_{2,r}^s) $.
\end{lem}

We also recall the optimal regularity estimate for the following Lam\'{e} system:
\begin{align} \label{G2.8}
\left\{\begin{aligned}
&\partial_t {u}-\mu\Delta  {u}-(\mu+\lambda)\nabla{\rm div}\,{u}=g, \quad\, x\in \mathbb{R}^d,\quad t>0,  \\
& {u}(x,0)={u}_0(x),\quad \,\,\,\quad\quad\quad\quad\quad\quad\quad  x\in \mathbb{R}^d.
 \end{aligned}
 \right.
\end{align}

\begin{lem}\label{Lemma2.6}{\rm(\!\!\cite[Chapter 3]{BCD-Book-2011})}
Let $ T > 0 $, $ \mu > 0 $, $ 2\mu + d \lambda > 0 $, $ s \in \mathbb{R} $, $ 1 \leq p, r \leq \infty $, and $ 1 \leq \varsigma_2 \leq \varsigma_1 \leq \infty $. Suppose that $ {u}_0 \in \dot{B}^{s}_{p,r} $ and $ g \in \widetilde{L}^{\varsigma_2}(0,T;\dot{B}_{p,r}^{s-2+\frac{2}{\varsigma_2}}) $. Then there exists a solution ${u} $ to \eqref{G2.8} satisfying
\begin{align*}
\min\{\mu,2\mu+\lambda\}^{\frac{1}{\varsigma_1}}\|{u}\|_{\widetilde{L}^{\varsigma_1}_T(\dot B^{s+\frac{2}{\varsigma_1}}_{p,r})}\lesssim \|{u}_0\|_{\dot B_{p,r}^s}+\min\{\mu,2\mu+\lambda\}^{\frac{1}{\varsigma_2}-1}\|g\|_{L^{\varsigma_2}_T(\dot B_{p,r}^{s-2+\frac{2}{\varsigma_2}})}.
\end{align*}
\end{lem}

Finally, we end this section with 
the following commutator estimates, which are used to control the nonlinearities in high frequencies.
\begin{lem}\label{LA.5} {\rm(\!\!\cite[Chapter 2]{BCD-Book-2011})}
Let $1\leq p\leq \infty$ and $-\frac{d}{p}<s\leq 1+\frac{d}{p}$. Then it holds that
\begin{align}\label{A.6}
\sum_{k\in\mathbb Z} 2^{ks}\|[g\cdot\nabla,\dot\Delta_k]h\|_{L^p}\lesssim \|g\|_{\dot B^{\frac{d}{p}+1}_{p,1}}\|h\|_{\dot B_{p,1}^s},    
\end{align}
where the commutator $[A,B]:=AB-BA$.
\end{lem}

\section{Local existence and uniqueness}
In this section, we prove the local well-posedness of the scaled compressible Navier--Stokes system \eqref{1.3}--\eqref{1.4}. We follow the classical Eulerian iteration scheme of Danchin in critical Besov spaces (cf. \cite[Section 3]{Danchin-2018}). The main idea is to construct
an approximate sequence by solving, successively, a linear transport equation and  the Lam\'{e} equation.
Throughout this section, we set
\begin{align*}
\mathcal A u^\varepsilon:=\mu\Delta u^\varepsilon+(\mu+\lambda)\nabla\dv u^\varepsilon.
\end{align*}
Then,   \eqref{1.3} can be written as
\begin{equation}\label{localG3.1}
\left\{
\begin{aligned}
&\partial_t a^\varepsilon+ u^\varepsilon\cdot\nabla a^\varepsilon
=-(1+a^\varepsilon){\rm div} u^\varepsilon,\\
&\partial_t  u^\varepsilon-\mathcal A u^\varepsilon+\varepsilon^2\nabla a^\varepsilon
= - u^\varepsilon\cdot\nabla u^{\varepsilon}
-f(a^\varepsilon) \mathcal A  u^\varepsilon.
\end{aligned}
\right.
\end{equation}

\begin{proof}[Proof of Theorem \ref{Th0}] 
Since the proof of Theorem \ref{Th0} is relatively complicated,
below we split it into six steps.

\medskip
\noindent
\textbf{Step 1. Approximate solutions.}
Let $\dot S_n$ be the usual homogeneous low-frequency cut-off operator. Define
\begin{align*}
a^\varepsilon_{0,n}:=\dot S_n a^\varepsilon_0,
\qquad
u^\varepsilon_{0,n}:=\dot S_n u^\varepsilon_0.    
\end{align*} 
By virtue of Lemma \ref{LA.2},
we deduce that
\begin{align*}
\|a^\varepsilon_0\|_{L^\infty}
\leq C\|a^\varepsilon_0\|_{\dot B^{\frac d2}_{2,1}}
\leq C\kappa_0,
\end{align*}
and
\begin{align*}
a^\varepsilon_{0,n}\to a^\varepsilon_0
\qquad\hbox{in }\qquad L^\infty(\mathbb R^d).
\end{align*}
Since $\kappa_0$ is small enough, we may assume that
\begin{align*}
  \|a^\varepsilon_0\|_{L^\infty}\leq C\kappa_{0} \leq {\frac{1}{8}},  
\end{align*}
which gives rise to
\begin{align*}
\|a^\varepsilon_{0,n}\|_{L^\infty}\leq \|a_{0,n}^{\varepsilon}\|_{\dot B_{2,1}^{\frac{d}{2}}}  \leq \frac14,\qquad \text{and}\qquad1+a_{0,n}^\varepsilon\geq \frac{3}{4},    
\end{align*}
for all  sufficiently large  $n$. If necessary, discard finitely many terms. Then, we assume that the above lower bound holds for all {\(n\geq1\)}. 

We define the first term of the sequence by
\begin{align*}
 {a^\varepsilon_1(t,x):=a^\varepsilon_{0,1}(x),
\qquad
 u^\varepsilon_1(t,x):=e^{t\mathcal A} u^\varepsilon_{0,1}(x), }
\end{align*}
where $\{e^{t\mathcal A}\}_{t\geq 0}$ represents the semigroup of operators associated with  \eqref{G2.8}.
Suppose that \((a^\varepsilon_n, u^\varepsilon_n)\) has been constructed. According to \eqref{localG3.1}, we further consider the following linear transport and Lam\'{e} equations corresponding to $(a^\varepsilon_{n+1},u^\varepsilon_{n+1})$:
\begin{equation} \label{G3.2}
\left\{
\begin{aligned}
&\partial_t a^\varepsilon_{n+1}
+ u^\varepsilon_n\cdot\nabla a^\varepsilon_{n+1}
=-(1+a^\varepsilon_n){\rm div}\,  u^\varepsilon_n,\\
&\partial_t u^\varepsilon_{n+1}-\mathcal A  u^\varepsilon_{n+1}+\varepsilon^2\nabla a^\varepsilon_n
=
- u^\varepsilon_n\cdot\nabla u^\varepsilon_n
- f(a^\varepsilon_n)\mathcal{A} u_{n}^\varepsilon ,\\
&a^\varepsilon_{n+1}|_{t=0}=a^\varepsilon_{0,n+1},\qquad u^\varepsilon_{n+1}|_{t=0}= u^\varepsilon_{0,n+1}.
\end{aligned}
\right.
\end{equation}
For each fixed \(n\), the functions \(a^\varepsilon_n\) and \( u^\varepsilon_n\) are smooth because the initial data are spectrally localized. Hence, the system \eqref{G3.2} has a unique smooth solution on some time interval. The estimates below show that all the approximate solutions are defined on a common time interval \([0,T_\varepsilon]\).


\medskip
\noindent
\textbf{Step 2. Uniform estimates for $a_{n}^{\varepsilon}$.}
For later use, we define $\mathcal{X}_{\varepsilon,n}(t)$, $A_{\varepsilon,n}(t)$, $B_{\varepsilon,n}(t)$ and $U_{\varepsilon,n}(t)$ as follows:
\begin{align}\label{G3.3}
\mathcal X_{\varepsilon,n}(t):=&\,{\varepsilon
\|a_n^\varepsilon\|_{\widetilde L_t^\infty
(\dot B^{\frac d2-1}_{2,1})}
+\|a_n^\varepsilon\|_{\widetilde L_t^\infty
(\dot B^{\frac d2}_{2,1})}+\varepsilon
\|a_n^\varepsilon\|_{L_t^1
(\dot B^{\frac d2+1}_{2,1}+\dot B^{\frac d2}_{2,1})}}\nonumber\\
&+\| u_n^\varepsilon\|_{\widetilde L_t^\infty
(\dot B^{\frac d2-1}_{2,1})} +
\| u_n^\varepsilon\|_{L_t^1
(\dot B^{\frac d2+1}_{2,1})},
\end{align}
and 
\begin{align*}
A_{\varepsilon,n}(t)=\|a_n^\varepsilon\|_{\widetilde L_t^\infty
(\dot B^{\frac d2}_{2,1})},\qquad B_{\varepsilon,n}(t):=\varepsilon\|a_n^\varepsilon\|_{\widetilde L_t^\infty
(\dot B^{\frac d2-1}_{2,1})},\qquad     U_{\varepsilon,n}(t)=\| u_n^\varepsilon\|_{L_t^1
(\dot B^{\frac d2+1}_{2,1})}.
\end{align*}
Below, we will prove that $\mathcal{X}_{\varepsilon,n}(t)$ is uniformly bounded with respect to $n$.
Applying the transport estimate \eqref{G2.7} with \(s=d/2 \) in Lemma \ref{Lemma2.5}, we have
\begin{align}\label{G3.4}
\|a^\varepsilon_{n+1} \|_{\widetilde L_{T_{\varepsilon}}^\infty
(\dot B^{\frac d2}_{2,1})}
&\leq
C\exp\bigg\{C\| u^\varepsilon_n\|_{L_{T_{\varepsilon}}^1
(\dot B^{\frac d2+1}_{2,1})}\bigg\}
\bigg(
\|a^\varepsilon_{0,n+1}\|_{\dot B^{\frac d2}_{2,1}}
+
\|(1+a^\varepsilon_n){\rm div} \,u^\varepsilon_n\|_{L_{T_{\varepsilon}}^1
(\dot B^{\frac d2}_{2,1})}\bigg).
\end{align}
It follows from Lemma \ref{LA.3} that
\begin{align}\label{G3.5}
\|(1+a^\varepsilon_n){\rm div} \, u^\varepsilon_n\|_{L_{T_{\varepsilon}}^1
(\dot B^{\frac d2}_{2,1})}\lesssim&\, \|u_{n}^\varepsilon\|_{L_{T_{\varepsilon}}^1(\dot B_{2,1}^{\frac{d}{2}+1})}\bigg(1+\|a_n^\varepsilon\|_{\widetilde L_{T_{\varepsilon}}^\infty(\dot B_{2,1}^\frac{d}{2})}\bigg).
\end{align}
Putting \eqref{G3.5} into \eqref{G3.4} yields 
\begin{align*} 
\|a^\varepsilon_{n+1} \|_{\widetilde L_{T_{\varepsilon}}^\infty
(\dot B^{\frac d2}_{2,1})}\lesssim&\, \exp\bigg\{C\| u^\varepsilon_n\|_{L_{T_{\varepsilon}}^1
(\dot B^{\frac d2+1}_{2,1})}\bigg\} \bigg\{\|a^\varepsilon_0\|_{\dot B^{\frac d2}_{2,1}}+
\|u_{n}^\varepsilon\|_{L_{T_{\varepsilon}}^1(\dot B_{2,1}^{\frac{d}{2}+1})}\bigg(1+\|a_n^\varepsilon\|_{\widetilde L_{T_{\varepsilon}}^\infty(\dot B_{2,1}^\frac{d}{2})}\bigg)\bigg\},
\end{align*}
which gives rise to
\begin{align}\label{G3.7}
A_{\varepsilon,n+1}(T_{\varepsilon})
&\leq
C e^{C U_{\varepsilon,n}(T_{\varepsilon})}
\|a^\varepsilon_0\|_{\dot B^{\frac d2}_{2,1}}
+C(1+A_{\varepsilon,n}(T_{\varepsilon}))
 \big(e^{C U_{\varepsilon,n}(T_{\varepsilon})}-1\big).    
\end{align}

We now close this estimate. Assume that
\begin{align}\label{G3.8}
A_{\varepsilon,n}(T_{\varepsilon})\leq 4C\kappa_{0}\leq 1,    
\end{align}
and
\begin{align}\label{G3.9}
C U_{\varepsilon,n}(T_{\varepsilon})\leq \log(1+\kappa_0).    
\end{align}
Hence, combining \eqref{G3.7}, \eqref{G3.8} and \eqref{G3.9}, we derive
\begin{align*}
A_{\varepsilon,n+1}(T_{\varepsilon})
\leq  C(1+\kappa_0)\kappa_0+C(1+4C\kappa_0)\kappa_0\leq 4C\kappa_0\leq 1.
\end{align*}
Observing that $\dot B_{2,1}^{\frac{d}{2}}(\mathbb{R}^d)$ is continuously embedded in $L^\infty(\mathbb{R}^d)$, one can choose $\kappa_0$ sufficiently small so that
\begin{align*}
\|a_{n}^\varepsilon\|_{\widetilde L^\infty_{T_{\varepsilon}}(\dot B^{\frac{d}{2}}_{2,1})}\leq 4C\kappa_{0}\leq 1, \qquad\|a^\varepsilon_{n}\|_{L^\infty_{T_{\varepsilon}}(L^\infty)}
\leq \frac12,
\end{align*}
for any  {$n\geq 1$}.
Therefore, we obtain 
\begin{align*} 1 + a_{n }^{\varepsilon} \geq \frac{1}{2}, \end{align*}
for any  {$n\geq 1$},
which implies that \(f(a_{n }^{\varepsilon})\) is well-defined.
Furthermore,  
\begin{align*}
\|f(a^\varepsilon_{n})\|_{\dot B^{\frac d2}_{2,1}}
\leq
C\|a^\varepsilon_{n}\|_{\dot B^{\frac d2}_{2,1}}.    
\end{align*}

Now, we proceed to the estimate of \(\varepsilon a^\varepsilon_{n + 1}\).
Similar to \eqref{G3.4}, taking $s = d/2 - 1$ in Lemma \ref{Lemma2.5} yields  
\begin{align} \label{G3.10}
B_{\varepsilon,n+1}(T_{\varepsilon})
&\leq
C\exp\bigg\{C\| u^\varepsilon_n\|_{L_{T_{\varepsilon}}^1
(\dot B^{\frac d2+1}_{2,1})}\bigg\}
\bigg(
\varepsilon\|a^\varepsilon_{0,n+1}\|_{\dot B^{\frac d2-1}_{2,1}}
+
\varepsilon\|(1+a^\varepsilon_n){\rm div} \, u^\varepsilon_n\|_{L_{T_{\varepsilon}}^1
(\dot B^{\frac d2-1}_{2,1})}\bigg).
\end{align}
Using Lemmas \ref{LA.2}--\ref{LA.3}, we arrive at
\begin{align*}
&\varepsilon\|a^\varepsilon_{0,n+1}\|_{\dot B^{\frac d2-1}_{2,1}}
+
\varepsilon\|(1+a^\varepsilon_n){\rm div} \,u^\varepsilon_n\|_{L_{T_{\varepsilon}}^1
(\dot B^{\frac d2-1}_{2,1})}\nonumber\\
\lesssim&\,  \varepsilon\|a^\varepsilon_{0 }\|_{\dot B^{\frac d2-1}_{2,1}} +\varepsilon  \|u_{n}^\varepsilon\|_{L_{T_{\varepsilon}}^1(\dot B_{2,1}^{\frac{d}{2}})}\bigg(1+\|a_n^\varepsilon\|_{\widetilde L_{T_{\varepsilon}}^\infty(\dot B_{2,1}^\frac{d}{2})}\bigg) \nonumber\\
\lesssim&\, \varepsilon\|a^\varepsilon_{0 }\|_{\dot B^{\frac d2-1}_{2,1}} +\varepsilon  {T_{\varepsilon}}^{\frac{1}{2}}
\| u^\varepsilon_n\|_{\widetilde L^\infty_T(\dot B^{\frac d2-1}_{2,1})}^{\frac{1}{2}}
\| u^\varepsilon_n\|_{L^1_T(\dot B^{\frac d2+1}_{2,1})}^{\frac{1}{2}}\bigg(1+\|a_n^\varepsilon\|_{\widetilde L_{T_{\varepsilon}}^\infty(\dot B_{2,1}^\frac{d}{2})}\bigg), 
\end{align*}
which together with \eqref{G3.10}, yields
\begin{align}\label{G3.11}
 B_{\varepsilon,n+1}(T_{\varepsilon})\leq
C e^{C U_{\varepsilon,n}(T_{\varepsilon})}
\Big(
\varepsilon\|a^\varepsilon_0\|_{\dot B^{\frac d2-1}_{2,1}}
+
\varepsilon\big(1+A_{{\varepsilon,n}}(T_{\varepsilon})\big)T_{\varepsilon}^{\frac{1}{2}}
\|u^\varepsilon_n\|_{\widetilde L^\infty_{T_{\varepsilon}}(\dot B^{\frac d2-1}_{2,1})}^{\frac{1}{2}}
U_{{\varepsilon,n}}^{\frac{1}{2}}(T_{\varepsilon})
\Big).    
\end{align}
To close the estimate of \eqref{G3.9} and \(B_{\varepsilon,n}(T_{\varepsilon})\) in \eqref{G3.11}, it remains to establish uniform bounds for \( u^\varepsilon_n\), in particular for
\(\| u^\varepsilon_n\|_{L^\infty_{T_\varepsilon}
(\dot B^{\frac d2-1}_{2,1})}\). We therefore proceed to derive the velocity estimates.

\medskip
\noindent
\textbf{Step 3. Uniform estimates for $u_{n}^{\varepsilon}$.}
Define $u^\varepsilon_{L,n}(t)$ and $\bar {{u}}^\varepsilon_n$ as
\begin{align*}
 u^\varepsilon_{L,n}(t):=e^{t\mathcal A}u^\varepsilon_{0,n},
\qquad
\bar { u}^\varepsilon_n:= u^\varepsilon_n-u^\varepsilon_{L,n}.    
\end{align*}
Thanks to the Lam\'{e} semigroup estimate and the dominated convergence theorem, we have  
\begin{align*}
\sup_{ {n\geq1}} \bigg\{
\| u^\varepsilon_{L,n}\|_{\widetilde L_{T_{\varepsilon}}^2(\dot B_{2,1}^{\frac{d}{2}})}+\| u^\varepsilon_{L,n}\|_{ L_{T_{\varepsilon}}^1(\dot B_{2,1}^{\frac{d}{2}+1})} \bigg\}
\longrightarrow0
\qquad\hbox{as }\qquad T_{\varepsilon}\to0.    
\end{align*}
Let \(\eta>0\) be a small number that will be chosen later, such that
\begin{align}\label{G3.12}
\sup_{ {n\geq1}} \bigg\{
\| u^\varepsilon_{L,n}\|_{\widetilde L_{T_{\varepsilon}}^2(\dot B_{2,1}^{\frac{d}{2}})}+\| u^\varepsilon_{L,n}\|_{ L_{T_{\varepsilon}}^1(\dot B_{2,1}^{\frac{d}{2}+1})} \bigg\}\leq \eta. 
\end{align}
Since \(\bar {u}^\varepsilon_{n+1}|_{t=0}=0\), owing to the   \eqref{G3.2}$_2$, it satisfies
\begin{align*} 
\partial_t\bar { u}^\varepsilon_{n+1}
-\mathcal A {\bar u}^\varepsilon_{n+1}
=-\varepsilon^2\nabla a^\varepsilon_n
- u^\varepsilon_n\cdot\nabla  u^\varepsilon_n
- f(a^\varepsilon_n)\mathcal{A} u_{n}^\varepsilon.
\end{align*}

Assume that 
\begin{align}\label{G3.13}
 \| u_{n}^\varepsilon\|_{\widetilde L_{T_{\varepsilon}}^2(\dot B_{2,1}^{\frac{d}{2}})}+   \| u_{n}^\varepsilon\|_{ L_{T_{\varepsilon}}^1(\dot B_{2,1}^{\frac{d}{2}+1})}\leq 2\eta,
\end{align}
 then through Lemma \ref{LA.3} and Lemma \ref{Lemma2.6}, we derive  
\begin{align} \label{G3.14}
&\|\bar { u}^\varepsilon_{n+1}\|_{_{\widetilde L_{T_{\varepsilon}}^{\infty}(\dot B_{2,1}^{\frac{d}{2}-1})}}
+\|\bar { u}^\varepsilon_{n+1}\|_{_{\widetilde L_{T_{\varepsilon}}^2(\dot B_{2,1}^{\frac{d}{2}})}}
+\|\bar { u}^\varepsilon_{n+1}\|_{_{\widetilde L_{T_{\varepsilon}}^1(\dot B_{2,1}^{\frac{d}{2}+1})}}\nonumber\\
\leq&\, C\varepsilon^2
\|\nabla a^\varepsilon_n\|_{L^1_{T_{\varepsilon}}
(\dot B^{\frac d2-1}_{2,1})}+
C
\| u^\varepsilon_n\cdot\nabla u^\varepsilon_n\|_{L^1_{T_{\varepsilon}}
(\dot B^{\frac d2-1}_{2,1})}
+C
\|f(a^\varepsilon_n) \mathcal{A} u^\varepsilon_n\|_{L^1_{T_{\varepsilon}}
(\dot B^{\frac d2-1}_{2,1})}\nonumber\\
\leq&\, C\varepsilon^2  T_{\varepsilon} A_{\varepsilon,n}(T_{\varepsilon})+C
\| u^\varepsilon_n\|_{\widetilde L^2_{T_{\varepsilon}}
(\dot B^{\frac d2}_{2,1})}^2+C
A_{\varepsilon,n}(T_{\varepsilon})
\| u^\varepsilon_n\|_{L^1_{T_{\varepsilon}}
(\dot B^{\frac d2+1}_{2,1})}\nonumber\\
 \leq&\, C\eta^2+C \eta\kappa_0+C\varepsilon^2T_\varepsilon. 
 \end{align}
Then, we choose \(\kappa_0 > 0\), \(\eta > 0\), and a sufficiently small \(T_\varepsilon>0\) such that
\begin{align*}
C\eta^2+C\kappa_0  \eta+C\varepsilon^2T_\varepsilon \leq \eta,   
\end{align*}
which together with \eqref{G3.14}, leads to
\begin{align*}
 \|\bar { u}^\varepsilon_{n+1}\|_{_{\widetilde L_{T_{\varepsilon}}^{\infty}(\dot B_{2,1}^{\frac{d}{2}-1})}}+\|\bar{ u}_{n+1}^\varepsilon\|_{\widetilde L_{T_{\varepsilon}}^2(\dot B_{2,1}^{\frac{d}{2}})}+   \|\bar{u}_{n+1}^\varepsilon\|_{ L_{T_{\varepsilon}}^1(\dot B_{2,1}^{\frac{d}{2}+1})}\leq \eta.
\end{align*}
 By virtue of \eqref{G3.12}, we further get
\begin{align*}
 &\| { u}_{n+1}^\varepsilon\|_{\widetilde L_{T_{\varepsilon}}^2(\dot B_{2,1}^{\frac{d}{2}})}+   \| { u}_{n+1}^\varepsilon\|_{ L_{T_{\varepsilon}}^1(\dot B_{2,1}^{\frac{d}{2}+1})}\nonumber\\
 \leq &\,\| u^\varepsilon_{L,n}\|_{\widetilde L_{T_{\varepsilon}}^2(\dot B_{2,1}^{\frac{d}{2}})}+\|u^\varepsilon_{L,n}\|_{ L_{T_{\varepsilon}}^1(\dot B_{2,1}^{\frac{d}{2}+1})} +\|\bar{u}_{n+1}^\varepsilon\|_{\widetilde L_{T_{\varepsilon}}^2(\dot B_{2,1}^{\frac{d}{2}})}+   \|\bar{u}_{n+1}^\varepsilon\|_{ L_{T_{\varepsilon}}^1(\dot B_{2,1}^{\frac{d}{2}+1})}\nonumber\\
\leq&\, 2\eta,
\end{align*}
which closes the bootstrap assumption \eqref{G3.13}.
If we take \(\eta>0\) even smaller if necessary such that
\begin{align*}
2C\eta\leq \log(1+\kappa_0),    
\end{align*}
then \eqref{G3.9} follows.

For the remaining estimate of \(\|u^\varepsilon_n\|_{\widetilde L^\infty_{T_\varepsilon}(\dot{B}^{\frac{d}{2}-1}_{2,1})}\), we deduce from Lemma \ref{Lemma2.6} and \eqref{G3.14} that
\begin{align*}
\| u^\varepsilon_{n+1}\|_{\widetilde L^\infty_{T_\varepsilon}
(\dot B^{\frac d2-1}_{2,1})}
\leq
C\| u^\varepsilon_0\|_{\dot B^{\frac d2-1}_{2,1}}+\eta,    
\end{align*}
which gives rise to
\begin{align*}
\sup_{ {n\geq1}}
\| u^\varepsilon_n\|_{\widetilde L^\infty_{T_\varepsilon}
(\dot B^{\frac d2-1}_{2,1})}
\leq C.    
\end{align*}
As a result, we obtain that the following main bootstrap estimates are now established:
\begin{align}\label{G3.15}
\sup_{ {n\geq1}}
\bigg\{
\|a^\varepsilon_n\|_{\widetilde L^\infty_{T_\varepsilon}
(\dot B^{\frac d2}_{2,1})}
+
\| u^\varepsilon_n\|_{\widetilde L^\infty_{T_\varepsilon}
(\dot B^{\frac d2-1}_{2,1})}
+
 \| u_{n}^\varepsilon\|_{\widetilde L_{T_{\varepsilon}}^2(\dot B_{2,1}^{\frac{d}{2}})}+   \| u_{n}^\varepsilon\|_{ L_{T_{\varepsilon}}^1(\dot B_{2,1}^{\frac{d}{2}+1})}
\bigg\}\leq C,    
\end{align}
for some uniform $C>0$.
In view of the definition of $\mathcal{X}_{\varepsilon,n}$ in \eqref{G3.3}, we now return to the estimate of $\varepsilon a^{\varepsilon}_{n}$.

\medskip
\noindent
\textbf{Step 4. Uniform estimates for $\varepsilon a_{n}^{\varepsilon}$.}
We now propagate $\varepsilon a^\varepsilon_n$ in $\dot B^{\frac d2-1}_{2,1}$. Putting the estimate \eqref{G3.15} into
\eqref{G3.11}, we arrive at
\begin{align*}
B_{\varepsilon,n+1}(T_{\varepsilon})\leq C e^{2C\eta}
\Big(
\varepsilon\|a^\varepsilon_0\|_{\dot B^{\frac d2-1}_{2,1}}
+
\varepsilon(1+4C\kappa_0)T_\varepsilon^{\frac{1}{2}} (2\eta)^{\frac{1}{2}}
\Big),   
\end{align*}
 which yields
\begin{align}\label{G3.16}
 \sup_{ {n\geq1}}
\varepsilon
\|a^\varepsilon_n\|_{\widetilde L^\infty_{T_\varepsilon}
(\dot B^{\frac d2-1}_{2,1})}\leq C,  
 \end{align}
for some uniform $C>0$.
Finally, making use of the property of $a_{n}^\varepsilon$ in the low-frequency regime, one has
\begin{align*} 
\varepsilon
\|a^\varepsilon_n\|_{L^1_{T_\varepsilon}
(\dot B^{\frac d2+1}_{2,1}+\dot B^{\frac d2}_{2,1})}
\leq
\varepsilon T_\varepsilon
\|a^\varepsilon_n\|_{\widetilde L^\infty_{T_\varepsilon}
(\dot B^{\frac d2}_{2,1})}
\leq
4C\varepsilon T_\varepsilon \kappa_0,
\end{align*}
which gives
\begin{align}\label{G3.17}
 \sup_{ {n\geq1}} 
\varepsilon
\|a^\varepsilon_n\|_{\widetilde L^1_{T_\varepsilon}
(\dot B^{\frac d2+1}_{2,1}+\dot B_{2,1}^{\frac{d}{2}})}\leq C,  
 \end{align}
for some uniform $C>0$.
Collecting all the estimates \eqref{G3.15}, \eqref{G3.16}, and \eqref{G3.17} with \eqref{G3.3}, we end up with
\begin{align*}
\mathcal{X}_{\varepsilon,n}(T_{\varepsilon})\leq C,    
\end{align*}
for some uniform $C>0$. The uniform estimates established above are independent of \(n\). Next, we show that the sequence 
{$ \{(a_n^\varepsilon, u_n^\varepsilon)\}_{n\geq1}$ }is convergent in suitable lower-order Besov spaces, and that its limit satisfies the same uniform bounds due to the Fatou Lemma.
 We distinguish the cases \(d\geq3\) and \(d=2\).
 
\medskip
\noindent
\textbf{Step 5. Convergence of the iterative sequence.}
First, we will discuss the case where $d\geq3$.
Define $\delta a^\varepsilon_n$ and $\delta {u}^\varepsilon_n$ as follows:
\begin{align*}
 {\delta a^\varepsilon_{n+1}}:=a^\varepsilon_{n+1}-a^\varepsilon_n,
\qquad
 {\delta u^\varepsilon_{n+1}}:= u^\varepsilon_{n+1}-u^\varepsilon_n.    
\end{align*}
By subtracting the equations at ranks \(n + 1\) and \(n\) with respect to the system \eqref{G3.2}, we obtain
\begin{equation}  \label{G3.18}
\left\{
\begin{aligned}
&\partial_t\delta a^\varepsilon_{n+1}
+u^\varepsilon_n\cdot\nabla\delta a^\varepsilon_{n+1}=
-\delta  u^\varepsilon_n\cdot\nabla a^\varepsilon_n
-(1+a^\varepsilon_n){\rm div}\,\delta u^\varepsilon_n
-\delta a^\varepsilon_n{\rm div}\, u^\varepsilon_{n-1},\\
&\partial_t\delta  u^\varepsilon_{n+1}
-\mathcal A\delta u^\varepsilon_{n+1}+\varepsilon^2\nabla\delta a^\varepsilon_n\\
&\quad=-\delta  u^\varepsilon_n\cdot\nabla  u^\varepsilon_n
- u^\varepsilon_{n-1}\cdot\nabla\delta u^\varepsilon_n-f(a_n^\varepsilon)\mathcal{A}\delta u_{n}^\varepsilon-\big(f(a^\varepsilon_n)-f(a^\varepsilon_{n-1})\big)
\mathcal{A} { u^\varepsilon_{n-1}},\\
&\delta a^\varepsilon_{n+1}|_{t=0}=
(\dot S_{n+1}-\dot S_n)a^\varepsilon_0,\qquad \delta u^\varepsilon_{n+1}|_{t=0}=
(\dot S_{n+1}-\dot S_n)u^\varepsilon_0.
\end{aligned}
\right.
\end{equation}
For \eqref{G3.18}$_1$, we apply Lemma \ref{Lemma2.5} to get
\begin{align}\label{G3.19}
 \|\delta a^\varepsilon_{n+1}\|_{\widetilde L^\infty_{T_{\varepsilon}}
(\dot B^{\frac d2-1}_{2,1})}\lesssim&\,
\|(\dot S_{n+1}-\dot S_n)a^\varepsilon_0\|_{\dot B^{\frac d2-1}_{2,1}}
+\|\delta u^\varepsilon_n\|_{L^1_{T_{\varepsilon}}(\dot B^{\frac d2}_{2,1})}+ U_{\varepsilon,n-1}( {T_\varepsilon})
\|\delta a^\varepsilon_n\|_{\widetilde L^\infty_{T_{\varepsilon}}
(\dot B^{\frac d2-1}_{2,1})}.  
\end{align}
Similarly, through a direct calculation, for   \eqref{G3.18}$_2$, by applying Lemma \ref{Lemma2.6}, we have
\begin{align} \label{G3.20}
&\|\delta u^\varepsilon_{n+1}\|_{\widetilde L^\infty_{T_{\varepsilon}}
(\dot B^{\frac d2-2}_{2,1})}
+
\|\delta u^\varepsilon_{n+1}\|_{L^1_{T_{\varepsilon}}
(\dot B^{\frac d2}_{2,1})}\nonumber\\
\leq&\,
C\|(\dot S_{n+1}-\dot S_n) u^\varepsilon_0\|_{\dot B^{\frac d2-2}_{2,1}}
+C(\kappa_0+\eta)
\Big(
\|\delta u^\varepsilon_n\|_{\widetilde L^\infty_{T_{\varepsilon}}
(\dot B^{\frac d2-2}_{2,1})}
+
\|\delta u^\varepsilon_n\|_{L^1_{T_{\varepsilon}}
(\dot B^{\frac d2}_{2,1})}
\Big)\nonumber\\
&+
C(\eta+\varepsilon^2{T_{\varepsilon}})
\|\delta a^\varepsilon_n\|_{\widetilde L^\infty_{T_{\varepsilon}}
(\dot B^{\frac d2-1}_{2,1})}.
\end{align}
By combining the estimates \eqref{G3.19} and \eqref{G3.20}, and choosing \(\kappa_0\), \(\eta\), and \(T_\varepsilon\) to be sufficiently small, we obtain
\begin{align*}
&\sum_{n=1}^{N}\Big(\|\delta a^\varepsilon_{n+1}\|_{\widetilde L^\infty_{T_{\varepsilon}}
(\dot B^{\frac d2-1}_{2,1})}+\|\delta u^\varepsilon_{n+1}\|_{\widetilde L^\infty_{T_{\varepsilon}}
(\dot B^{\frac d2-2}_{2,1})}
+
\|\delta u^\varepsilon_{n+1}\|_{L^1_{T_{\varepsilon}}
(\dot B^{\frac d2}_{2,1})}\Big)\nonumber\\
\lesssim&\, \sum_{n=1}^N\|(\dot S_{n+1}-\dot S_n)a^\varepsilon_0\|_{\dot B^{\frac d2-1}_{2,1}}+\|(\dot S_{n+1}-\dot S_n) u^\varepsilon_0\|_{\dot B^{\frac d2-2}_{2,1}} \nonumber\\
\lesssim&\, {\sum_{n\geq 1}^{+\infty}} 2^{-n} 2^{\frac{d}{2}n}\|\dot\Delta_{n} a_0^\varepsilon\|_{L^2}+{\sum_{n\geq 1}^\infty} 2^{-n} 2^{(\frac{d}{2}-1)n}\|\dot\Delta_{n} u_0^\varepsilon\|_{L^2}\nonumber\\
\lesssim&\, \|a^\varepsilon_0\|_{\dot B^{\frac d2}_{2,1}}+\|u_0^\varepsilon\|_{\dot B_{2,1}^{\frac{d}{2}-1}}\lesssim 1,
\end{align*}
for a  fixed integer $N>1$. 
Letting \(N\to\infty\), we obtain
\begin{align*}
 \sum_{n\geq1} {\Big(\|\delta a^\varepsilon_{n+1}\|_{\widetilde L^\infty_{T_{\varepsilon}}
(\dot B^{\frac d2-1}_{2,1})}+\|\delta u^\varepsilon_{n+1}\|_{\widetilde L^\infty_{T_{\varepsilon}}
(\dot B^{\frac d2-2}_{2,1})}
+
\|\delta u^\varepsilon_{n+1}\|_{L^1_{T_{\varepsilon}}
(\dot B^{\frac d2}_{2,1})}\Big)}
 <+\infty.   
\end{align*}
Consequently, for any \(m>n\),
\begin{align*}
&\|(a^\varepsilon_m-{a^\varepsilon_1})
-(a^\varepsilon_n-{a^\varepsilon_1})\|_{\widetilde L^\infty_{T_\varepsilon}
(\dot B^{\frac d2-1}_{2,1})}+\|( u^\varepsilon_m-{u^\varepsilon_1})
-( u^\varepsilon_n-{u^\varepsilon_1})\|_{\widetilde L^\infty_{T_\varepsilon}
(\dot B^{\frac d2-2}_{2,1})}    \nonumber\\
&+\|( u^\varepsilon_m- {u^\varepsilon_1})
-( u^\varepsilon_n-{u^\varepsilon_1})\|_{L^1_{T_\varepsilon}
(\dot B^{\frac d2}_{2,1})}\nonumber\\
\leq &\,{\sum_{k=n+1}^{m}}
\|\delta a^\varepsilon_k\|_{\widetilde L^\infty_{T_\varepsilon}
(\dot B^{\frac d2-1}_{2,1})}+ {\sum_{k=n+1}^{m}}
\Big(
\|\delta u^\varepsilon_k\|_{\widetilde L^\infty_{T_\varepsilon}
(\dot B^{\frac d2-2}_{2,1})}
+
\|\delta u^\varepsilon_k\|_{L^1_{T_\varepsilon}
(\dot B^{\frac d2}_{2,1})}
\Big)
\longrightarrow0,
\end{align*}
as $m,n\rightarrow+\infty$,
which implies that \((a^\varepsilon_n-{a_1^\varepsilon},{u}^\varepsilon_n-{{u}_1^\varepsilon)}\) is a Cauchy sequence in
\begin{align*}
\widetilde L^\infty(0,T_\varepsilon;\dot B^{\frac d2-1}_{2,1})
\times
\big(
\widetilde L^\infty(0,T_\varepsilon;\dot B^{\frac d2-2}_{2,1})
\cap
L^1(0,T_\varepsilon;\dot B^{\frac d2}_{2,1})
\big).
\end{align*}
Consequently, there exists a pair \((a^\varepsilon, u^\varepsilon)\) such that
\begin{equation}  \label{NJKG3.20}
\left\{
\begin{aligned}
& a_n^\varepsilon-{a_1^\varepsilon}\to a^\varepsilon-{a_1^\varepsilon}
\qquad\text{in}\qquad
\widetilde L^\infty(0,T_\varepsilon;\dot B^{\frac d2-1}_{2,1}),\\
& u_n^\varepsilon-{u_1^\varepsilon}\to u^\varepsilon- {u_1^\varepsilon}
\,\,\,\,\,\,\,\,\,\,\,\text{in}\qquad
\widetilde L^\infty(0,T_\varepsilon;\dot B^{\frac d2-2}_{2,1})
\cap
L^1(0,T_\varepsilon;\dot B^{\frac d2}_{2,1}).
\end{aligned}
\right.
\end{equation}
Moreover, the uniform estimates and the Fatou lemma also yield
\begin{equation}  \label{G3.21}
\left\{
\begin{aligned}
& a^\varepsilon\in
\widetilde L^\infty(0,T_\varepsilon;\dot B^{\frac d2}_{2,1}),\\
& u^\varepsilon\in
\widetilde L^\infty(0,T_\varepsilon;\dot B^{\frac d2-1}_{2,1})
\cap
L^1(0,T_\varepsilon;\dot B^{\frac d2+1}_{2,1}),\\
&\varepsilon a^\varepsilon\in
\widetilde L^\infty(0,T_\varepsilon;\dot B^{\frac d2-1}_{2,1})
\cap
L^1(0,T_\varepsilon;
\dot B^{\frac d2+1}_{2,1}+\dot B^{\frac d2}_{2,1}).
\end{aligned}
\right.
\end{equation}

We now pass to the limit in the system \eqref{G3.2} for \(d\geq3\). Let $\varphi\in C_c^\infty([0,T_\varepsilon)\times\mathbb R^d)$ and $\psi\in C_c^\infty([0,T_\varepsilon)\times\mathbb R^d;\mathbb R^d)$.
For the density equation \eqref{G3.2}$_1$, in the sense of distributions and via integration by parts, one gets
\begin{align}\label{G3.22}
&-\int_0^{T_\varepsilon}\!\!\int_{\mathbb R^d}
a^\varepsilon_{n+1}\partial_t\varphi {\rm d}x {\rm d}t
-\int_{\mathbb R^d}
(\dot S_{n+1}a^\varepsilon_0)\varphi(0,x) {\rm d}x-\int_0^{T_\varepsilon}\!\!\int_{\mathbb R^d}
(1+a^\varepsilon_{n+1}) u^\varepsilon_n\cdot\nabla\varphi {\rm d}x{\rm d}t
\nonumber\\    
&\quad=\int_0^{T_\varepsilon}\!\!\int_{\mathbb R^d}
( a^\varepsilon_{n+1}-a^\varepsilon_{n}) {\rm div}\, u^\varepsilon_n
\varphi{\rm d}x{\rm d}t .
\end{align}
The first and second terms on the left-hand side of \eqref{G3.22} pass to the limit by \eqref{NJKG3.20}$_1$ and  the fact that $\dot S_{n + 1}a^\varepsilon_0\to a^\varepsilon_0$ in $\dot B^{\frac{d}{2}-1}_{2,1}$. For the third term on the left-hand side of \eqref{G3.22},  it is easy to check that
\begin{align*}
(1+a^\varepsilon_{n+1}) u^\varepsilon_n
\to
(1+a^\varepsilon) u^\varepsilon
\qquad\text{in}\qquad
L^1(0,T_\varepsilon;{\dot B^{\frac d2-1}_{2,1}}
) .   
\end{align*}
Finally, with the help of \eqref{NJKG3.20}$_1$ and \eqref{G3.21}$_2$, one has
\begin{align*}
\|( a^\varepsilon_{n+1}-a^\varepsilon_{n}) {\rm div}\, u^\varepsilon_n\|_{L^1_{T_{\varepsilon}}
 {(\dot B^{\frac d2-1}_{2,1})}}
\lesssim \|a^\varepsilon_{n+1}-a^\varepsilon_{n}\|_{\widetilde L_{T_{\varepsilon}}^\infty(\dot B_{2,1}^{\frac{d}{2}-1})}
\| u^\varepsilon_n\|_{L^1_{T_{\varepsilon}}(\dot B^{\frac d2+1}_{2,1})}
\to0.    
\end{align*}
Passing to the limit in \eqref{G3.22}, we obtain
\begin{align*}
-\int_0^{T_\varepsilon}\!\!\int_{\mathbb R^d}
a^\varepsilon\partial_t\varphi {\rm d}x{\rm d}t
-\int_{\mathbb R^d}
a_0^\varepsilon\varphi(0,x) {\rm d}x
-\int_0^{T_\varepsilon}\!\!\int_{\mathbb R^d}
(1+a^\varepsilon) u^\varepsilon\cdot\nabla\varphi{\rm d}x{\rm d}t
=0,
\end{align*}
which implies
\begin{align*}
\partial_t a^\varepsilon
+{\rm div}\,  u^\varepsilon =-{\rm div}\,(a^\varepsilon u^{\varepsilon}),    
\end{align*}
in $\mathcal D'((0,T_\varepsilon)\times\mathbb R^d)$.

Next, we consider the velocity equation \eqref{G3.2}$_2$. In weak form, the approximate equation reads
\begin{align}\label{G3.23}
&-\int_0^{T_\varepsilon}\!\!\int_{\mathbb R^d}
 u_{n+1}^\varepsilon\cdot\partial_t\boldsymbol\psi  {\rm d}x{\rm d}t
-\int_{\mathbb R^d}
(\dot S_{n+1} u_0^\varepsilon)\cdot\boldsymbol\psi(0,x) {\rm d}x
\nonumber\\
&-\int_0^{T_\varepsilon}\!\!\int_{\mathbb R^d}
 u_{n+1}^\varepsilon\cdot\mathcal A\boldsymbol\psi {\rm d}x{\rm d}t+\varepsilon^2
\int_0^{T_\varepsilon}\!\!\int_{\mathbb R^d}
\nabla a_n^\varepsilon\cdot\boldsymbol\psi{\rm d}x{\rm d}t\nonumber\\
=&\,
-\int_0^{T_\varepsilon}\!\!\int_{\mathbb R^d}
( u_n^\varepsilon\cdot\nabla u_n^\varepsilon)
\cdot\boldsymbol\psi{\rm d}x{\rm d}t
-
\int_0^{T_\varepsilon}\!\!\int_{\mathbb R^d}
f(a_n^\varepsilon)\mathcal A u_n^\varepsilon
\cdot\boldsymbol\psi{\rm d}x{\rm d}t.
\end{align}
Since  $\dot S_{n+1} u_0^\varepsilon\to u_0^\varepsilon$ in $\mathcal S'(\mathbb R^d)$ and $ u_{n+1}^\varepsilon- u^\varepsilon
\to0$ in $\widetilde L^\infty_{T_\varepsilon}(\dot B^{\frac d2-2}_{2,1})+
L^1_{T_\varepsilon}(\dot B^{\frac d2}_{2,1})$,  the linear terms and the initial term converge on the left-hand side of \eqref{G3.23}.
For the remaining terms, from \eqref{NJKG3.20} and \eqref{G3.21},  we have
\begin{align*}
&\| u_n^\varepsilon\cdot\nabla u_n^\varepsilon
- u^\varepsilon\cdot\nabla u^\varepsilon
\|_{L^1_{T_{\varepsilon}}(\dot B^{\frac d2-2}_{2,1})} \nonumber\\
\lesssim&\,
\| u_n^\varepsilon- u^\varepsilon\|_{\widetilde L^\infty_{T_{\varepsilon}}
(\dot B^{\frac d2-2}_{2,1})}
\|u_n^\varepsilon\|_{L^1_T(\dot B^{\frac d2+1}_{2,1})}
+\| u^\varepsilon\|_{\widetilde L^\infty_{T_{\varepsilon}}
(\dot B^{\frac d2-1}_{2,1})}
\| u_n^\varepsilon- u^\varepsilon\|_{L^1_{T_{\varepsilon}}
(\dot B^{\frac d2}_{2,1})}
\longrightarrow 0,    
\end{align*}
and 
\begin{align*}
&\|f(a_n^\varepsilon)\mathcal A u_n^\varepsilon
-f(a^\varepsilon)\mathcal A u^\varepsilon\|_{L^1_{T_{\varepsilon}}
(\dot B^{\frac d2-2}_{2,1})}\nonumber\\
\lesssim&\,\|(f(a_n^\varepsilon)-f(a^\varepsilon))
\mathcal A u_n^\varepsilon\|_{L^1_{T_{\varepsilon}}
(\dot B^{\frac d2-2}_{2,1})}+\|f(a^\varepsilon)\mathcal A( u_n^\varepsilon- u^\varepsilon)
\|_{L^1_{T_{\varepsilon}}(\dot B^{\frac d2-2}_{2,1})}\nonumber\\
\lesssim&\, 
\|f(a_n^\varepsilon)-f(a^\varepsilon)\|_{\widetilde L^\infty_{T_\varepsilon}
(\dot B^{\frac d2-1}_{2,1})}
\| u_n^\varepsilon\|_{L^1_{T_\varepsilon}
(\dot B^{\frac d2+1}_{2,1})}+\|f(a^\varepsilon)\|_{\widetilde L^\infty_{T_\varepsilon}
(\dot B^{\frac d2}_{2,1})}
\| u_n^\varepsilon- u^\varepsilon\|_{L^1_{T_\varepsilon}
(\dot B^{\frac d2}_{2,1})} \to0.    
\end{align*}
Passing to the limit in \eqref{G3.23}, we end up with
\begin{align*}
\partial_t u^\varepsilon
-\mathcal A u^\varepsilon+\varepsilon^2\nabla a^\varepsilon  =
- u^\varepsilon\cdot\nabla u^\varepsilon - f(a^\varepsilon)\mathcal A {u}^\varepsilon,
\end{align*}
in $\mathcal D'((0,T_\varepsilon)\times\mathbb R^d)$.
Thus, $(a^\varepsilon, u^\varepsilon)$ is a distributional solution of the scaled system \eqref{1.3}--\eqref{1.4}  for \(d\geq3\).

We further discuss the endpoint case \(d=2\). In this case, the difference estimates would have to be performed in
\begin{align*}
 \delta a^\varepsilon_n\in
\widetilde L^\infty(0,T_\varepsilon;\dot B^0_{2,1}),
\qquad
\delta u^\varepsilon_n\in
\widetilde L^\infty(0,T_\varepsilon;\dot B^{-1}_{2,1})
\cap
L^1(0,T_\varepsilon;\dot B^1_{2,1}),   
\end{align*}
and the above \(\ell^1\)-Besov contraction does not close directly at this endpoint.
We therefore use the standard compactness procedure, as in
\cite[Section 3]{Danchin-2018}; see also \cite{Danchin-IM-00}.
More precisely, one constructs smooth approximate solutions by a Friedrichs regularization of the system. The uniform estimates obtained above remain valid and result in $\{a^\varepsilon_n\}_{{n\geq1}}$ being bounded in $\widetilde L^\infty(0,T_\varepsilon;\dot B^1_{2,1})$ and $\{ u^\varepsilon_n\}_{{n\geq1}}$ being bounded in $\widetilde L^\infty(0,T_\varepsilon;\dot B^0_{2,1})\cap L^1(0,T_\varepsilon;\dot B^2_{2,1})$. By direct calculation, we deduce that
 $\{\partial_t a^\varepsilon_n\}_{{n\geq1}}$  and  $\{\partial_t  u^\varepsilon_n\}_{{n\geq1}}$ are bounded in
 $L^1(0,T_\varepsilon;\dot B^0_{2,1})$. Due to the embeddings $\dot B^1_{2,1}(\mathbb R^2)\hookrightarrow \dot H^1_{\rm loc}(\mathbb R^2)\stackrel{c}{\hookrightarrow}\ L^2_{\rm loc}(\mathbb R^2)$, the Aubin--Lions lemma and Cantor's diagonal argument imply that, up to a subsequence,
\begin{equation*}  
\left\{
\begin{aligned}
& a_n^\varepsilon\to a^\varepsilon
\qquad\text{strongly in}\qquad
C([0,T_\varepsilon];L^2_{\rm loc}(\mathbb R^2)),\\
& u_n^\varepsilon\to u^\varepsilon
\,\,\,\,\,\quad\text{strongly in}\qquad
L^2(0,T_\varepsilon;{ L_{\rm loc}^2}(\mathbb R^2)),
\end{aligned}
\right.
\end{equation*}
which, together with the uniform Besov bounds, is
sufficient to pass to the limit in the nonlinear terms of the Friedrichs regularized system. Hence, the limiting pair $(a^\varepsilon, u^\varepsilon)$
is a distributional solution to the system \eqref{1.3}--\eqref{1.4} on \((0,T_\varepsilon)\times\mathbb R^2\).
Finally, by leveraging the Fatou  lemma, we directly get
\begin{equation*}  
\left\{
\begin{aligned}
& a^\varepsilon\in
\widetilde L^\infty(0,T_\varepsilon;\dot B^1_{2,1}),\\
& u^\varepsilon\in
\widetilde L^\infty(0,T_\varepsilon;\dot B^0_{2,1})
\cap
L^1(0,T_\varepsilon;\dot B^2_{2,1}),\\
& \varepsilon a^\varepsilon\in
\widetilde L^\infty(0,T_\varepsilon;\dot B^0_{2,1})
\cap
L^1(0,T_\varepsilon;\dot B^2_{2,1}+\dot B^1_{2,1}).
\end{aligned}
\right.
\end{equation*}

Having constructed a distributional solution with the desired a priori bounds, we now turn to the uniqueness and the time-continuity properties of the solution.

\textbf{Step 6. Uniqueness and time continuity.}
We first prove the uniqueness. Let \((a^\varepsilon_1, u^\varepsilon_1)\) and \((a^\varepsilon_2,u^\varepsilon_2)\) be two solutions to \eqref{1.3}--\eqref{1.4} on \([0,T^\prime]\subset[0,T_\varepsilon]\) with the same initial data and satisfying the regularity obtained above.
Define
\begin{align*}
\delta a:=a^\varepsilon_2-a^\varepsilon_1,
\qquad
\delta u:= u^\varepsilon_2- u^\varepsilon_1.    
\end{align*}
Then \((\delta a,\delta u)\) satisfies
\begin{equation}\label{NZG3.24}
\left\{
\begin{aligned}
&\partial_t\delta a+ u^\varepsilon_2\cdot\nabla\delta a
=
-\delta u\cdot\nabla a^\varepsilon_1
-(1+a^\varepsilon_1){\rm div}\,\delta u
-\delta a{\rm div}\, u^\varepsilon_2,
\\
&\partial_t\delta u-\mathcal A\delta u+\varepsilon^2\nabla\delta a
=
-\delta u\cdot\nabla u^\varepsilon_2
- u^\varepsilon_1\cdot\nabla\delta u
-f(a^\varepsilon_1)\mathcal A\delta u
-\big(f(a^\varepsilon_2)-f(a^\varepsilon_1)\big)
\mathcal A u^\varepsilon_2,
\\
&(\delta a,\delta u)|_{t=0}=(0,0).
\end{aligned}
\right.
\end{equation}

For \(d\geq3\), we repeat the low-order stability estimates used in the
convergence part. Applying Lemma \ref{Lemma2.5} to
\eqref{NZG3.24}\(_1\) in \(\dot B^{\frac d2-1}_{2,1}\), we get
\begin{align}\label{NZG3.25}
\|\delta a\|_{\widetilde L^\infty_t
(\dot B^{\frac d2-1}_{2,1})}
\lesssim&\,\|\delta u\|_{L^1_t(\dot B^{\frac d2}_{2,1})}
+\| u^\varepsilon_2\|_{L^1_t(\dot B^{\frac d2+1}_{2,1})}
\|\delta a\|_{\widetilde L^\infty_t
(\dot B^{\frac d2-1}_{2,1})}.
\end{align}
Similarly, applying Lemma \ref{Lemma2.6}  to
\eqref{NZG3.24}\(_2\) in
\(\dot B^{\frac d2-2}_{2,1}\),  we obtain
\begin{align} \label{NZG3.26}
&\|\delta u\|_{\widetilde L^\infty_t
(\dot B^{\frac d2-2}_{2,1})}
+\|\delta u\|_{L^1_t(\dot B^{\frac d2}_{2,1})}\nonumber\\
\lesssim&\,
 {\kappa_0
\|\delta u\|_{L^1_t(\dot B^{\frac d2}_{2,1})}
+}
\Big(
\| u^\varepsilon_1\|_{L^1_t(\dot B^{\frac d2+1}_{2,1})}
+
\| u^\varepsilon_2\|_{L^1_t(\dot B^{\frac d2+1}_{2,1})}
\Big)
\|\delta u\|_{\widetilde L^\infty_t
(\dot B^{\frac d2-2}_{2,1})}
\nonumber\\
& +
\Big(
\| u^\varepsilon_2\|_{L^1_t(\dot B^{\frac d2+1}_{2,1})}
+\varepsilon^2t
\Big)
\|\delta a\|_{\widetilde L^\infty_t
(\dot B^{\frac d2-1}_{2,1})}.
\end{align}
Define
\begin{align*}
 Y(t):=
\|\delta a\|_{\widetilde L^\infty_t
(\dot B^{\frac d2-1}_{2,1})}
+
\|\delta u\|_{\widetilde L^\infty_t
(\dot B^{\frac d2-2}_{2,1})}
+
\|\delta u\|_{L^1_t(\dot B^{\frac d2}_{2,1})}.   
\end{align*}
Combining \eqref{NZG3.25} and \eqref{NZG3.26}, we infer that
\begin{align*}
 Y(t)\leq
C\bigg(
\kappa_0+\varepsilon^2t
+
\| u^\varepsilon_1\|_{L^1_t
(\dot B^{\frac d2+1}_{2,1})}+\| u^\varepsilon_2\|_{L^1_t
(\dot B^{\frac d2+1}_{2,1})}
\bigg)Y(t).   
\end{align*}
Taking \(\kappa_0>0\) sufficiently small and then taking {\(T^{\prime}>0\)} sufficiently small
such that
\begin{align*}
\kappa_0+\varepsilon^2T^{\prime}
+\|u^\varepsilon_1\|_{L^1_{ {T^{\prime}}}
(\dot B^{\frac d2+1}_{2,1})}+\|u^\varepsilon_2\|_{L^1_{ {T^{\prime}}}
(\dot B^{\frac d2+1}_{2,1})}
\leq \frac{1}{2C}.   
\end{align*}
Therefore, we have $Y(T^{\prime})=0$.
 
For \(d = 2\), the aforementioned \(\ell^1\)-Besov stability estimate is an endpoint estimate and cannot be directly closed.
Define
\begin{align*}
Z(t):=
\|\delta a\|_{\widetilde L^\infty_t(\dot B^0_{2,\infty})}+
\|\delta u\|_{\widetilde L^\infty_t(\dot B^{-1}_{2,\infty})}+
\|\delta u\|_{L^1_t(\dot B^1_{2,\infty})}.    
\end{align*}
By applying the estimates in \cite[Section 3]{Danchin-2018} to \eqref{NZG3.24}, one can obtain
\begin{align}\label{NZG3.27}
Z(t)\lesssim
C\int_0^t
\Gamma(\tau) 
Z(\tau)
\log\bigg(e+\frac{M_{T^\prime}}{Z(\tau)}\bigg){\rm d}\tau, \qquad\text{with}  \qquad Z(0)=0,  
\end{align}
where 
\begin{align*}
\Gamma(t):=
1+
\varepsilon^2+
\sum_{i=1}^2
\Big(
\| u^\varepsilon_i(t)\|_{\dot B^2_{2,1}}
+
\|a^\varepsilon_i(t)\|_{\dot B^1_{2,1}}
\| u^\varepsilon_i(t)\|_{\dot B^2_{2,1}}
\Big)
\in L^1(0,T^\prime),    
\end{align*}
and \(M_{ {T^\prime}}>0\) depends only on the norms of the two solutions on \([0,T^\prime]\).
Applying Osgood's lemma to \eqref{NZG3.27}, it follows that
\begin{align*}
Z(t)=0,\qquad 0\leq t\leq T^\prime.    
\end{align*}
Hence, uniqueness also holds in the endpoint case \(d = 2\).

We ultimately prove the time-continuity properties. This argument is applicable for all \(d\geq2\). We stress that the distinction between \(d\geq3\) and \(d = 2\) has only been utilized in the convergence and uniqueness arguments. Once the solution has been constructed with the regularity stated above, the subsequent time-continuity proof remains the same in both cases.
It is easy to check that $\partial_t a^\varepsilon
\in
L^1(0,T_\varepsilon;\dot B^{\frac d2-1}_{2,1})$.
Since, for fixed \(0<\varepsilon\leq1\), $
\varepsilon a^\varepsilon\in
\widetilde L^\infty(0,T_\varepsilon;\dot B^{\frac d2-1}_{2,1})$, 
we also have $
a^\varepsilon\in
\widetilde L^\infty(0,T_\varepsilon;\dot B^{\frac d2-1}_{2,1})$.
Hence, it holds that $
a^\varepsilon\in
\mathcal C([0,T_\varepsilon];\dot B^{\frac d2-1}_{2,1})$.

We now enhance the continuity of \(a^\varepsilon\) to the critical space \(\dot B^{\frac d2}_{2,1}\).  Let \(J\in\mathbb N\). For
\(0\leq t,t'\leq T_\varepsilon\), we decompose 
\begin{align*}
\|a^\varepsilon(t')-a^\varepsilon(t)\|_{\dot B^{\frac d2}_{2,1}}
{=}
I_J^\ell(t,t')+I_J^m(t,t')+I_J^h(t,t'),    
\end{align*}
where
\begin{align*}
I_J^\ell(t,t')
:=&\,
\sum_{j<-J}
2^{j\frac d2}
\|\dot\Delta_j(a^\varepsilon(t')-a^\varepsilon(t))\|_{L^2},\\
I_J^m(t,t')
:=&\,
\sum_{|j|\leq J}
2^{j\frac d2}
\|\dot\Delta_j(a^\varepsilon(t')-a^\varepsilon(t))\|_{L^2},\\
I_J^h(t,t')
:=&\,
\sum_{j>J}
2^{j\frac d2}
\|\dot\Delta_j(a^\varepsilon(t')-a^\varepsilon(t))\|_{L^2}.
\end{align*}
For the low-frequency part, it holds that
\begin{align*}
I_J^\ell(t,t')
\leq
2^{-J}
\|a^\varepsilon(t')-a^\varepsilon(t)\|_{\dot B^{\frac d2-1}_{2,1}}
\leq
C2^{-J},
\end{align*}
for some fixed $J$. In addition, the middle-frequency part satisfies
\begin{align*}
 I_J^m(t,t')\leq
C_J
\|a^\varepsilon(t')-a^\varepsilon(t)\|_{\dot B^{\frac d2-1}_{2,1}}
\to0
\qquad\text{as }\qquad t'\to t.   
\end{align*}
 For the high-frequency part,  
we get
\begin{align*}
I_J^h(t,t')
\leq
2\sum_{j>J}
2^{j\frac d2}
\sup_{\tau\in[0,T_\varepsilon]}
\|\dot\Delta_j a^\varepsilon(\tau)\|_{L^2}
\to0
\qquad\text{as }\qquad J\to\infty.    
\end{align*}
Combining the above three estimates gives
\begin{align*}
a^\varepsilon\in
C([0,T_\varepsilon];\dot B^{\frac d2}_{2,1}).    
\end{align*}
Proving the continuity of $u$ is exactly the same. Thus, we complete the proof of Theorem \ref{Th0}.
\end{proof}

\section{Proof of global well-posedness}

This section is devoted to the proof of the uniform-in-\(\varepsilon\) global well-posedness of solutions to the reformulated Cauchy problem \eqref{1.3}--\eqref{1.4}.
We begin by deriving {\it uniform-in-time} a priori estimates, which will be used to prove Theorem \ref{Th1}.


\subsection{A priori estimates}

\begin{prop}[Uniform a priori estimates]\label{prop3.1}
 Let \(0<\varepsilon<1\), \(T>0\) and 
\((a^\varepsilon,u^\varepsilon)\) be a strong solution to the Cauchy problem
\eqref{1.3}--\eqref{1.4} on \([0,T]\) satisfying
\begin{align}\label{NJKG4.1}
|a^\varepsilon(t,x)|\le \frac12,
\qquad (t,x)\in[0,T]\times\mathbb R^d.  
\end{align}
Then there exists a constant \(C_0>0\), independent of \(T\) and
\(\varepsilon\), such that for all \(t\in(0,T)\),
\begin{align}\label{NJKG3.1}
\mathcal X_\varepsilon(t)
\le
C_0\mathcal X_{\varepsilon,0}
+C_0\mathcal X_\varepsilon^2(t),
\end{align}
where \(\mathcal X_{\varepsilon,0}\) and \(\mathcal X_\varepsilon(t)\) are
defined by \eqref{NJK1.13} and \eqref{NJK1.15}, respectively.
Furthermore, there exists a sufficiently small constant \(\eta_0>0\) such
that if
\begin{align}\label{3.1}
\mathcal X_\varepsilon(t)\le \eta_0,
\end{align}
then
\begin{align}\label{NJKG3.3}
\mathcal X_\varepsilon(t)
\le
2C_0\mathcal X_{\varepsilon,0}.
\end{align}
\end{prop}



Since the proof of Proposition \ref{prop3.1} is lengthy and complicated, it is divided into Lemmas \ref{L4.2} to \ref{L4.5}, in which we estimate the low-, medium-, and high-frequency regimes of $(a^\varepsilon,u^\varepsilon)$, respectively.

\subsubsection{Low-frequency regime}
\begin{lem}\label{L4.2}
Let $(a^\varepsilon, {u}^\varepsilon)$ be a strong solution to the Cauchy problem \eqref{1.3}--\eqref{1.4} on $[0,T]\times \mathbb{R}^d$. Then, under the condition $\eqref{NJKG4.1}$, we have
 \begin{align}\label{3.2.1}
&\|(\varepsilon a^\varepsilon,u^\varepsilon)\|_{\widetilde L_t^\infty(\dot B_{2,1}^{\frac{d}{2}-1})}^\ell
+\|(\varepsilon a^\varepsilon,u^\varepsilon)\|_{ L^1_t(\dot B_{2,1}^{\frac{d}{2}+1})}^\ell
\lesssim 
\varepsilon\|a^\varepsilon_0\|_{\dot B_{2,1}^{\frac{d}{2}-1})}^\ell
+\|u^\varepsilon_0\|_{\dot B_{2,1}^{\frac{d}{2}-1})}^\ell+\mathcal{X}_{\varepsilon}^2(t).
\end{align}
\end{lem}

\begin{proof}
Applying the operator $\dot\Delta_k$ to \eqref{1.3}, one has
\begin{equation}\label{3.2}
\left\{
\begin{aligned}
&\partial_t\dot\Delta_ka^\varepsilon+
\dv\dot\Delta_k u^\varepsilon=-\dot\Delta_k\dv (a^\varepsilon u^\varepsilon),\\
&\partial_t \dot\Delta_ku^\varepsilon-\mathcal{A}
\dot\Delta_k u^\varepsilon
+\varepsilon^2\nabla \dot\Delta_ka^\varepsilon=-\dot\Delta_k(u^\varepsilon\cdot\nabla u^\varepsilon)
-\dot\Delta_k(f(a^\varepsilon)\mathcal{A} u^\varepsilon).
\end{aligned}
\right.
\end{equation}
Multiplying $\eqref{3.2}_1$ by $\varepsilon^2\dot\Delta_k a^{\varepsilon}$ and integrating over $\mathbb{R}^d$, we get
\begin{align}\label{3.3}
\frac{\varepsilon^2}{2}\frac{{\rm d}}{{\rm d}t}\|\dot\Delta_k a^\varepsilon\|_{L^2}^2
 +\varepsilon^2\langle \dv\dot\Delta_k u^\varepsilon, \dot\Delta_k a^\varepsilon \rangle
\lesssim\varepsilon^22^k\|\Delta_k(a^{\varepsilon}u^{\varepsilon})\|_{L^2}
\|\Delta_ka^{\varepsilon}\|_{L^2}.
\end{align}
Taking the inner product of $\eqref{3.2}_2$ and $\dot\Delta_k u^{\varepsilon}$ in $L^2$, we have 
\begin{align}\label{3.4}
&\frac{1}{2}\frac{{\rm d}}{{\rm d}t}\|\dot\Delta_k u^{\varepsilon}\|_{L^2}^2
+\mu \|\nabla \dot\Delta_k u^{\varepsilon}\|_{L^2}^2+(\mu+\lambda)\|{\rm div}\,\dot\Delta_k u^{\varepsilon}\|_{L^2}^2
+\varepsilon^2\langle\dot\Delta_k u^\varepsilon,\nabla\dot\Delta_k a^\varepsilon \rangle
\nonumber\\
\lesssim&  \|\dot\Delta_k u^{\varepsilon}\|_{L^2}
\big(\|\dot\Delta_k(u^{\varepsilon}\cdot\nabla u^{\varepsilon})\|_{L^2}
+\|\dot\Delta_k(f(a^{\varepsilon})\mathcal{A} u^{\varepsilon})\|_{L^2}\big).
\end{align}
Thus, adding \eqref{3.3} to \eqref{3.4} and employing integration by parts, we arrive at 
\begin{align}\label{3.5}
&\frac{1}{2}\frac{{\rm d}}{{\rm d}t}(\|\dot\Delta_k u^{\varepsilon}\|_{L^2}^2
+\varepsilon^2\|\dot\Delta_k a^\varepsilon\|_{L^2}^2)
+\mu \|\nabla \dot\Delta_k u^{\varepsilon}\|_{L^2}^2+(\mu+\lambda)\|{\rm div}\,\dot\Delta_k u^{\varepsilon}\|_{L^2}^2
\nonumber\\
\lesssim&  \|\dot\Delta_k(\varepsilon a^{\varepsilon}, u^{\varepsilon})\|_{L^2}
\big(\varepsilon2^k\|\Delta_k(a^{\varepsilon}u^{\varepsilon})\|_{L^2}+\|\dot\Delta_k(u^{\varepsilon}\cdot\nabla u^{\varepsilon})\|_{L^2}+\|\dot\Delta_k(f(a^{\varepsilon})\mathcal{A}u^{\varepsilon})\|_{L^2}\big).
\end{align}

In addition, multiplying $\eqref{3.2}_2$ by $\nabla\dot\Delta_ka^{\varepsilon}$ and integrating over $\mathbb{R}^d$ yield 
\begin{align}\label{3.6}
&\frac{{\rm d}}{{\rm d}t}\langle\dot\Delta_k u^\varepsilon,\nabla\dot\Delta_k a^\varepsilon \rangle
+\varepsilon^2\|\nabla\dot\Delta_k a^\varepsilon\|^2_{L^2}
-\langle\dot\Delta_k u^\varepsilon,\nabla\dot\Delta_k \partial_ta^\varepsilon \rangle
\nonumber\\
\lesssim
&\,\|\nabla\dot\Delta_k a^{\varepsilon}\|_{L^2}
\big(\|\dot\Delta_k(u^{\varepsilon}\cdot\nabla u^{\varepsilon})\|_{L^2}+\|\dot\Delta_k(f(a^{\varepsilon})\mathcal{A} u^{\varepsilon})\|_{L^2}\big)+\langle\dot\Delta_k\mathcal{A} {u}^\varepsilon,\nabla\dot\Delta_k a^\varepsilon \rangle.
\end{align}
By virtue of $\eqref{1.3}_1$ and integration by parts, we further get 
\begin{align}\label{3.7}
&\frac{{\rm d}}{{\rm d}t}\langle\dot\Delta_k u^\varepsilon,\nabla\dot\Delta_k a^\varepsilon \rangle
+\varepsilon^2\|\nabla\dot\Delta_k a^\varepsilon\|^2_{L^2}
-\|\dv\Delta_k u^\varepsilon\|^2_{L^2}
\nonumber\\
\lesssim
&\,\|\nabla\dot\Delta_k a^{\varepsilon}\|_{L^2}
\big(\|\dot\Delta_k(u^{\varepsilon}\cdot\nabla u^{\varepsilon})\|_{L^2}+\|\dot\Delta_k(f(a^{\varepsilon})\mathcal{A} u^{\varepsilon})\|_{L^2}
\big)+\langle\mathcal{A}\dot\Delta_k u^\varepsilon,\nabla\dot\Delta_k a^\varepsilon \rangle\nonumber\\
&-\langle\nabla\dv\dot\Delta_k (a^\varepsilon u^\varepsilon),\dot\Delta_k u^\varepsilon \rangle.
\end{align}
Then, multiplying $\eqref{3.7}$ by a small enough constant $\eta_1$,  adding the resulting inequality to \eqref{3.5} and using the smallness of $\eta_1$, we deduce that
\begin{align}\label{3.5.1}
\frac{1}{2}\frac{{\rm d}}{{\rm d}t}E_{l,k}(t)
+D_{l,k}(t)
\lesssim 
&\|\dot\Delta_k(\varepsilon a^{\varepsilon}, u^{\varepsilon})\|_{L^2}
\big(\varepsilon2^k\|\Delta_k(a^{\varepsilon}u^{\varepsilon})\|_{L^2}
+\|\dot\Delta_k(u^{\varepsilon}\cdot\nabla u^{\varepsilon})\|_{L^2}\nonumber\\
&+\|\dot\Delta_k(f(a^{\varepsilon})\mathcal{A}u^{\varepsilon})\|_{L^2}
\big),
\end{align}
where
\begin{align}
&E_{l,k}(t):=\varepsilon^2\|\dot\Delta_k a^\varepsilon\|^2_{L^2}+\|\dot\Delta_k u^\varepsilon\|^2_{L^2}+\eta_1\langle\dot\Delta_k {u}^\varepsilon,\nabla\dot\Delta_k a^\varepsilon \rangle,\nonumber\\
&D_{l,k}(t):=\varepsilon^2\|\nabla\dot\Delta_k a^\varepsilon\|^2_{L^2}+\|\nabla\dot\Delta_k u^\varepsilon\|^2_{L^2}.\nonumber
\end{align}
Given that  $k\leq J^\varepsilon_0$ and $2^{J^\varepsilon_0}\sim\varepsilon$, and using the inequality \eqref{1.13.1}$_1$, it follows that
\begin{align}\label{3.9}
\eta_1\langle\dot\Delta_k u^\varepsilon,\nabla\dot\Delta_k a^\varepsilon \rangle
\leq \frac{\eta_1}{2}(\|\nabla\dot\Delta_k a^\varepsilon\|^2_{L^2}+\|\dot\Delta_k u^\varepsilon\|^2_{L^2})
\leq \frac{\eta_1}{2}(\varepsilon^2\|\dot\Delta_k a^\varepsilon\|^2_{L^2}+\|\dot\Delta_k u^\varepsilon\|^2_{L^2}),\nonumber
\end{align}
which together with the smallness of $\eta_1$, implies 
\begin{align}
E_{l,k}(t)\sim\varepsilon^2\|\dot\Delta_k a^\varepsilon\|^2_{L^2}+\|\dot\Delta_k u^\varepsilon\|^2_{L^2}.
\end{align}
Additionally, we also have 
\begin{align}\label{3.10}
D_{l,k}(t)\gtrsim 2^{2k}E_{l,k}(t).
\end{align}
Thus, combining \eqref{3.5.1} with \eqref{3.9}--\eqref{3.10}, we obtain 
\begin{align}\label{3.11}
\frac{1}{2}\frac{{\rm d}}{{\rm d}t}E_{l,k}(t)
+2^{2k}E_{l,k}(t)
\lesssim
&\sqrt{E_{l,k}(t)}
\big(\varepsilon2^k\|\Delta_k(a^{\varepsilon}u^{\varepsilon})\|_{L^2}
+\|\dot\Delta_k(u^{\varepsilon}\cdot\nabla u^{\varepsilon})\|_{L^2}+\|\dot\Delta_k(f(a^{\varepsilon})\mathcal{A}u^{\varepsilon})\|_{L^2}\big).
\end{align}
Dividing \eqref{3.11} by $\big(E_{l,k}(t) + \eps_{*}^2\big)^{\frac{1}{2}}$ with $\eps_{*} > 0$, we integrate the resulting inequality over $[0,t]$, and then pass to the limit as $\eps_{*} \to 0$ to obtain
\begin{align}\label{3.12}
&\|\dot\Delta_k (\varepsilon a^\varepsilon, u^\varepsilon)\|_{L^2} +2^{2k}\int_0^t \|\dot\Delta_k (\varepsilon a^\varepsilon, u^\varepsilon)\|_{L^2}{\rm d}\tau\nonumber\\
\lesssim&\,  \|\dot\Delta_k (\varepsilon a^\varepsilon_0, u^\varepsilon_0)\|_{L^2} +\int_0^t  \big(\varepsilon2^k\|\Delta_k(a^{\varepsilon}u^{\varepsilon})\|_{L^2}+\|\dot\Delta_k(u^\varepsilon\cdot\nabla u^\varepsilon)\|_{L^2}+\|\dot\Delta_k(f(a^\varepsilon)\mathcal{A} u^\varepsilon)\|_{L^2}\big){\rm d}\tau.
\end{align}
Multiplying \eqref{3.12} by $2^{k(\frac{d}{2}-1)}$ , taking the supremum over $[0,t]$, and summing over all 
$k \leq J^\varepsilon_0$,  we arrive at
\begin{align}\label{3.13}
&\varepsilon\|a^\varepsilon\|_{\widetilde L_t^\infty(\dot B_{2,1}^{\frac{d}{2}-1})}^\ell+
\|u^\varepsilon\|_{\widetilde L_t^\infty(\dot B_{2,1}^{\frac{d}{2}-1})}^\ell
+\varepsilon\|a^\varepsilon\|_{ L^1_t(\dot B_{2,1}^{\frac{d}{2}+1})}^\ell+
\|u^\varepsilon\|_{ L^1_t(\dot B_{2,1}^{\frac{d}{2}+1})}^\ell\nonumber\\
\lesssim& \, 
\varepsilon\|a^\varepsilon_0\|_{\dot B_{2,1}^{\frac{d}{2}-1})}^\ell
+\|u^\varepsilon_0\|_{\dot B_{2,1}^{\frac{d}{2}-1})}^\ell+
\varepsilon\|a^{\varepsilon}u^{\varepsilon}\|_{L_t^1(\dot B_{2,1}^{\frac{d}{2}})}^\ell
+\|u^{\varepsilon}\cdot\nabla u^{\varepsilon}\|_{L_t^1(\dot B_{2,1}^{\frac{d}{2}-1})}^\ell+ \|f(a^{\varepsilon})\mathcal{A}u^{\varepsilon}\|_{L_t^1(\dot B_{2,1}^{\frac{d}{2}-1})}^\ell.
\end{align}

It remains to estimate the nonlinear terms on the right-hand side of \eqref{3.13}. 
First, we focus on the term $\varepsilon\|a^{\varepsilon}u^{\varepsilon}\|_{L_t^1(\dot B_{2,1}^{\frac{d}{2}})}^\ell$. By virtue of Lemmas \ref{LA.2} and \ref{LA.3}, we have
\begin{align}\label{3.14}
\varepsilon\|a^{\varepsilon}u^{\varepsilon}\|_{L_t^1(\dot B_{2,1}^{\frac{d}{2}})}^\ell 
\lesssim& \,\|\varepsilon a^{\varepsilon}\|_{L_t^2(\dot B_{2,1}^{\frac{d}{2}})} \| u^{\varepsilon}\|_{L_t^2(\dot B_{2,1}^{\frac{d}{2}})}\nonumber\\
\lesssim&\,\bigg(\|u^{\varepsilon}\|^{\frac{1}{2}}_{L_t^1(\dot B_{2,1}^{\frac{d}{2}+1})}
\|u^{\varepsilon}\|^{\frac{1}{2}}_{\widetilde L_t^\infty(\dot B_{2,1}^{\frac{d}{2}-1})}\bigg)
\bigg(\Big(\varepsilon\|a^{\varepsilon}\|^\ell_{L_t^1(\dot B_{2,1}^{\frac{d}{2}+1})}\Big)^{\frac{1}{2}}
\Big(\varepsilon\|a^{\varepsilon}\|^\ell_{\widetilde L_t^\infty(\dot B_{2,1}^{\frac{d}{2}-1})}\Big)^{\frac{1}{2}}\nonumber\\
&\,+\Big(\varepsilon^2\|a^{\varepsilon}\|^{m}_{L_t^1(
\dot B_{2,1}^{\frac{d}{2}})}\Big)^{\frac{1}{2}}
\Big(\|a^{\varepsilon}\|^{m}_{\widetilde L_t^\infty(\dot B_{2,1}^{\frac{d}{2}})}\Big)^{\frac{1}{2}}
+\Big(\varepsilon^2\|a^{\varepsilon}\|^{h}_{L_t^1(
\dot B_{2,1}^{\frac{d}{2}})}\Big)^{\frac{1}{2}}
\Big(\|a^{\varepsilon}\|^{h}_{\widetilde L_t^\infty(\dot B_{2,1}^{\frac{d}{2}})}\Big)^{\frac{1}{2}}
\bigg)\nonumber\\
\lesssim &\,\mathcal{X}_{\varepsilon}^2(t).
\end{align}
Using \eqref{1.13.2}$_1$, \eqref{1.13.1}$_2$ and Lemmas \ref{LA.3}--\ref{LA.4}, we get
\begin{align}\label{3.15}
\|u^{\varepsilon}\cdot\nabla u^{\varepsilon}\|_{L_t^1(\dot B_{2,1}^{\frac{d}{2}-1})}^\ell\lesssim \|u^{\varepsilon}\|_{\widetilde L_t^\infty(\dot B_{2,1}^{\frac{d}{2}-1})}
\|u^{\varepsilon}\|_{ L_t^1(\dot B_{2,1}^{\frac{d}{2}+1})}
\lesssim \mathcal{X}_{\varepsilon}^2(t),
\end{align}
and
\begin{align}\label{3.16}
\|f(a^{\varepsilon})\mathcal{A} u^{\varepsilon}\|_{L_t^1(\dot B_{2,1}^{\frac{d}{2}-1})}^\ell
\lesssim& \|f(a^{\varepsilon})\|_{\widetilde L_t^\infty(\dot B_{2,1}^{\frac{d}{2}})}
\|u^{\varepsilon}\|_{L_t^1(\dot B_{2,1}^{\frac{d}{2}+1})}\nonumber\\
\lesssim& \|a^{\varepsilon}\|_{\widetilde L_t^\infty(\dot B_{2,1}^{\frac{d}{2}})}
\|u^{\varepsilon}\|_{L_t^1(\dot B_{2,1}^{\frac{d}{2}+1})}\nonumber\\
\lesssim& \bigg(\varepsilon\|a^{\varepsilon}\|^\ell_{\widetilde L_t^\infty(\dot B_{2,1}^{\frac{d}{2}-1})}
+\|a^{\varepsilon}\|^{m}_{\widetilde L_t^\infty(\dot B_{2,1}^{\frac{d}{2}})}+\|a^{\varepsilon}\|^{h}_{\widetilde L_t^\infty(\dot B_{2,1}^{\frac{d}{2}})}\bigg)
\|u^{\varepsilon}\|_{L_t^1(\dot B_{2,1}^{\frac{d}{2}+1})}
\lesssim \mathcal{X}_{\varepsilon}^2(t).
\end{align}
Thus, putting \eqref{3.14}--\eqref{3.16} into \eqref{3.13} yields \eqref{3.2.1}.
\end{proof}

\subsubsection{Medium-frequency regime.}
\begin{lem}\label{L4.3}
Let $(a^\varepsilon, {u}^\varepsilon)$ be a strong solution to the Cauchy problem \eqref{1.3}--\eqref{1.4} on $[0,T]\times \mathbb R^d$. Then, under the condition $\eqref{NJKG4.1}$, we have
 \begin{align}
&\|a^\varepsilon\|_{\widetilde L_t^\infty(\dot B_{2,1}^{\frac{d}{2}})}^m+
\|u^\varepsilon\|_{\widetilde L_t^\infty(\dot B_{2,1}^{\frac{d}{2}-1})}^m
+\varepsilon^2\bigg(\|a^\varepsilon\|_{ L^1_t(\dot B_{2,1}^{\frac{d}{2}})}^m+
\|u^\varepsilon\|_{ L^1_t(\dot B_{2,1}^{\frac{d}{2}-1})}^m\bigg)
\lesssim
\|a^\varepsilon_0\|_{\dot B_{2,1}^{\frac{d}{2}}}^m
+\|u^\varepsilon_0\|_{\dot B_{2,1}^{\frac{d}{2}-1}}^m+\mathcal{X}_{\varepsilon}^2(t).
\end{align}
\end{lem}

\begin{proof}
We first apply $\nabla$ to the equation $\eqref{3.2}_1$, and then multiply this resulting identity as well as the equation $\eqref{3.2}_2$ by $\nabla\Delta_k a^\varepsilon$. 
Consequently, we get
\begin{align}\label{3.19}
\frac{1}{2}\frac{{\rm d}}{{\rm d}t}\|\nabla\Delta_k a^\varepsilon\|_{L^2}^2
+\langle \nabla\dv\dot\Delta_k u^\varepsilon, \nabla\dot\Delta_k a^\varepsilon \rangle
=-
\langle \nabla\dv\dot\Delta_k (a^\varepsilon u^\varepsilon ), \nabla\dot\Delta_k a^\varepsilon \rangle.
\end{align}
and 
\begin{align}\label{3.20}
&\frac{{\rm d}}{{\rm d}t}\langle\dot\Delta_k u^\varepsilon, \nabla\dot\Delta_k a^\varepsilon \rangle
-\langle\mathcal{A}\dot\Delta_k u^\varepsilon, \nabla\dot\Delta_k a^\varepsilon \rangle
+\varepsilon^2\|\nabla\dot\Delta_k a^{\varepsilon}\|^2_{L^2}
\nonumber\\
\lesssim&  \|\dv\dot\Delta_k u^{\varepsilon}\|^2_{L^2}
-\langle\nabla\dv\dot\Delta_k (a^\varepsilon u^\varepsilon), \dot\Delta_k u^\varepsilon \rangle
-\langle\dot\Delta_k (u^\varepsilon\cdot\nabla u^\varepsilon), \nabla\dot\Delta_k a^\varepsilon \rangle\nonumber\\
&+\|\nabla\dot\Delta_k a^{\varepsilon}\|_{L^2}
\|\dot\Delta_k(f(a^{\varepsilon})\mathcal{A}u^{\varepsilon})\|_{L^2}.
\end{align}
Thus, multiplying \eqref{3.19} by $(2\mu+\lambda)$, then adding the resulting equality to \eqref{3.20} and employing integration by parts, we arrive at 
\begin{align}\label{4.24}
&\frac{{\rm d}}{{\rm d}t}\bigg(\frac{(2\mu+\lambda)}{2}\|\nabla\Delta_k a^\varepsilon\|_{L^2}^2
+\langle\dot\Delta_k u^\varepsilon, \nabla\dot\Delta_k a^\varepsilon \rangle\bigg)+\varepsilon^2\|\nabla\dot\Delta_k a^\varepsilon\|^2
\nonumber\\
\lesssim& \, \|\dv\dot\Delta_k u^{\varepsilon}\|^2_{L^2}
+\|\dot\Delta_k(u^\varepsilon,\nabla a^\varepsilon)\|_{L^2}
\Big(\|\nabla \dot\Delta_k(a^\varepsilon \dv u^\varepsilon)\|_{L^2}+\sum^d_{j=1}\|[u^\varepsilon\cdot\nabla,\partial_j\dot\Delta_k]a^\varepsilon\|_{L^2}\nonumber\\
&+\|\dv u^\varepsilon\|_{L^\infty}
\|\dot\Delta_k(u^\varepsilon,\nabla a^\varepsilon)\|_{L^2}
+\|[u^\varepsilon\cdot\nabla,\dot\Delta_k]u^\varepsilon\|_{L^2}
+\|\dot\Delta_k(f(a^{\varepsilon})\mathcal{A}u^{\varepsilon})\|_{L^2}\Big).
\end{align}
Arguing analogously to \eqref{3.5}, we directly obtain 
\begin{align}\label{4.25}
&\frac{1}{2}\frac{{\rm d}}{{\rm d}t}(\|\dot\Delta_k u^{\varepsilon}\|_{L^2}^2
+\varepsilon^2\|\dot\Delta_k a^\varepsilon\|_{L^2}^2)
+\mu \|\nabla \dot\Delta_k u^{\varepsilon}\|_{L^2}^2
+(\mu+\lambda)\|\dv\dot\Delta_k u^{\varepsilon}\|_{L^2}^2
\nonumber\\
\lesssim& \, \|\dot\Delta_k(\varepsilon a^{\varepsilon}, u^{\varepsilon})\|_{L^2}
\Big(\varepsilon\|[u^\varepsilon\cdot\nabla,\dot\Delta_k]a^\varepsilon\|_{L^2}+
\varepsilon\|\dv u^\varepsilon\|_{L^\infty}\|\dot\Delta_k a^\varepsilon\|_{L^2}
+\varepsilon\|\dot\Delta_k(a^\varepsilon\dv u^\varepsilon)\|_{L^2}\nonumber\\
&+\|[u^\varepsilon\cdot\nabla,\dot\Delta_k]u^\varepsilon\|_{L^2}
+\|\dv u^\varepsilon\|_{L^\infty}\|\dot\Delta_ku^\varepsilon\|_{L^2}
+\|\dot\Delta_k(f(a^{\varepsilon})\mathcal{A} u^{\varepsilon})\|_{L^2}\Big).
\end{align}
Furthermore, multiplying $\eqref{4.24}$ by a small enough constant $\eta_2$,  adding the resulting inequality to \eqref{4.25} and using the smallness of $\eta_2$, we arrive at
\begin{align}\label{4.26}
&\frac{{\rm d}}{{\rm d}t}E_{m,k}(t)
+D_{m,k}(t)\nonumber\\
\lesssim& \|\dot\Delta_k(\nabla a^{\varepsilon}, u^{\varepsilon})\|_{L^2}\Big(\varepsilon\|[u^\varepsilon\cdot\nabla,\dot\Delta_k]a^\varepsilon\|_{L^2}+
\|\dv u^\varepsilon\|_{L^\infty}\|\dot\Delta_k (\nabla a^\varepsilon,u^\varepsilon)\|_{L^2}
+\|\dot\Delta_k\nabla(a^\varepsilon\dv u^\varepsilon)\|_{L^2}
\nonumber\\
&+\sum^d_{j=1}\|[u^\varepsilon\cdot\nabla,\partial_j\dot\Delta_k]a^\varepsilon\|_{L^2}
+\|[u^\varepsilon\cdot\nabla,\dot\Delta_k]
u^\varepsilon\|_{L^2}
+\|\dot\Delta_k(f(a^{\varepsilon})\mathcal{A} u^{\varepsilon})\|_{L^2}
\Big),
\end{align}
where
\begin{align}
&E_{m,k}(t):=\varepsilon^2\|\dot\Delta_k a^\varepsilon\|^2_{L^2}+\|\dot\Delta_k u^\varepsilon\|^2_{L^2}+\eta_2\Big(
\frac{2\mu+\lambda}{2}\|\nabla\dot\Delta_ka^\varepsilon\|^2_{L^2}+\langle\dot\Delta_k u^\varepsilon,\nabla\dot\Delta_k a^\varepsilon \rangle\Big),\nonumber\\
&D_{m,k}(t):=\varepsilon^2\|\nabla\dot\Delta_k a^\varepsilon\|^2_{L^2}+\mu\|\nabla\dot\Delta_k u^\varepsilon\|^2_{L^2}+(\mu+\lambda)\|\dv\dot\Delta_k u^\varepsilon\|^2_{L^2}.\nonumber
\end{align}
It is easy to check that 
\begin{align}
\big|\langle\dot\Delta_k u^\varepsilon,\nabla\dot\Delta_k a^\varepsilon \rangle\big|\leq \frac{2\mu+\lambda}{4}\|\dot\Delta_k \nabla a^\varepsilon\|^2_{L^2}+\frac{\|\dot\Delta_k u^\varepsilon\|^2_{L^2}}{2\mu+\lambda},
\end{align}
which together with the smallness of $\eta_2$, yields 
\begin{align}\label{4.28}
E_{m,k}(t)\sim\|\dot\Delta_k \nabla a^\varepsilon\|^2_{L^2}+\|\dot\Delta_k u^\varepsilon\|^2_{L^2}.
\end{align}
In addition, based on the fact of $\varepsilon<2^k\leq \varepsilon^{1-\delta_c}$, we get
\begin{align}\label{4.29}
D_{m,k}(t)\geq\varepsilon^2\big(\|\nabla\dot\Delta_k a^\varepsilon\|^2_{L^2}+\|\dot\Delta_k u^\varepsilon\|^2_{L^2}\big).
\end{align}
Therefore, integrating \eqref{4.26} over $[0,t]$, and then  using \eqref{4.28} and \eqref{4.29}, we deduce that 
\begin{align}\label{4.30}
&\|\dot\Delta_k (\nabla a^\varepsilon, u^\varepsilon)\|_{L^2} +\varepsilon^2\|\dot\Delta_k (\nabla a^\varepsilon, u^\varepsilon)\|_{L^1_t(L^2)}\nonumber\\
\lesssim&\,  \|\dot\Delta_k (\nabla a^\varepsilon_0, u^\varepsilon_0)\|_{L^2} +\varepsilon\|[u^\varepsilon\cdot\nabla,\dot\Delta_k]a^\varepsilon\|_{L^1_t(L^2)}+
\|\dv u^\varepsilon\|_{L^1_t(L^\infty)}\|\dot\Delta_k (\nabla a^\varepsilon,u^\varepsilon)\|_{L^\infty_t(L^2)}\nonumber\\
&+2^k\|\dot\Delta_k(a^\varepsilon\dv u^\varepsilon)\|_{L^1_t(L^2)}
+\sum^d_{j=1}\|[u^\varepsilon\cdot\nabla,\partial_j\dot\Delta_k]a^\varepsilon\|_{L^1_t(L^2)}
+\|[u^\varepsilon\cdot\nabla,\dot\Delta_k]
u^\varepsilon\|_{L^1_t(L^2)}\nonumber\\
&+\|\dot\Delta_k(f(a^{\varepsilon})\mathcal{A}u^{\varepsilon})\|_{L^1_t(L^2)}.
\end{align}
Multiplying \eqref{4.30} by $2^{k(\frac{d}{2}-1)}$, summing over all $J_0^\varepsilon< k\leq J^\varepsilon_1$ and using Lemma \ref{A.2},  we get 
\begin{align}\label{4.31}
&\|a^\varepsilon\|_{\widetilde L_t^\infty(\dot B_{2,1}^{\frac{d}{2}})}^{m}+
\|u^\varepsilon\|_{\widetilde L_t^\infty(\dot B_{2,1}^{\frac{d}{2}-1})}^m
+\varepsilon^2\big(\|a^\varepsilon\|_{ L^1_t(\dot B_{2,1}^{\frac{d}{2}})}^m+
\|u^\varepsilon\|_{ L^1_t(\dot B_{2,1}^{\frac{d}{2}-1})}^m\big)\nonumber\\
\lesssim& \, \|a^\varepsilon_0\|_{\dot B_{2,1}^{\frac{d}{2}}}^m
+\|u^\varepsilon_0\|_{\dot B_{2,1}^{\frac{d}{2}-1}}^m+
\|f(a^{\varepsilon})\mathcal{A}u^{\varepsilon}\|^m_{L^1_t(\dot B_{2,1}^{\frac{d}{2}-1})}+\|a^\varepsilon\dv u^\varepsilon\|^m_{L^1_t(\dot B_{2,1}^{\frac{d}{2}})}\nonumber\\
&+\|u^\varepsilon\|_{L^1_t(\dot B_{2,1}^{\frac{d}{2}+1})} (\|\nabla a^\varepsilon\|^m_{\widetilde L^\infty_t(\dot B_{2,1}^{\frac{d}{2}-1})}+\|u^\varepsilon\|^m_{\widetilde L^\infty_t(\dot B_{2,1}^{\frac{d}{2}-1})})+\varepsilon\sum_{J^\varepsilon_0< k\leq J^\varepsilon_1}2^{k(\frac{d}{2}-1)}\|[u^\varepsilon\cdot\nabla,\dot\Delta_k]a^\varepsilon\|_{L^1_t(L^2)}\nonumber\\
&
+\sum^d_{j=1}\sum_{J^\varepsilon_0< k\leq J^\varepsilon_1}2^{k(\frac{d}{2}-1)}\|[u^\varepsilon\cdot\nabla,\partial_j\dot\Delta_k]a^\varepsilon\|_{L^1_t(L^2)}
+\sum_{J^\varepsilon_0< k\leq J^\varepsilon_1}2^{k(\frac{d}{2}-1)}\|[u^\varepsilon\cdot\nabla,\dot\Delta_k]
u^\varepsilon\|_{L^1_t(L^2)}.
\end{align}

Next, we focus on the estimates of the nonlinear terms on the right-hand side of \eqref{4.31}. For the first nonlinear term, by virtue of Lemmas \ref{A.3} and \ref{A.4}, it holds that
\begin{align}\label{4.32}
\|f(a^{\varepsilon})\mathcal{A} u^{\varepsilon}\|^m_{L^1_t(\dot B_{2,1}^{\frac{d}{2}-1})}\lesssim\|f(a^\varepsilon)\|_{L^\infty_t(\dot B_{2,1}^{\frac{d}{2}})}\|u^\varepsilon\|_{L^1_t(\dot B_{2,1}^{\frac{d}{2}+1})}\lesssim\|a^\varepsilon\|_{L^\infty_t(\dot B_{2,1}^{\frac{d}{2}})}\|u^\varepsilon\|_{L^1_t(\dot B_{2,1}^{\frac{d}{2}+1})}\lesssim \mathcal{X}^2_\varepsilon(t)
\end{align}
and 
\begin{align}\label{4.33}
\|a^\varepsilon\dv u^\varepsilon\|^m_{L^1_t(\dot B_{2,1}^{\frac{d}{2}})}\lesssim\|a^\varepsilon\|_{L^\infty_t(\dot B_{2,1}^{\frac{d}{2}})}\|u^\varepsilon\|_{L^1_t(\dot B_{2,1}^{\frac{d}{2}+1})}\lesssim \mathcal{X}^2_\varepsilon(t).
\end{align}
It remains to control the terms containing commutators.
With the help of Lemma \ref{LA.5}, \eqref{1.13.2}$_1$ and \eqref{1.13.1}$_2$, we obtain 
\begin{align}\label{4.34}
&\varepsilon\sum_{J^\varepsilon_0< k\leq J^\varepsilon_1}2^{k(\frac{d}{2}-1)}\|[u^\varepsilon\cdot\nabla,\dot\Delta_k]a^\varepsilon\|_{L^1_t(L^2)}\nonumber\\
\lesssim& \,\varepsilon\|a^\varepsilon\|_{\widetilde L^\infty_t(\dot B_{2,1}^{\frac{d}{2}-1})}\|u^\varepsilon\|_{L^1_t(\dot B_{2,1}^{\frac{d}{2}+1})}\nonumber\\
\lesssim& \,\big(\varepsilon\|a^\varepsilon\|^\ell_{\widetilde L^\infty_t(\dot B_{2,1}^{\frac{d}{2}-1})}+\|a^\varepsilon\|^m_{\widetilde L^\infty_t(\dot B_{2,1}^{\frac{d}{2}})}+\|a^\varepsilon\|^h_{\widetilde L^\infty_t(\dot B_{2,1}^{\frac{d}{2}})})\|u^\varepsilon\|_{L^1_t(\dot B_{2,1}^{\frac{d}{2}+1})}
\lesssim \mathcal{X}^2_\varepsilon(t)
\end{align}
and 
\begin{align}\label{4.35}
\sum_{J^\varepsilon_0< k\leq J^\varepsilon_1}2^{k(\frac{d}{2}-1)}\|[u^\varepsilon\cdot\nabla,\dot\Delta_k]
u^\varepsilon\|_{L^1_t(L^2)}
\lesssim \|u^\varepsilon\|_{\widetilde L^\infty_t(\dot B_{2,1}^{\frac{d}{2}-1})}\|u^\varepsilon\|_{L^1_t(\dot B_{2,1}^{\frac{d}{2}+1})}\lesssim \mathcal{X}^2_\varepsilon(t).
\end{align}
Simple calculation yields 
\begin{align}
[u^\varepsilon\cdot\nabla,\partial_j\dot\Delta_k]a^\varepsilon
=&\,u^\varepsilon\cdot\nabla(\dot\Delta_k\partial_ja^\varepsilon)-\dot\Delta_k(\partial_ju^\varepsilon\cdot\nabla a^\varepsilon)-\dot\Delta_k(u^\varepsilon\cdot\nabla \partial_ja^\varepsilon)\nonumber\\
=&\,[u^\varepsilon\cdot\nabla, \dot\Delta_k]\partial_ja^\varepsilon-\dot\Delta_k(\partial_ju^\varepsilon\cdot\nabla a^\varepsilon).\nonumber
\end{align}
Thus, by virtue of Lemmas \ref{A.3} and \ref{LA.5}, we have
\begin{align}\label{4.37}
\sum^d_{j=1}\sum_{J^\varepsilon_0<k\leq J^\varepsilon_1}2^{k(\frac{d}{2}-1)}\|[u^\varepsilon\cdot\nabla,\partial_j\dot\Delta_k]a^\varepsilon\|_{L^1_t(L^2)}
\lesssim \|\nabla a^\varepsilon\|_{\widetilde L^\infty_t(\dot B_{2,1}^{\frac{d}{2}-1})}
\|u^\varepsilon\|_{L^1_t(\dot B_{2,1}^{\frac{d}{2}+1})}\lesssim \mathcal{X}^2_\varepsilon(t).
\end{align}

Finally, putting \eqref{4.32}--\eqref{4.37} into \eqref{4.31}, we arrive at 
\begin{align}
&\|a^\varepsilon\|_{\widetilde L_t^\infty(\dot B_{2,1}^{\frac{d}{2}})}^{m}+
\|u^\varepsilon\|_{\widetilde L_t^\infty(\dot B_{2,1}^{\frac{d}{2}-1})}^m
+\varepsilon^2\bigg(\|a^\varepsilon\|_{ L^1_t(\dot B_{2,1}^{\frac{d}{2}})}^m+
\|u^\varepsilon\|_{ L^1_t(\dot B_{2,1}^{\frac{d}{2}-1})}^m\bigg)
\lesssim \|a^\varepsilon_0\|_{\dot B_{2,1}^{\frac{d}{2}}}^m
+\|u^\varepsilon_0\|_{\dot B_{2,1}^{\frac{d}{2}-1}}^m+\mathcal{X}_\varepsilon^2(t).\nonumber
\end{align}
This completes the proof of this lemma.
\end{proof}

Since the dissipative structure of  $u^\varepsilon$ is preserved,  the velocity field should admit an estimate for $\|u^\varepsilon\|^m_{ L^1_t(\dot B_{2,1}^{\frac{d}{2}+1})}$. Motivated by this observation, we now turn our attention to deriving such an estimate.

\begin{lem}\label{L4.4}
Let $(a^\varepsilon, u^\varepsilon)$ be a strong solution to the Cauchy problem \eqref{1.3}--\eqref{1.4} on $[0,T]\times \mathbb R^d$. Then, under the condition $\eqref{NJKG4.1}$, we have
 \begin{align}
\|u^\varepsilon\|_{\widetilde L_t^\infty(\dot B_{2,1}^{\frac{d}{2}-1})}^m
+\|u^\varepsilon\|_{ L^1_t(\dot B_{2,1}^{\frac{d}{2}+1})}^m
\lesssim \|u^\varepsilon_0\|_{\dot B_{2,1}^{\frac{d}{2}-1}}^m+\varepsilon^2\|a^\varepsilon\|_{L^1_t(\dot B_{2,1}^{\frac{d}{2}})}^m+\mathcal{X}_{\varepsilon}^2(t).\nonumber
\end{align}
\end{lem}
\begin{proof}
Based on the optimal regularity estimate in Lemma \ref{Lemma2.6}, we directly get
\begin{align}\label{4.38}
&\|u^\varepsilon\|_{\widetilde L_t^\infty(\dot B_{2,1}^{\frac{d}{2}-1})}^m
+
\|u^\varepsilon\|_{ L^1_t(\dot B_{2,1}^{\frac{d}{2}+1})}^m\nonumber\\
\lesssim&\|u^\varepsilon_0\|_{\dot B_{2,1}^{\frac{d}{2}-1}}^m+
\varepsilon^2\|a^\varepsilon\|_{L^1_t(\dot B_{2,1}^{\frac{d}{2}})}^m
+\|u^\varepsilon\cdot\nabla u^\varepsilon\|_{L^1_t(\dot B_{2,1}^{\frac{d}{2}-1})}^m
+\|f(a^{\varepsilon})\mathcal{A} u^{\varepsilon}\|^m_{L^1_t(\dot B_{2,1}^{\frac{d}{2}-1})}.
\end{align}
Using Lemma \ref{A.3}, we easily obtain 
\begin{align}\label{4.39}
\|u^\varepsilon\cdot\nabla u^\varepsilon\|^m_{L^1_t(\dot B_{2,1}^{\frac{d}{2}-1})}\lesssim\|u^\varepsilon\|_{\widetilde L^\infty_t(\dot B_{2,1}^{\frac{d}{2}-1})}\|u^\varepsilon\|_{L^1_t(\dot B_{2,1}^{\frac{d}{2}+1})}\lesssim \mathcal{X}^2_\varepsilon(t).
\end{align}
Thus, combining \eqref{4.32} and \eqref{4.39} with \eqref{4.38},
we arrive at 
\begin{align}
\|u^\varepsilon\|_{\widetilde L_t^\infty(\dot B_{2,1}^{\frac{d}{2}-1})}^m
+\|u^\varepsilon\|_{ L^1_t(\dot B_{2,1}^{\frac{d}{2}+1})}^m
\lesssim \|u^\varepsilon_0\|_{\dot B_{2,1}^{\frac{d}{2}-1}}^m+\varepsilon^2\|a^\varepsilon\|_{L^1_t(\dot B_{2,1}^{\frac{d}{2}})}^m+\mathcal{X}_{\varepsilon}^2(t).\nonumber
\end{align}
This completes the proof of this lemma.
\end{proof}

Therefore, for the estimate of the medium-frequency, it follows from Lemmas \ref{L4.3} and \ref{L4.4} that
\begin{align}\label{4.39.1}
&\|a^\varepsilon\|_{\widetilde L_t^\infty(\dot B_{2,1}^{\frac{d}{2}})}^m+
\|u^\varepsilon\|_{\widetilde L_t^\infty(\dot B_{2,1}^{\frac{d}{2}-1})}^m
+\varepsilon^2\|a^\varepsilon\|_{ L^1_t(\dot B_{2,1}^{\frac{d}{2}})}^m+
\|u^\varepsilon\|_{ L^1_t(\dot B_{2,1}^{\frac{d}{2}+1})}^m
\lesssim
\|a^\varepsilon_0\|_{\dot B_{2,1}^{\frac{d}{2}}}^m
+\|u^\varepsilon_0\|_{\dot B_{2,1}^{\frac{d}{2}-1}}^m+\mathcal{X}_{\varepsilon}^2(t).
\end{align}

\subsubsection{High-frequency regime.} Inspired by the high-frequency analysis performed for the isentropic compressible Navier--Stokes equations in \cite{Danchin-2018}, 
we introduce the effective velocity 
$\omega^\varepsilon=\nabla(-\Delta)^{-1}(\frac{\varepsilon^2a^\varepsilon}{2\mu+\lambda}
-\dv u^\varepsilon)$, 
and rewrite \eqref{1.3} into the following two systems
\begin{align}\label{4.40}
\partial_{t}\mathbb{P}u^\varepsilon - \mu \Delta \mathbb{P}u^\varepsilon = \mathbb{P}g
\end{align}
where $g:=-u^\varepsilon\cdot\nabla u^\varepsilon
-f(a^\varepsilon)\mathcal{A} u^\varepsilon$, and 
\begin{equation}\label{4.41}
\left\{
\begin{aligned}
&\partial_ta^\varepsilon+
\dv (a^\varepsilon u^\varepsilon)+\frac{\varepsilon^2 a^\varepsilon}{2\mu+\lambda}=-\dv \omega^\varepsilon,\\
&\partial_t \omega^\varepsilon-(2\mu +\lambda )\Delta
\omega^\varepsilon
=\nabla(-\Delta)^{-1}\bigg(-\varepsilon^2\frac{\dv(a^\varepsilon u^\varepsilon)}{2\mu+\lambda}-\dv g\bigg)+\frac{\varepsilon^2 \omega^{\varepsilon}}{2\mu+\lambda}-\frac{\varepsilon^4\nabla(-\Delta)^{-1}a^\varepsilon}{(2\mu+\lambda)^2}.
\end{aligned}
\right.
\end{equation}
It is easy to check that
\begin{align}\label{4.42.1}
&\|u^\varepsilon\|^h_{\widetilde L^\infty_t(\dot B^{\frac{d}{2}-1}_{2,1})}+\|u^\varepsilon\|^h_{L^1_t(\dot B^{\frac{d}{2}+1}_{2,1})} \nonumber\\
\leq&\,\|(\mathbb{P}u^\varepsilon,\omega^\varepsilon)\|^h_{\widetilde L^\infty_t(\dot B^{\frac{d}{2}-1}_{2,1})}+\|(\mathbb{P}u^\varepsilon,\omega^\varepsilon)\|^h_{L^1_t(\dot B^{\frac{d}{2}+1}_{2,1})} + C \left( \|a^\varepsilon\|^h_{\widetilde L_t^\infty(\dot{B}_{2,1}^{\frac{d}{2}})} + \varepsilon^2\|a^\varepsilon\|_{L_t^1(\dot{B}_{2,1}^{\frac{d}{2}})}^{h} \right),
\end{align}
Regarding the incompressible part $\mathbb{P} u^\varepsilon$, taking the same arguments as Lemma \ref{L4.4}, we directly get 
\begin{align}\label{4.43.1}
\|\mathbb{P}u^\varepsilon\|_{\widetilde L_t^\infty(\dot B_{2,1}^{\frac{d}{2}-1})}^h
+\|\mathbb{P}u^\varepsilon\|_{ L^1_t(\dot B_{2,1}^{\frac{d}{2}+1})}^h
\lesssim \|u^\varepsilon_0\|_{\dot B_{2,1}^{\frac{d}{2}-1}}^h+\mathcal{X}_{\varepsilon}^2(t).
\end{align}

We next pay attention to the estimates of $a^\varepsilon$ and  $\omega^\varepsilon$.
\begin{lem}\label{L4.5}
Let $(a^\varepsilon, u^\varepsilon)$ be a strong solution to the Cauchy problem \eqref{1.3}--\eqref{1.4} on $[0,T]\times \mathbb R^d$. Then, under the condition $\eqref{NJKG4.1}$, we have
\begin{align}
\|\omega^\varepsilon\|^h_{\widetilde L^\infty_t(\dot B^{\frac{d}{2}-1}_{2,1})}
+\|\omega^\varepsilon\|^h_{L^1_t(\dot B^{\frac{d}{2}+1}_{2,1})}
+\|a^\varepsilon\|^h_{\widetilde L^\infty_t(\dot B^{\frac{d}{2}}_{2,1})}
+\varepsilon^2\|a^\varepsilon\|^h_{L^1_t(\dot B^{\frac{d}{2}}_{2,1})}
\lesssim \|a^\varepsilon_0\|^h_{\dot B^{\frac{d}{2}}_{2,1}} +\|\omega^\varepsilon_0\|^h_{\dot B^{\frac{d}{2}-1}_{2,1}}
+\mathcal{X}^2_\varepsilon(t).\nonumber
\end{align}
\end{lem}

\begin{proof}
Applying the operator $\dot\Delta_k$ to $\eqref{4.41}_1$, we have 
\begin{align}\label{4.42}
&\partial_t\dot\Delta_ka^\varepsilon+u^\varepsilon\cdot\nabla 
\dot\Delta_ka^\varepsilon+
\frac{\varepsilon^2\dot\Delta_ka^\varepsilon}{2\mu+\lambda}
=-\dot\Delta_k (a^\varepsilon\dv u^\varepsilon)- \dv \dot\Delta_{k}\omega^\varepsilon+\dot R_k,
\end{align}
where $\dot R_k=[u^\varepsilon\cdot\nabla,\dot\Delta_k]a^\varepsilon$.
By virtue of Lemma \ref{LA.5}, it holds that
\begin{align}
\|\dot R_k\|_{L^2}\leq Cc_k2^{-k\frac{d}{2}}
\|\nabla u^\varepsilon\|_{\dot B^{\frac{d}{2}}_{2,1}}\|a^\varepsilon\|_{\dot B^{\frac{d}{2}}_{2,1}},
\end{align}
where $\sum_{k\in\mathbb{Z}}c_k=1$.
 Multiplying \eqref{4.42} by $\dot\Delta_ka^\varepsilon$ and integrating by parts over $\mathbb{R}^d$, we have
\begin{align}\label{4.43}
&\|\dot\Delta_ka^\varepsilon\|_{L^2}+\varepsilon^2\|\dot\Delta_ka^\varepsilon\|_{L^1_t(L^2)}\nonumber\\
\lesssim&\, \|\dot\Delta_ka^\varepsilon_0\|_{L^2}
+\int^t_0\Big(\|\dot\Delta_k(a^\varepsilon \dv u^\varepsilon)\|_{L^2}+\|\dot R_k\|_{L^2}+\|\dot\Delta_k \dv \omega^\varepsilon\|_{L^2}+\|\dv u^\varepsilon\|_{L^\infty}
\|\dot\Delta_ka^\varepsilon\|_{L^2}
\Big){\rm d}s.
\end{align}
Furthermore, multiplying \eqref{4.43} by $2^{k\frac{d}{2}}$, and then taking the summation over all $k\geq J^\varepsilon_1$, we deduce that
\begin{align}\label{4.45}
&\|a^\varepsilon\|^h_{\widetilde L^\infty_t(\dot B^{\frac{d}{2}}_{2,1})}
+\varepsilon^2\|a^\varepsilon\|^h_{L^1_t(\dot B^{\frac{d}{2}}_{2,1})}\nonumber\\
\lesssim&\, \|a^\varepsilon_0\|^h_{\dot B^{\frac{d}{2}}_{2,1}}
+\int^t_0\Big(\|a^\varepsilon \dv u^\varepsilon\|^h_{\dot B^{\frac{d}{2}}_{2,1}}+\|\nabla u^\varepsilon\|_{\dot B^{\frac{d}{2}}_{2,1}}\|a^\varepsilon\|_{\dot B^{\frac{d}{2}}_{2,1}}+\| \dv \omega^\varepsilon\|^h_{\dot B^{\frac{d}{2}}_{2,1}}+\|\dv u^\varepsilon\|_{L^\infty}
\|a^\varepsilon\|^h_{\dot B^{\frac{d}{2}}_{2,1}}
\Big){\rm d}s.
\end{align}
For the nonlinear terms in \eqref{4.45}, making use of Lemmas \ref{LA.2} and \ref{LA.3}, respectively, we get
\begin{align}\label{4.46}
\|\dv u^\varepsilon\|_{L^\infty}\lesssim\|\nabla u^\varepsilon\|_{\dot B^{\frac{d}{2}}_{2,1}},
\end{align}
and 
\begin{align}\label{4.47}
\|a^\varepsilon \dv u^\varepsilon\|^h_{\dot B^{\frac{d}{2}}_{2,1}}\lesssim \|a^\varepsilon\|_{\dot B^{\frac{d}{2}}_{2,1}}\|\nabla u^\varepsilon\|_{\dot B^{\frac{d}{2}}_{2,1}}.
\end{align}
Thus, it follows from \eqref{4.45}--\eqref{4.47} that
\begin{align}\label{4.48}
\|a^\varepsilon\|^h_{\widetilde L^\infty_t(\dot B^{\frac{d}{2}}_{2,1})}
+\varepsilon^2\|a^\varepsilon\|^h_{L^1_t(\dot B^{\frac{d}{2}}_{2,1})}
\lesssim \|a^\varepsilon_0\|^h_{\dot B^{\frac{d}{2}}_{2,1}}
+\int^t_0\Big(\|\nabla u^\varepsilon\|_{\dot B^{\frac{d}{2}}_{2,1}}\|a^\varepsilon\|_{\dot B^{\frac{d}{2}}_{2,1}}+\| \dv \omega^\varepsilon\|^h_{\dot B^{\frac{d}{2}}_{2,1}}
\Big){\rm d}s.
\end{align}

In addition, using Lemma \ref{Lemma2.6}, we have 
\begin{align}\label{4.49}
&\|\omega^\varepsilon\|^h_{\widetilde L^\infty_t(\dot B^{\frac{d}{2}-1}_{2,1})}+\|\omega^\varepsilon\|^h_{L^1_t(\dot B^{\frac{d}{2}+1}_{2,1})}
\nonumber\\
\lesssim \,& \|\omega^\varepsilon_0\|^h_{\dot B^{\frac{d}{2}-1}_{2,1}}
+\Big\|\varepsilon^2\frac{\dv(a^\varepsilon u^\varepsilon)}{2\mu+\lambda}+\dv g\Big\|^h_{L^1_t(\dot B^{\frac{d}{2}-2}_{2,1})}+
\Big\|\frac{\varepsilon^2 \omega^\varepsilon}{2\mu+\lambda}-\frac{\varepsilon^4\nabla(-\Delta)^{-1}a^\varepsilon}{(2\mu+\lambda)^2}\Big\|^h_{L^1_t(\dot B^{\frac{d}{2}-1}_{2,1})}.
\end{align}
Multiplying \eqref{4.49} by a large constant $M$, and adding the resulting inequality to \eqref{4.48}, we arrive at the following inequality
\begin{align}\label{4.50}
&M\|\omega^\varepsilon\|^h_{\widetilde L^\infty_t(\dot B^{\frac{d}{2}-1}_{2,1})}
+\frac{M}{2}\|\omega^\varepsilon\|^h_{L^1_t(\dot B^{\frac{d}{2}+1}_{2,1})}
+\|a^\varepsilon\|^h_{\widetilde L^\infty_t(\dot B^{\frac{d}{2}}_{2,1})}
+\varepsilon^2\|a^\varepsilon\|^h_{L^1_t(\dot B^{\frac{d}{2}}_{2,1})}\nonumber\\
\lesssim &\|a^\varepsilon_0\|^h_{\dot B^{\frac{d}{2}}_{2,1}} +M\|\omega^\varepsilon_0\|^h_{\dot B^{\frac{d}{2}-1}_{2,1}}
+M\Big\|\varepsilon^2\frac{\dv(a^\varepsilon u^\varepsilon)}{2\mu+\lambda}+\dv g\Big\|^h_{L^1_t(\dot B^{\frac{d}{2}-2}_{2,1})}\nonumber\\
&+M\Big\|\frac{\varepsilon^2 \omega^\varepsilon}{2\mu+\lambda}-\frac{\varepsilon^4\nabla(-\Delta)^{-1}a^\varepsilon}{(2\mu+\lambda)^2}\Big\|^h_{L^1_t(\dot B^{\frac{d}{2}-1}_{2,1})}
+\int^t_0\|\nabla u^\varepsilon\|_{\dot B^{\frac{d}{2}}_{2,1}}\|a^\varepsilon\|_{\dot B^{\frac{d}{2}}_{2,1}}{\rm d}s
.
\end{align}
Making use of the inequality \eqref{1.13.1}$_2$, 
we deduce that 
\begin{align}\label{4.51}
\Big\|\frac{\varepsilon^2 \omega^\varepsilon}{2\mu+\lambda}-\frac{\varepsilon^4\nabla(-\Delta)^{-1}a^\varepsilon}{(2\mu+\lambda)^2}\Big\|^h_{L^1_t(\dot B^{\frac{d}{2}-1}_{2,1})}
\lesssim \, \varepsilon^{2\delta_c}
\|\omega^\varepsilon\|^h_{L^1_t(\dot B^{\frac{d}{2}+1}_{2,1})}
+\varepsilon^{2+2\delta_c}
\|a^\varepsilon\|^h_{L^1_t(\dot B^{\frac{d}{2}}_{2,1})}
\end{align}
As a result, from \eqref{4.50} and \eqref{4.51}, we derive 
\begin{align}\label{4.52}
&\|\omega^\varepsilon\|^h_{\widetilde L^\infty_t(\dot B^{\frac{d}{2}-1}_{2,1})}
+\|\omega^\varepsilon\|^h_{L^1_t(\dot B^{\frac{d}{2}+1}_{2,1})}
+\|a^\varepsilon\|^h_{\widetilde L^\infty_t(\dot B^{\frac{d}{2}}_{2,1})}
+\varepsilon^2\|a^\varepsilon\|^h_{L^1_t(\dot B^{\frac{d}{2}}_{2,1})}\nonumber\\
\lesssim &\|a^\varepsilon_0\|^h_{\dot B^{\frac{d}{2}}_{2,1}} +\|\omega^\varepsilon_0\|^h_{\dot B^{\frac{d}{2}-1}_{2,1}}
+\varepsilon^2\|a^\varepsilon u^\varepsilon\|^h_{L^1_t(\dot B^{\frac{d}{2}-1}_{2,1})}+\|\dv g\|^h_{L^1_t(\dot B^{\frac{d}{2}-2}_{2,1})}+\int^t_0\|\nabla u^\varepsilon\|_{\dot B^{\frac{d}{2}}_{2,1}}\|a^\varepsilon\|_{\dot B^{\frac{d}{2}}_{2,1}}{\rm d}s
.
\end{align}
By Lemma \ref{LA.3}, \eqref{1.13.2}$_1$ and \eqref{1.13.1}$_2$ directly, we get
\begin{align}\label{4.53}
\varepsilon^2\|a^\varepsilon u^\varepsilon\|^h_{L^1_t(\dot B^{\frac{d}{2}-1}_{2,1})}\lesssim&\, \varepsilon\|a^\varepsilon u^\varepsilon\|^h_{L^1_t(\dot B^{\frac{d}{2}}_{2,1})}
\lesssim \varepsilon\|a^\varepsilon\|_{\widetilde L^\infty_t(\dot B^{\frac{d}{2}-1}_{2,1})}
\|u^\varepsilon\|_{L^1_t(\dot B^{\frac{d}{2}+1}_{2,1})}\nonumber\\
\lesssim&\, (\varepsilon\|a^\varepsilon \|^\ell_{\widetilde L^\infty_t(\dot B^{\frac{d}{2}-1}_{2,1})}+\|a^\varepsilon \|^m_{\widetilde L^\infty_t(\dot B^{\frac{d}{2}}_{2,1})}+\|a^\varepsilon \|^h_{\widetilde L^\infty_t(\dot B^{\frac{d}{2}}_{2,1})})\|u^\varepsilon\|_{L^1_t(\dot B^{\frac{d}{2}+1}_{2,1})}
\lesssim \mathcal{X}^2_\varepsilon(t).
\end{align}
Arguing analogously to \eqref{4.32} and \eqref{4.39}, we have 
\begin{align}\label{4.54}
\|\dv g\|^h_{L^1_t(\dot B^{\frac{d}{2}-2}_{2,1})}\lesssim \mathcal{X}^2_\varepsilon(t).
\end{align}
Consequently, putting \eqref{4.53} and \eqref{4.54} into \eqref{4.52} yields 
\begin{align*}
\|\omega^\varepsilon\|^h_{\widetilde L^\infty_t(\dot B^{\frac{d}{2}-1}_{2,1})}
+\|\omega^\varepsilon\|^h_{L^1_t(\dot B^{\frac{d}{2}+1}_{2,1})}
+\|a^\varepsilon\|^h_{\widetilde L^\infty_t(\dot B^{\frac{d}{2}}_{2,1})}
+\varepsilon^2\|a^\varepsilon\|^h_{L^1_t(\dot B^{\frac{d}{2}}_{2,1})}
\lesssim&\, \|a^\varepsilon_0\|^h_{\dot B^{\frac{d}{2}}_{2,1}} +\|\omega^\varepsilon_0\|^h_{\dot B^{\frac{d}{2}-1}_{2,1}}
+\mathcal{X}^2_\varepsilon(t)\\
\lesssim &\, \|a^\varepsilon_0\|^h_{\dot B^{\frac{d}{2}}_{2,1}} +\|u^\varepsilon_0\|^h_{\dot B^{\frac{d}{2}-1}_{2,1}}
+\mathcal{X}^2_\varepsilon(t).
\end{align*}
This completes the proof of this lemma.
\end{proof}

Combining \eqref{4.42.1}, \eqref{4.43.1} and Lemma \ref{L4.5}, we derive that 
\begin{align}\label{4.57}
\|u^\varepsilon\|^h_{\widetilde L^\infty_t(\dot B^{\frac{d}{2}-1}_{2,1})}
+\|u^\varepsilon\|^h_{L^1_t(\dot B^{\frac{d}{2}+1}_{2,1})}
+\|a^\varepsilon\|^h_{\widetilde L^\infty_t(\dot B^{\frac{d}{2}}_{2,1})}
+\varepsilon^2\|a^\varepsilon\|^h_{L^1_t(\dot B^{\frac{d}{2}}_{2,1})}
\lesssim \|a^\varepsilon_0\|^h_{\dot B^{\frac{d}{2}}_{2,1}} +\|u^\varepsilon_0\|^h_{\dot B^{\frac{d}{2}-1}_{2,1}}
+\mathcal{X}^2_\varepsilon(t).
\end{align}
Thus, we complete the estimates of $(a^\varepsilon,u^\varepsilon)$ in the high-frequency regime.

\subsubsection{Proof of Proposition \ref{prop3.1}}
It follows from Lemma \ref{L4.2}, \eqref{4.39.1} and \eqref{4.57} that there exists a constant $C_0$, independent of $T$ and $\varepsilon$, such that for all $t\in (0,T)$,
\begin{align}\label{4.58}
\mathcal{X}_\varepsilon(t)\leq C_0(\mathcal{X}_{\varepsilon,0}
+\mathcal{X}^2_\varepsilon(t)).
\end{align}
Obviously, for a sufficiently small constant $\eta_0$, when $\mathcal{X}_\varepsilon(t)\leq \eta_0$,
the inequality \eqref{4.58} ensures that 
\begin{align}\label{4.60}
\mathcal{X}_\varepsilon(t)\leq 2C_0\mathcal{X}_{\varepsilon,0}.
\end{align}

\subsection{Proof of Theorem \ref{Th1}}
Let \(T_1\in(0,+\infty]\) be the maximal lifespan of the local strong solution given by Theorem \ref{Th0}. We shall prove that \(T_1 = +\infty\). 
Based on the embedding inequality  \(\dot B^{\frac{d}{2}}_{2,1}\hookrightarrow L^\infty\), we have
\begin{align*}
\|a_0^\varepsilon\|_{L^\infty}\le C_{1}\|a_0^\varepsilon\|_{\dot B^{\frac d2}_{2,1}}. 
\end{align*}
Choose a smaller constant \(\eta_0>0\) in Proposition \ref{prop3.1}, if necessary, such that
\begin{align*}C_{1}\eta_0\leq \frac{1}{2}.    
\end{align*} Then, select \(\delta_0>0\) so that \begin{align*}C_{1}\delta_0\leq \frac{1}{2},\qquad2C_0\delta_0\leq \frac{\eta_0}{2}.    
\end{align*} 
Subsequently, it can be shown that \begin{align*}
\|a_0^\varepsilon\|_{L^\infty}\leq C_{1}\|a_0^\varepsilon\|_{\dot B^{\frac{d}{2}}_{2,1}}\leq C_{1}\mathcal{X}_{\varepsilon,0} \leq C_1 \delta_0 \leq \frac{1}{2},    
\end{align*}  
and thus \(1 + a_0^\varepsilon\geq \frac{1}{2}\).

Define
\begin{align*}
T^{*}:=\sup\big\{T\in(0,T_1):\ \mathcal X_\varepsilon(t)\leq \eta_0 \quad\text{for all }\quad 0\leq t\leq T\big\}.    
\end{align*}
By using the continuity argument of the local solution, we can conclude that \(T^{*}>0\). 
For any \(0 < T < T^{*}\), we have \(\mathcal{X}_\varepsilon(T)\leq \eta_0\). Therefore,
\begin{align*}
\|a^\varepsilon\|_{L^\infty([0,T];L^\infty)}
\le
C_{1}
\|a^\varepsilon\|_{\widetilde L_T^\infty
(\dot B^{\frac d2}_{2,1})}
\leq
C_{1}\mathcal X_\varepsilon(T)
\leq \frac12.    
\end{align*}
Hence, the assumption \eqref{NJKG4.1} of Proposition \ref{prop3.1} is satisfied on \([0,T]\). By applying Proposition \ref{prop3.1}, we obtain the following inequality:
\begin{align*}
\mathcal X_\varepsilon(t)\leq 2C_0\mathcal X_{\varepsilon,0}\leq 2C_0\delta_0\leq\frac{\eta_0}{2},\qquad 0\leq t\leq T,
\end{align*}
Through  \eqref{NJKG3.1} and the standard continuity argument, we can conclude that \(T^{*}=T_1\). Consequently, we have
\begin{align*}
\mathcal X_\varepsilon(t)\leq 2C_0\mathcal X_{\varepsilon,0},\qquad 0\leq t < T_1.
\end{align*}
If \(T_1<+\infty\), then the  estimate contradicts the continuation criterion in Theorem \ref{Th0}. Therefore, \(T_1 = +\infty\). As a result, we obtain that \eqref{NJKG3.3} holds for all $t\geq 0$, which yields \eqref{NJK1.14}. Therefore, we complete the proof of Theorem \ref{Th1}.

\section{High Mach number limit and error estimates}

In this section, we justify the high Mach number limit for the scaled compressible Navier--Stokes system \eqref{1.3}--\eqref{1.4}. The content consists of two parts. In Subsection \ref{AAAA}, we utilize the uniform estimates in Theorem \ref{Th1} to construct the global solution of the pressureless Navier--Stokes system \eqref{1.5}--\eqref{1.5.1} through compactness. In Subsection \ref{BBBB}, we will further establish the finite-time error estimate.
\subsection{Proof of Theorem \ref{Th2}}\label{AAAA}
Let $(a_0, u_0)$ satisfy \eqref{HM2.1} and $\delta_0>0$ be given by \eqref{NJK1.13}. Thanks to the density argument, we select a sequence $\{(a_0^\varepsilon,{u}_0^\varepsilon)\}_{0<\varepsilon<\varepsilon_0}$ (for some $\varepsilon_0>0$) converging to $(a_0,{u}_0)$ in 
$\big(\dot B_{2,1}^{\frac{d}{2}-1}\cap \dot B_{2,1}^{\frac{d}{2}}\big)\times\dot B_{2,1}^{\frac{d}{2}-1}$.
In particular, we take $\varepsilon_0$ small enough such that for all $0<\varepsilon<\varepsilon_0$,
$$
 {\|a_0^\varepsilon - a_0\|_{\dot B_{2,1}^{\frac{d}{2}-1}}}+
\|a_0^\varepsilon - a_0\|_{\dot B_{2,1}^{\frac{d}{2}}} + \|{u}_0^\varepsilon - {u}_0\|_{\dot B_{2,1}^{\frac{d}{2}-1}} < \frac{\delta_0}{2}.
$$
This together with \eqref{HM2.1}, implies that
\begin{align*}
\mathcal X_{\varepsilon,0}
:=\varepsilon\|a_0^\varepsilon\|^\ell_{\dot B^{\frac d2-1}_{2,1}}
+\|a_0^\varepsilon\|^{m+h}_{\dot B^{\frac d2}_{2,1}}
+\| u_0^\varepsilon\|_{\dot B^{\frac d2-1}_{2,1}}
<\frac{\delta_0}{2}+\mathcal{X}_0\leq \frac{\delta_0}{2}+\delta_1 \leq \delta_0,
\end{align*}
provided that we set $\delta_1 \leq  \frac{\delta_0}{2}$.
According to Theorem \ref{Th1}, for any $0 < \varepsilon < \varepsilon_0$, we construct a unique global solution $(a^{\varepsilon}, u^\varepsilon)$ to the Cauchy problem \eqref{1.3}--\eqref{1.4}, which {satisfies the uniform-in-$\varepsilon$ global estimate \eqref{NJK1.14}} 
As a direct consequence of \eqref{NJK1.14}, taking the limit as $\varepsilon\rightarrow 0$, there exists a limit $(a,{u})$.
Up to a subsequence, we have
\begin{equation*}
\left\{
\begin{aligned}
&a^\varepsilon\rightharpoonup  a\,
\qquad\text{weakly-* in}\qquad
L_{\rm loc}^\infty(\mathbb R^+;\dot B^{\frac d2}_{2,1}),\\
& u^\varepsilon\rightharpoonup  u
\qquad\text{weakly-* in}\qquad
L_{\rm loc}^\infty( \mathbb R^+;\dot B^{\frac d2-1}_{2,1}),\\
& u^\varepsilon\rightharpoonup u
\qquad\hbox{weakly in}\qquad\,\,\,\,\,
L_{\rm loc}^1(\mathbb R+;\dot B^{\frac d2+1}_{2,1}).
\end{aligned}
\right.
\end{equation*}

To prove \eqref{NZGG1.18}$_4$--\eqref{NZGG1.18}$_5$, we need to analyze the temporal derivatives of $(a^\varepsilon, u^\varepsilon)$. 
Utilizing \eqref{1.3} and the estimate \eqref{NJK1.14}, we have  that for any $0\leq t\leq T$, 
\begin{align*}
 \| \partial_{t} a^{\varepsilon}\|_{L_t^2(\dot B^{\frac{d}{2}-1}_{2,1})}\lesssim&\,  \|u^\varepsilon\|_{L_t^2(\dot B^{\frac{d}{2}}_{2,1})}+\|a^\varepsilon u^\varepsilon\|_{L_t^2(\dot B^{\frac{d}{2}}_{2,1})} \nonumber\\
 \lesssim&\, \bigg(1+\|a^\varepsilon\|_{\widetilde L_t^\infty(\dot B^{\frac{d}{2}}_{2,1})}\bigg)\|u^\varepsilon\|_{L_t^1(\dot B^{\frac{d}{2}+1}_{2,1})}^\frac{1 }{2} \|u^\varepsilon\|_{\widetilde L_t^{\infty}(\dot B^{\frac{d}{2}-1}_{2,1})}^\frac{1 }{2}\leq C_{T},
\end{align*}
and
\begin{align*}
\|\partial_{t}u^\varepsilon\|_{L_t^1(\dot B_{2,1}^{\frac{d}{2}-1})}\lesssim &\,\bigg(1+\|u^\varepsilon\|_{\widetilde L_t^\infty(\dot B^{\frac{d}{2}-1}_{2,1})}+\|f(a^\varepsilon)\|_{\widetilde L_t^\infty(\dot B_{2,1}^{\frac{d}{2}})} \bigg)    \|u^\varepsilon\|_{L_t^1(\dot B^{\frac{d}{2}+1}_{2,1})}+\varepsilon^2\|a^\varepsilon\|_{L_t^1(\dot B^{\frac{d}{2}}_{2,1})} \nonumber\\
\lesssim&\, \bigg(1+\|u^\varepsilon\|_{\widetilde L_t^\infty(\dot B^{\frac{d}{2}-1}_{2,1})}+\|a^\varepsilon\|_{\widetilde L_t^\infty(\dot B_{2,1}^{\frac{d}{2}})} \bigg)    \|u^\varepsilon\|_{L_t^1(\dot B^{\frac{d}{2}+1}_{2,1})}+{\varepsilon^2
T\|a^\varepsilon\|_{\widetilde L_t^{\infty}( \dot B^{\frac d2}_{2,1})} }
\leq  C_{T},
\end{align*}
for the uniform constant $C_{T}>0$, where $C_{T}$ only depends on $T$. Here, we used the smallness of $\varepsilon$. By the local compact embedding
\begin{align*}
 \dot B^{\frac d2}_{2,1}(\mathbb R^d)
\hookrightarrow H^{\frac d2}_{\rm loc}(\mathbb R^d)
\stackrel{c}{\hookrightarrow} L^2_{\rm loc}(\mathbb R^d),   
\end{align*}
 the Aubin--Lions lemma  and  Cantor's diagonal argument, there exists a limit pair \((a,{u})\) such that, up to a subsequence,
\begin{equation*}
\left\{
\begin{aligned}
&a^\varepsilon\to a\,
\qquad\text{strongly in}\qquad\,\,
\mathcal C_{\rm loc}(\mathbb R^+;L^2_{\rm loc} ),\\
& {u}^\varepsilon\to u
\qquad\text{strongly in}\qquad\,
L_{\rm loc}^2(\mathbb R^+;L^2_{\rm loc} ).
\end{aligned}
\right.
\end{equation*}

Consequently, it is straightforward to verify that the pair $(a, {u})$  is a solution to the pressureless Navier--Stokes system \eqref{1.5}--\eqref{1.5.1}.
Since $a \in L^\infty(\mathbb R^+; \dot B^{\frac d2}_{2,1})$, the condition $\rho=1+a>0$ is preserved due to smallness.
Moreover, the inequality \eqref{HM2.2} can be derived from Fatou’s lemma. Thus,  \eqref{NZGG1.18}--\eqref{HM2.2} are verified. Hence, the proof of Theorem \ref{Th2} is completed.
 \hfill $\Box$

\subsection{Proof of Theorem \ref{Th3}}\label{BBBB}
We now pay attention to establishing the  convergence rate of the high Mach number limit in any {\it finite-time} interval.
We define the differences  as follows:
\begin{align*}
\delta a:=a^\varepsilon-a,
\qquad
\delta{u}:= u^\varepsilon- u,    
\end{align*}
and define the error functional $\delta\mathcal{X}(t)$ by
\begin{align*}
\delta\mathcal X(t)
:=
\|\delta a\|_{\widetilde L_t^\infty(\dot B^{\frac d2-1}_{2,1})}
+
\|\delta u\|_{\widetilde L_t^\infty(\dot B^{\frac d2-2}_{2,1})}+\|\partial_{t}\delta u\|_{L_t^1(\dot B^{\frac{d}{2}-2}_{2,1})}
+
\|\delta u\|_{L_t^1(\dot B^{\frac d2}_{2,1})}.    
\end{align*}
Subtracting \eqref{1.5} from \eqref{1.3} gives
\begin{equation}\label{G5.1}
\left\{
\begin{aligned}
&\partial_t\delta a+{\rm div}\,\delta u
=-\delta a {\rm div}u -a^\varepsilon{\rm div}\, \delta u-u \cdot\nabla \delta a-\delta u\cdot\nabla   a^\varepsilon
\\
&\partial_t\delta u-\mathcal A\delta u+\varepsilon^2\nabla a^\varepsilon
=-\delta u\cdot\nabla u^\varepsilon-u\cdot\nabla\delta u
-(f(a^\varepsilon)-f(a))\mathcal A u 
-f(a^\varepsilon)\mathcal A\delta u,
\\
&(\delta a,\delta u)|_{t=0}
=
(a_0^\varepsilon-a_0, u_0^\varepsilon- u_0),
\end{aligned}
\right.
\end{equation}
where $\mathcal{A}$ and $f$ have the same definitions as in Sections 1 and 3.
It follows from \eqref{G5.1}$_2$ and Lemma \ref{Lemma2.6} that
\begin{align}\label{G5.2}
&\|\delta u\|_{\widetilde L_t^\infty(\dot B^{\frac d2-2}_{2,1})}
+
\|\delta u\|_{L_t^1(\dot B^{\frac d2}_{2,1})}+\|\partial_{t}\delta u\|_{L_t^1(\dot B^{\frac{d}{2}-2}_{2,1})}\nonumber\\
\lesssim&\,
\|u_0^\varepsilon-u_0\|_{\dot B^{\frac d2-2}_{2,1}}
+
\|\delta u\cdot\nabla u^\varepsilon\|_{L_t^1
(\dot B^{\frac d2-2}_{2,1})}
+
\| u\cdot\nabla\delta u\|_{L_t^1
(\dot B^{\frac d2-2}_{2,1})}+
\varepsilon^2
\|\nabla a^\varepsilon\|_{L_t^1(\dot B^{\frac d2-2}_{2,1})}
\nonumber\\
& +
\|f(a^\varepsilon)\mathcal A\delta u\|_{L_t^1
(\dot B^{\frac d2-2}_{2,1})}
+
\|(f(a^\varepsilon)-f(a))\mathcal A u\|_{L_t^1
(\dot B^{\frac d2-2}_{2,1})} \nonumber\\
\lesssim&\, \| u_0^\varepsilon- u_0\|_{\dot B^{\frac d2-2}_{2,1}}+{\|\delta u\|_{\widetilde L_t^\infty(\dot B^{\frac d2-2}_{2,1})}
\| u^\varepsilon\|_{L_t^1(\dot B^{\frac d2+1}_{2,1})}}+\|u\|_{\widetilde L_t^\infty(\dot B^{\frac d2-1}_{2,1})}
\|\delta u\|_{L_t^1(\dot B^{\frac d2}_{2,1})}\nonumber\\
&+\|a^\varepsilon\|_{\widetilde L_t^\infty(\dot B^{\frac d2}_{2,1})}
\|\delta u\|_{L_t^1(\dot B^{\frac d2}_{2,1})}+\|\delta a\|_{\widetilde L_t^\infty(\dot B^{\frac d2-1}_{2,1})}
\| u\|_{L_t^1(\dot B^{\frac d2+1}_{2,1})}+\varepsilon
T\|\varepsilon a^\varepsilon\|_{\widetilde L_t^\infty
(\dot B^{\frac d2-1}_{2,1})}\nonumber\\
\leq&\, C_{T} \big[\varepsilon+(\delta_0+\delta_1)\delta\mathcal{X}(t)\big],
\end{align}
for any $0\leq t\leq T$, where here and in what follows \(C_T>0\) denotes a constant depending on \(T\). 

Moreover, applying Lemma \ref{Lemma2.5} to \eqref{G5.1}$_1$, we have
\begin{align}\label{G5.3}
 \|\delta a\|_{\widetilde L_t^\infty(\dot B^{\frac{d}{2}-1}_{2,1})}
\lesssim&\,
\|a_0^\varepsilon-a_0\|_{\dot B^{\frac{d}{2}-1}_{2,1}}
+
\|\delta u\cdot\nabla a^{\varepsilon}\|_{L_t^1(\dot B^{\frac{d}{2}-1}_{2,1})}+
\|(1+a^\varepsilon){\rm div}\,\delta u\|_{L_t^1
(\dot B^{\frac{d}{2}-1}_{2,1})}
\nonumber\\
&+
\|\delta a\,{\rm div}\, u\|_{L_t^1(\dot B^{\frac{d}{2}-1}_{2,1})}+T^{\frac{1}{2}}
\| u\|_{L_t^2(\dot B^{\frac{d}{2}}_{2,1})}
\|\delta a\|_{\widetilde L_t^\infty(\dot B^{\frac{d}{2}-1}_{2,1})}\nonumber\\
\lesssim&\, \|a_0^\varepsilon-a_0\|_{\dot B^{\frac{d}{2}-1}_{2,1}}+  {T}
\|\delta u\|_{L_t^1(\dot B_{2,1}^{\frac{d}{2}})} \|a^{\varepsilon}\|_{\widetilde L_t^\infty(\dot B_{2,1}^{\frac{d}{2}})}+\|\delta a\|_{\widetilde L_t^\infty(\dot B_{2,1}^{\frac{d}{2}-1})} \|u\|_{L_t^1(\dot B_{2,1}^{\frac{d}{2}+1})}\nonumber\\
&+ \bigg( 1+\|a^{\varepsilon}\|_{\widetilde L_t^\infty(\dot B_{2,1}^{\frac{d}{2}})}\bigg)\|\delta u\|_{L_t^1(\dot B_{2,1}^{\frac{d}{2}})}\nonumber\\
\leq&\,C_{T} \big[\varepsilon+(\delta_0+\delta_1)\delta\mathcal{X}(t)\big],
\end{align}
for any $0\leq t\leq T$.
Combining \eqref{G5.2} with \eqref{G5.3}, we  derive that
\begin{align}\label{G5.4}
\delta\mathcal{X}(t)\leq C_{T} \big[\varepsilon+(\delta_0+\delta_1)\delta\mathcal{X}(t)\big], 
\end{align}
which together with the smallness of $\delta_0$ and $\delta_1$ gives
\begin{align*}
\delta\mathcal{X}(t)\leq C_{T} \varepsilon. 
\end{align*}
Therefore, we complete the proof of Theorem \ref{Th3}.
 \hfill $\Box$

\bigskip 
\noindent{\bf Acknowledgements:} 
Zhang  is supported by NSFC (Grant Nos. 12471215 and 12331007) and Taishan Scholars Program (tsqn202507101).

\vspace{2mm}

\noindent\textbf{Conflict of interest.} The authors do not have any possible conflicts of interest.

\vspace{2mm}

\noindent\textbf{Data availability statement.}
 Data sharing is not applicable to this article as no data sets were generated or analyzed during the current study.

\bibliographystyle{plain}

\end{document}